\documentclass[11pt]{amsart}
\usepackage{amssymb}
\usepackage[margin=3cm, marginparwidth=2.4cm]{geometry}

\usepackage{xcolor, overpic}
\usepackage[colorlinks=true, pdfstartview=FitV, linkcolor=blue, citecolor=blue, urlcolor=blue]{hyperref}

\theoremstyle{plain}
\newtheorem{thm}{Theorem}[section]
\newtheorem{lemma}[thm]{Lemma}
\newtheorem{prop}[thm]{Proposition}
\newtheorem{cor}[thm]{Corollary}

\newtheorem{introtheorem}{Theorem}

\numberwithin{equation}{section}
\numberwithin{figure}{section}

\theoremstyle{definition}
\newtheorem{definition}[thm]{Definition}
\newtheorem{example}[thm]{Example}
\newtheorem{remark}[thm]{Remark}

\newtheorem{obs}[thm]{Observation}
\newtheorem{claim}[thm]{Claim}
\newtheorem{qn}[thm]{Question}

\usepackage{marginnote}
\newcounter{mynote}% a new counter for use in margin notes

\newcommand{\G}{\Gamma} 
\newcommand{\D}{\Delta}
\renewcommand{\L}{\Lambda}
\newcommand{\gequal}{\bumpeq}
\renewcommand{\S}{\Sigma}

\title[RAAG subgroups of RACGs and RAAGs]{Right-angled Artin subgroups of right-angled Coxeter and Artin groups}
\author{Pallavi Dani}
\author{Ivan Levcovitz}
\date{}
\thanks{This work of the first author was supported
by a grant from the Simons Foundation (\#426932, Pallavi Dani) and by NSF Grant No.~DMS-1812061. This work of the second author was supported by the Israel Science Foundation and in part by a Technion fellowship.}
\begin{document}

\begin{abstract}
We determine when certain natural classes of subgroups of right-angled Coxeter groups (RACGs) and right-angled Artin groups (RAAGs) are themselves RAAGs.   
We characterize finite-index \emph{visual RAAG} subgroups of $2$-dimensional RACGs.
As an application, we
show that any $2$-dimensional, one-ended RACG with planar defining graph is quasi-isometric to a RAAG if and only if it is commensurable to a RAAG.
Additionally, we give new examples of RACGs with non-planar defining graphs which are commensurable to RAAGs.

Finally,
we 
give a new proof of a result of Dyer:
every subgroup 
generated by conjugates of  RAAG generators is itself a RAAG.
\end{abstract}
\maketitle

\section{Introduction}
Let $\Gamma$ be a finite simplicial graph with vertex set $V(\G)$ and edge set $E(\G)$. The \textit{right-angled Artin group}
(RAAG for short) 
 associated to $\G$ is the group $A_\G$ given by the presentation:
\[A_\G = \langle V(\G) ~ |~ st = ts \text{ for all } (s,t) \in E(\Gamma) \rangle \]

This article is concerned with the following question.  Given a finite set $S$ of elements in a group, when is the group generated by $S$ isomorphic to a RAAG in the ``obvious'' way (i.e.~with $S$ as the ``standard'' RAAG generating set)?  To make this precise, we define the notion of \emph{RAAG system}.
\begin{definition}[RAAG system]\label{def:RAAG_system}
Let $G$ be any group with generating set $S$. 
Let $\Delta$ be the graph whose vertex set is in bijection with $S$ and which has an edge between distinct $s, t \in S \equiv V(\Delta)$ if and only if $s$ and $t$ commute. 
We call $\Delta$ the \textit{commuting graph} associated to $S$.
There is a 
canonical 
 homomorphism $\phi: A_{\Delta} \to G$ extending the bijection $V(\Delta) \to S$. We say that $(G, S)$ is a \textit{RAAG system} if $\phi$ is an isomorphism. In particular, $(A_\G, V(\G))$ is a RAAG system for any RAAG $A_\G$.
\end{definition}

The  \textit{right-angled Coxeter group} (RACG for short) 
 associated to the finite simplicial graph $\G$ is the group $W_\G$ given by the presentation:
\[ W_\G = \langle V(\G) ~ |~ s^2 = 1 \text{ for all } s \in V(\G), st = ts \text{ for all } (s,t) \in E(\Gamma) \rangle \] 

In this article we study subgroups $G$ 
generated 
 by particular 
natural 
 subsets
  $S$ 
  of right-angled Coxeter and Artin groups, and we 
 give  characterizations for 
  when $(G, S)$ is a RAAG system
or  a finite-index RAAG system. 

A theorem of Davis--Januszkiewicz states that every RAAG is 
commensurable to 
 some RACG \cite{DJ}.  
This leads to the following question  addressing the converse:
\begin{qn} \label{ques:commensurable}
Which RACGs are commensurable to RAAGs?
\end{qn}

 A RACG that is commensurable to a RAAG is, in particular, quasi-isometric to a RAAG.
By considering different quasi-isometry invariants, one sees that the converse to 
the Davis--Januszkiewicz theorem above
is far from being true. For instance, there are many RACGs that are one-ended hyperbolic (such as virtual hyperbolic surface groups), while no RAAG is both one-ended and hyperbolic. Furthermore, RAAGs  have linear, quadratic or infinite divergence \cite{Behrstock-Charney}, whereas the divergence of a 
RACG can be a polynomial of any degree \cite{Dani-Thomas}. 
 Restricting to RACGs of at most quadratic divergence is still not enough to guarantee they are
quasi-isometric to RAAGs.
For instance, the Morse boundary of a RAAG with quadratic divergence is always totally disconnected (see~\cite{charney-sultan, cordes-hume}), 
while the Morse boundary of a RACG of quadratic divergence can have nontrivial connected components \cite{Behrstock}. 
The above examples show that there are numerous families of RACGS which are not quasi-isometric and, hence, not commensurable to any RAAG.
Within the subclass of one-ended RACGs with planar, triangle-free defining graphs, Nguyen--Tran 
characterize those quasi-isometric to RAAGs~\cite{nguyen-tran}. 
Theorem~\ref{intro_thm_planar} below answers Question~\ref{ques:commensurable} in this setting. 

We note that every RACG (indeed, every Coxeter group) is virtually special, and therefore has a finite-index subgroup which is a subgroup of a RAAG~\cite{haglund-wise}.  However, this subgroup is not of finite index in the RAAG, which would be required for establishing commensurability. 

One approach to proving that a RACG is commensurable to a RAAG is to look for finite index subgroups that are isomorphic to RAAGs. 
We focus on a 
 class of subgroups of RACGs, 
introduced by LaForge in his PhD thesis~\cite{laforge}, 
 that are logical candidates for being RAAGs.
Given a RACG defined  by a graph $\G$ and two non-adjacent vertices $s, t \in V(\G)$, it follows that $st$ is an infinite order element of $W_\G$. There is then a correspondence between edges of the complement graph $\G^c$ with such infinite order elements of $\G$. Given a subgraph $\L$ of $\G^c$,
%which without loss of generality we assume contains no isolated vertices,
let $G$ be the subgroup generated by $E(\L)$ (thought of as infinite order elements of $W_\G$). 
As $G$ is generated by the edges of $\L$, we may as well assume that $\L$ has no isolated vertices. 
A natural question is: 

\begin{qn}
When is 
$(G, E(\L))$ a finite-index RAAG system? 
\end{qn}
If $(G, E(\L))$ is indeed a RAAG system, then $G$ is called a \textit{visual RAAG subgroup} of $W_\G$.  
LaForge obtained some necessary conditions for such subgroups to be visual RAAGs. 

We say that $W_\G$ is \textit{2-dimensional}
 if $\G$ is triangle-free.
Our first main theorem gives an exact characterization of the finite-index visual RAAG subgroups of $2$-dimensional RACGs
in terms of 
 graph theoretic conditions:

\begin{introtheorem} \label{intro_thm:finite-index}
	Let $W_\Gamma$ be a
	$2$-dimensional RACG.
	Let $\L$ be a 
	subgraph of $\G^c$ with no isolated vertices, and let $G$ be the subgroup generated by $E(\L)$. Then the following are equivalent.
	\begin{enumerate}
		\item 
		$(G, E(\L))$ is a RAAG system and $G$ is finite index in $W_\G$.
		\item 
		 $(G, E(\L))$ is a RAAG system and $G$ has index either two or four in $W_\G$ (and exactly four if $W_\G$ is not virtually free).
		
		\item  $\L$ has at most 
		two 
		components and satisfies conditions $\mathcal R_1$--$\mathcal R_4$, $\mathcal F_1$ and $\mathcal F_2$. 
	\end{enumerate}
\end{introtheorem}
The conditions $\mathcal R_1$--$\mathcal R_4$, $\mathcal F_1$ and $\mathcal F_2$ in the above theorem are algorithmically checkable graph theoretic conditions on $\G$ and $\L$. 
See Section~\ref{sec:visual} for precise definitions of these conditions. 

In Section~\ref{sec:applicaitons} we provide several applications to concrete families of RACGs. In particular we prove:
\begin{introtheorem} \label{intro_thm_planar}
	Let $W_\G$ be a $2$-dimensional, one-ended RACG with planar defining graph.
	Then $W_\G$ is quasi-isometric to a RAAG if and only if it 
	contains an index $4$ subgroup isomorphic to a RAAG.
\end{introtheorem}
A complete description of which RACGs considered in Theorem~\ref{intro_thm_planar}
are quasi-isometric to RAAGs is given by Nguyen-Tran \cite[Theorem 1.2]{nguyen-tran}.
Theorem~\ref{intro_thm_planar} shows these are actually commensurable to RAAGs.

We also give two families of RACGs defined by non-planar graphs 
which contain finite-index RAAG subgroups (see Corollaries~\ref{cor:example} and~\ref{cor:non_planar_example}).  These cannot be obtained by applying the Davis--Januszkiewicz constructions to the defining graphs of the RAAGs they are commensurable to. 
For the family in Corollary~\ref{cor:example}, 
we use work of Bestvina--Kleiner--Sageev on RAAGs \cite{BKS}, to conclude the RACGs are quasi-isometrically distinct.
We believe that the methods from this article may be used to further study commensurability
of RACGs.

The proof of Theorem~\ref{intro_thm:finite-index} consists of two main parts. 
One part involves obtaining
an understanding of when $G$ is of finite index, leading to  conditions $\mathcal F_1$ and $\mathcal F_2$.  
 To obtain these, we use \emph{completions of subgroups},
introduced
in \cite{DL}.
 The other aspect consists of obtaining criteria
to recognize when $(G, E(\L))$ is a RAAG system. To do so, we prove the following 
theorem by careful analysis of disk diagrams:
\begin{introtheorem} \label{intro_thm:RAAG-system}
Let $W_{\G}$ be a RACG. 
Let $\L$ be a 
	subgraph of $\G^c$ with no isolated vertices
	and at most two components. 
	Then the subgroup 
	 $(G, E(\L) ) < W_\G$ 
	 is a RAAG system if and only if $\mathcal R_1$--$\mathcal R_4$ are satisfied.
\end{introtheorem}

Conditions $\mathcal R_1$, $\mathcal R_2$, and a 
	condition more or less equivalent to 
 $\mathcal R_3$ were known to be necessary conditions for $(G, E(\L))$ to be a RAAG system by work of LaForge \cite{laforge}.  
 We show in Example~\ref{ex:R4_necessary} that they are not sufficient. We
 introduce a fourth graph-theoretic condition $\mathcal R_4$
to obtain a complete characterization of all visual RAAG subgroups defined by subgraphs of 
$\G^c$ with at most two components.  The bulk of the proof of Theorem~\ref{intro_thm:RAAG-system} consists of showing that the conditions $\mathcal R_1$--$\mathcal R_4$ are sufficient.

Note that, unlike in Theorem~\ref{intro_thm:finite-index}, 
 there is no assumption on the dimension of the RACGs 
in Theorem~\ref{intro_thm:RAAG-system}. On the other hand, there is an additional assumption in Theorem~\ref{intro_thm:RAAG-system}, namely that the subgraph $\L$ of $\G$ can have at most two components.

When $\L$ contains more than two components, the situation becomes much more complex. We show that additional graph-theoretic conditions are necessary to generalize the Theorem~\ref{intro_thm:RAAG-system}  to this setting (see Lemma~\ref{lem:R5} 
 and Lemma~\ref{lem:triangle}). 
Remarkably, a consequence of these conditions is that if $\G$ is triangle-free and $(G, E(\L) )$ is a finite-index RAAG system, then $\L$ can have at most two components. This fact is crucial to the proof of Theorem~\ref{intro_thm:finite-index}, which does not have any assumption on the number of components of $\L$.
Additionally, 
we 
 are aware that even more conditions are necessary than those in this article, but 
 we 
 do not have a complete conjectural list of conditions that would be sufficient to characterize visual RAAGs.

We next turn our attention to RAAG subgroups of RAAGs. 
A classical theorem on Coxeter groups, proven independently by Deodhar \cite{Deodhar} and Dyer \cite{Dyer}, states that reflection subgroups of Coxeter groups (i.e., those generated by conjugates of generators) are themselves Coxeter groups. 
In fact, Dyer proves an analogous result for the class of groups defined by  \emph{reflection systems} (see~\cite{Dyer} for the definition), which includes Coxeter groups as well as RAAGs.  Specifically, he shows that 
 subgroups generated by conjugates of standard generators are themselves in this class.  As RAAGs are the only torsion-free groups in this class, one obtains the following result.  Here, we
define a \textit{generalized RAAG reflection} to be an element of a RAAG $A_\D$ that is conjugate to a generator in $V(\D)$.
\begin{introtheorem}[\cite{Dyer}] \label{intro_thm:generalized_reflections}
	Let $\mathcal{T}$ be a finite set of generalized RAAG reflections in the RAAG $A_\G$. Then the subgroup $G < A_\G$ generated by $\mathcal{T}$ is a RAAG.
\end{introtheorem}

We thank Luis Paris for informing us that this result is contained in~\cite{Dyer}, and the explanation in the preceding paragraph.  We include our proof of Theorem~\ref{intro_thm:generalized_reflections}, as our geometric approach is very different from  from that of Dyer, which is algebraic and uses cocycles.  
Our proof uses 
a characterization of RAAG systems in terms of  the deletion 
condition, given by Basarab \cite{basarab}. 
We use disk diagrams to show that subgroups generated by generalized RAAG reflections satisfy  the criteria in Basarab's characterization. 

We note that, although $G$ (from Theorem~\ref{intro_thm:generalized_reflections}) is a RAAG, $(G, \mathcal{T})$ is not necessarily a RAAG \textit{system} and in general $G$ is not isomorphic to the RAAG $A_\Delta$ where $\Delta$ is the commuting graph corresponding to $\mathcal{T}$.
Kim-Koberda show that there exists a subgroup of $G$ (generated by sufficiently high powers of the elements of $\mathcal T$) which is isomorphic to $A_\Delta$ \cite{Kim-Koberda}.

 Genevois, as well as an anonymous referee, pointed out to us that a proof of Theorem~\ref{intro_thm:generalized_reflections} may be possible using~\cite[Theorem 10.54]{Genevois1} (see also~\cite[Theorem~3.24]{Genevois2}).

\subsection*{Acknowledgements}
	The authors would like to thank Jingyin Huang for suggesting the question that led to our proof of Theorem~\ref{intro_thm:generalized_reflections},
 Luis Paris for informing us that Theorem~\ref{intro_thm:generalized_reflections} is a result of Dyer,
	 Kevin Schreve for a comment that led to 
	Corollary~\ref{cor_bipartite}, and Hung Tran for encouraging us to look at the examples considered in Theorem~\ref{intro_thm_planar}.  Finally, we would like to thank Jason Behrstock,
	Anthony Genevois,  Garret LaForge, Kim Ruane and the referees  for helpful comments and conversations. 
\section{Background}

\subsection{Basic terminology and notation}
	Let $G$ be a group with generating set $S$. We say that $w = s_1 \dots s_n$, with $s_i \in \big( S \cup S^{-1} \big)$ for $1 \le i \le n$, is a \textit{word over $S$} or a \textit{word in $G$}. 
	If the words $w$ and $w'$ represent the same element of $G$, then we say that $w'$ is an \textit{expression} for $w$ and write \textit{$w' \gequal w$}.
	We say the word $w = s_1 \dots s_n$ is \textit{reduced} (or \textit{reduced over $S$} for emphasis) if given $w' = t_1 \dots t_m \gequal w$, it follows that $n \le m$.

\subsection{Right-angled Coxeter and Artin groups} \label{subsec_racgs_and_raags}
Coxeter groups can be characterized as those groups which are generated by involutions and which satisfy the deletion condition, see Definition~\ref{def:deletion} below (for a proof of this fact, see \cite[Theorem 3.3.4]{Davis}). By  work of Basarab \cite{basarab}, RAAGs can be characterized in a similar manner (see Theorem~\ref{thm_raag_char} below). This characterization will be utilized in Section~\ref{sec:reflections}.

\begin{definition}[Deletion Condition]\label{def:deletion}
	Let $G$ be a group generated by $S$. We say that $(G, S)$ satisfies the \textit{deletion condition} if, given any word $w$	over $S$, either $w$ is reduced or $w = s_1 \dots s_k$ and there exist $1 \le i < j \le k$ such that $s_1 \dots \hat{s}_i \dots \hat{s}_j \dots s_k$ is an expression for $w$.
\end{definition}

The result below directly follows from a result of Basarab.
\begin{thm}[\cite{basarab}] \label{thm_raag_char}
	Let $G$ be a group generated by $S$ such that $S \cap S^{-1} = \emptyset$
	 and $1 \notin S$. Then $(G, S)$ is a RAAG system if and only if each of the following holds:
	\begin{enumerate}
		\item Every $s$ in $S$ has infinite order. 
		\item $(G, S)$ satisfies the deletion condition.
	\end{enumerate}
\end{thm}
\begin{proof}
	If $(G, S)$ is a RAAG system,  then $G$ is torsion-free (see~\cite{charney}), so (1) holds.   Furthermore, $(G, S)$ satisfies (2) by \cite[Corollary 1.4.2]{basarab} (see also \cite[pg.~31, ex.~17]{bahls} for a simpler proof in this setting). The converse also follows from a direct application of \cite[Corollary 1.4.2]{basarab}.
\end{proof}

We now define certain moves which can be performed on a word that produce another expression for it. These moves provide a solution to the word problem for RAAGs and RACGs (see Theorem~\ref{thm_tits_solution} below).

\begin{definition}[Tits moves]
	Let $G$ be a group generated by $S$. Let $w = s_1 \dots s_n$ be a word over $S$. If $s_i$ and $s_{i+1}$ commute for some $1 \le i < n$, then the word $s_1 \dots s_{i-1}s_{i+1}s_i s_{i+2} \dots s_n$ is an expression for $w$ obtained by a \textit{swap operation} performed to $w$, which \textit{swaps} $s_i$ and $s_{i+1}$. If $s_i = s_{i+1}^{-1}$ for some $1 \le i < n$, then  $s_1 \dots s_{i-1}s_{i+2} \dots s_n$ is an expression for $w$ is obtained by a \textit{deletion operation} performed to $w$. A \textit{Tits move} is either a swap operation or a deletion operation. We say a word is \textit{Tits reduced} if no sequence of Tits moves can be performed to the word to obtain an expression with fewer generators.
\end{definition}

Theorem~\ref{thm_tits_solution} below shows that RAAGs and RACGs admit a nice solution to the word problem. This solution to the word problem for RACGs is a well known result of Tits \cite{Tits}, a version of which holds more generally for all Coxeter groups. The result below in the setting of RAAGs follows from 
a  theorem of Basarab \cite[Theorem 1.4.1]{basarab} which generalizes 
 Tits' result (see also \cite[Theorem 3.9]{Green}).

\begin{thm}[\cite{Tits, basarab}] \label{thm_tits_solution}
	Let $A_\G$ be either a RAAG or a RACG. Then the following hold:
	\begin{enumerate}
		\item If $w_1$ and $w_2$ are reduced words over $V(\G)$ representing the same element of $G$, then $w_2$ can be obtained from $w_1$ by Tits swap 
		moves.
		\item Given any word $w$ over $V(\G)$, a reduced expression for $w$ can be obtained by applying Tits moves to $w$.
	\end{enumerate}
\end{thm}
We will often not refer directly to the above theorem, and we will instead simply say that a given RAAG or RACG  \textit{admits a Tits solution to the word problem.}

The next two lemmas are well known and will often be implicitly assumed.
\begin{lemma}
	Let $A_\G$ either be a RAAG or RACG. Then $s, t \in V(\G)$ commute as elements of $A_\G$ if and only if $(s,t)$ is an edge of $\G$. 
\end{lemma}
\begin{proof}
	One direction of the claim follows from 
 the definitions of a RAAG and a RACG.
	 If $A_\G$ is a RACG, then the other direction follows from \cite[Prop 4.1.2]{BB}.
	
Now suppose that $A_\G$ is a RAAG, and let $s, t \in V(\G)$ 
	be non-adjacent vertices. Suppose, for a contradiction that $w = sts^{-1}t^{-1} \gequal 1$. Let $D$ be a disk diagram with boundary $w$ (see Section~\ref{subsec:disk} for a reference for disk diagrams). This disk diagram contains exactly two intersecting hyperplanes: one labeled by $s$ and one labeled by $t$. However, this is a contradiction as a pair of hyperplanes whose labels are non-adjacent vertices of $\G$ cannot intersect.
\end{proof}

\begin{lemma} \label{lem:commuting}
	Let $W_\G$ be a RACG, and let
	 $s, t, q, r \in V(\Gamma)$ be such that $s$ and $t$ do not commute, and $r$ and $q$ do not commute. We have that $(st)(qr) \gequal (qr)(st)$ 
	 if and only if one of the following holds:
	\begin{enumerate}
		\item There is a square in $\Gamma$ formed by $s, q, t, r$.
		\item $t = q$ and $s = r$.
		\item $t = r$ and $s = q$.
	\end{enumerate}
\end{lemma}
\begin{proof}
	Clearly each of (1), (2)  and (3) implies that $(st)(qr) \gequal (qr)(st)$. 
	
	To prove the converse, suppose that $(st)(qr) \gequal (qr)(st)$. 
	Suppose first that $t = q$,  
	and consequently $stqr \gequal sr$. 
	As $s$ and $t$ do not commute and $q$ and $r$ do not commute, this is only possible if $r = t$. Thus, (2) holds.
	
	If $s=q$, as $qrts \gequal tsqr$, we  apply the same argument to conclude that $t = r$, showing (3) holds. By similar arguments, if $s = r$ then $t = q$, and if $t = r$ then $s = q$. Thus, we may assume that $s$, $t$, $q$ and $r$ are all distinct vertices of $\G$. In this case we again conclude by Tits' solution to the word problem, that as $stqr \gequal qrst$ then $s$, $q$, $t$ and $r$ form a square in~$\G$.
\end{proof}

\subsection{Disk Diagrams}\label{subsec:disk}
	
We give a brief background on disk diagrams as they are used in our setting, and we refer the reader to \cite{Sageev} and \cite{wise-qc-hierarchy} 
for the general theory of disk diagrams over cube complexes. We then give some preliminary lemmas that are needed in later sections.

Let $A_\D$ be a RAAG, and let $w = s_1 \dots s_n$, with $s_i \in V(\D)$, be a word equal in $A_\Delta$ to the identity, i.e.~$w  \gequal 1$. There exists a van Kampen diagram $D$ with boundary label 
$w$, and we call this planar $2$-complex a \textit{disk diagram in $A_\D$ with boundary label $w$}. We now describe some additional properties of $D$ in our setting. The edges of $D$ are oriented and labeled by generators in $V(\D)$. A \textit{path in $D$} is a path $\gamma$ in the $1$-skeleton of $D$, traversing edges $e_1, \dots, e_m$, and the label of $\gamma$ is the word $a_1 \dots a_m$ where, for each $1 \le i \le m$, $a_i$ is the label of $e_i$ if $e_i$ is traversed along its orientation, and $a_i^{-1}$ is the label of $e_i$ if $e_i$ is traversed opposite to its orientation. 
Every cell in $D$ is a square that has a boundary path with label $aba^{-1}b^{-1}$ for some commuting generators $a$ and $b$ in
 $V(\D) \cup V(\D)^{-1}$.

There is a base vertex $p\in \partial D$ and an orientation on $D$, 
such that the smallest closed path $\delta$ which traverses the boundary of $D$  in the clockwise orientation starting at $p$ and 
 traversing 
 every edge outside the interior of $D$ has label $w$. We call $\delta$ the \textit{boundary path} of $D$. Note that if $D$ contains an edge $e$ not contained in a square, then necessarily $\delta$ traverses $e$ exactly twice.
	
If $W_\G$ is a RACG and $w$ is a word over $V(\G)$ equal in $W_\G$ to the identity, then we define a disk diagram $D$ in $W_\G$ with boundary $w$ similarly. However, as each generator in $V(\G)$ is an involution, we do not need to orient the edges of $D$.

Let $D$ be a disk diagram and $q = [0,1] \times [0,1]$ be a square in $D$. The subset $\{\frac{1}{2} \} \times [0, 1] \subset q$ (similarly, $[0,1] \times \{\frac{1}{2} \} \subset q$) is a \textit{midcube}. 
The midpoint of an edge in $D$ is also defined to be a \textit{midcube}.
 A \textit{hyperplane} in $D$ is a minimal non-empty collection $H$ of midcubes in $D$ 
with the property that given any midcube $m \in H$ and a midcube $m'$ in $D$ such that $m \cap m'$ is contained in an edge of $D$, it follows that $m' \in H$.  We say that $H$ is dual to an edge $e$ if the midpoint of $e$ is in $H$. 

Since opposite edges in every square in $D$ have the same label, it follows that every edge intersecting a fixed hyperplane $H$ has the same label.  We call this the \emph{label of the 
 hyperplane}. Since adjacent sides of a square have distinct labels which commute, 
 it follows that no hyperplane self-intersects, and if two hyperplanes intersect, then their labels correspond to distinct, commuting generators.
 	(See Figure~\ref{fig:disk_diagram-ex} for an example of a disk diagram and some of its hyperplanes.)

\begin{figure}[h!]
	\medskip\begin{overpic}[scale=0.2]
		{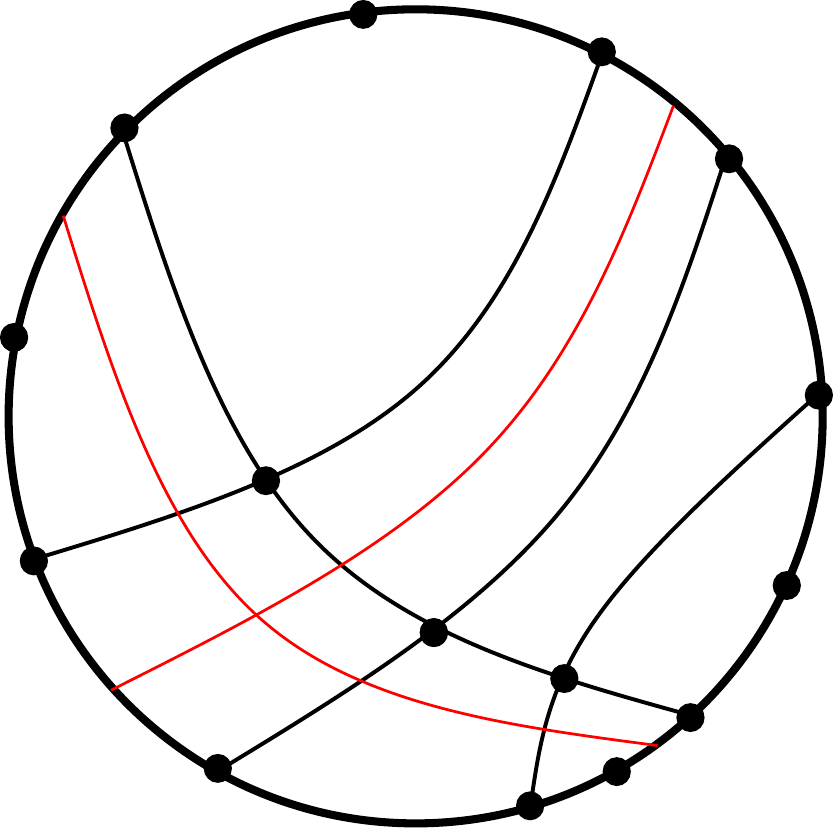}
		\put(42, -7){\tiny $s_1$}
		\put(4, 14){\tiny $s_2$}
		\put(-9, 43){\tiny $s_3$}
		\put(-4, 73){\tiny $s_4$}
		\put(20, 96){\tiny $s_5$}
		\put(50, 100){\tiny $s_3$}
		\put(77, 91){\tiny $s_2$}
		\put(96, 67){\tiny $s_1$}
		\put(100, 40){\tiny $s_6$}
		\put(92, 19){\tiny $s_5$}
		\put(78, 5){\tiny $s_4$}
		\put(68, 0){\tiny $s_6$}
		\put(24, 0){\tiny $p$}

	\end{overpic}
	%\medskip
	\caption{
		 A disk diagram in a RACG with boundary the word $s_2 s_3 s_4 s_5 s_3 s_2 s_1 s_6 s_5 s_4 s_6 s_1$ and base vertex $p$. Two hyperplanes are shown in red. As these hyperplanes intersect, it must be that $s_4$ commutes with $s_2$.
	}
	\label{fig:disk_diagram-ex}
\end{figure}

\begin{definition}[Maps preserving boundary combinatorics] \label{def_preserve_bdry_comb}
Let $D$ and $D'$ be disk diagrams, and let $\delta$ and $\delta'$ 
respectively be their boundary paths. Let $E = \{e_1, \dots, e_m\}$ (resp. $E' = \{e_1', \dots, e_n'\}$) be the edges traversed by $\delta$ (resp. $\delta'$).  More precisely, $e_i$ (resp. $e_i'$) is the $i$th edge traversed by $\delta$ (resp. $\delta'$) for each $i$.
Observe that every hyperplane of $D$ is dual to two edges $e_j, e_k \subset E$ for some $j \neq k$.
(It could be that $e_j = e_k$ thought of as edges of $D$.) A similar statement holds for $D'$.

Let $F \subset E$ and $F' \subset E'$, 
and let $\psi: F \to F'$ be a bijection.  We say that $\psi$ \emph{preserves boundary combinatorics} if for every pair   of edges $e, f \in F$ which are dual to the same hyperplane of $D$, their images $\psi(e)$ and $\psi(f)$ are dual to the same hyperplane of $D'$. 

Note that if $\Psi$ preserves boundary combinatorics, then $\Psi^{-1}$ does as well.
\end{definition}

A pair of hyperplanes $H$ and $H'$ in a disk diagram $D$ form a \textit{bigon} if they intersect in at least two distinct points. 
The following lemma, first proven in \cite[Theorem 4.3]{Sageev},  guarantees that  we can always choose a disk diagram without 
 bigons. The boundary combinatorics statement below is guaranteed by the proof of this fact in \cite[Lemma 2.3, Corollary 2.4]{wise-qc-hierarchy}.

\begin{lemma}[\cite{Sageev} \cite{wise-qc-hierarchy}] \label{lem_remove_pathologies}
	Given a disk diagram $D$ with boundary label $w$, there exists a disk diagram $D'$ also with boundary label $w$ such that $D'$ does not contain any bigons.
	 Moreover, the natural bijection between the 
 edges traversed by the boundary paths of $D$ and $D'$  induced by the label $w$ preserves boundary combinatorics.	
\end{lemma}

\begin{remark} \label{rem:no-bigons}
	In light of Lemma~\ref{lem_remove_pathologies}, for the rest of this paper we will always assume that any disk diagrams we consider do not have bigons.
\end{remark}

\begin{remark}
	Let $\alpha$ be a path with label $s_1 \dots s_n$ in some disk diagram. The ``edge of $\alpha$ with label $s_i$'' is understood to be the $i$th edge $\alpha$ traverses (even though there may be several edges of $\alpha$ with the same label as this edge.
	 A similar statement holds when we refer to subpaths of $\alpha$.
\end{remark}

Given a disk diagram  with boundary label $w$, we will often want to 
produce a new disk diagram with boundary label $w'$, where $w'$ is obtained from $w$ by a Tits move, and such that boundary combinatorics are preserved on appropriate subsets of the boundary paths. The following lemma exactly describes how we can perform these operations.

\begin{lemma} \label{lem:new_diagrams}
	Let $D$ be a disk diagram over the group $W$, where $W$ is either a RACG or a RAAG.  Suppose the boundary path of $D$ traverses the edges $e_1, \dots e_n$ and has label $w = s_1 \dots s_n$.  
	\begin{enumerate}
		\item If $s_r$ and $s_{r+1}$ (taken modulo $n$) are distinct and commute for some $1 \le r \le n$, then there is a disk diagram $D'$ whose boundary path traverses the edges $e_1', \dots, e_n'$ and has label $s_1 \dots s_{i+1}s_i\dots s_n$.  Furthermore, the map $\psi$ preserves boundary combinatorics, where $\psi$ is defined by $\psi(e_r) = e_{r+1}'$, $\phi(e_{r+1}) = e_{r}'$, and $\psi(e_j)= e_j'$ for $j \neq r, r+1$. 
		\item  If 
		$s_r = s_{r+1}^{-1}$ 
		(taken modulo $n$) for some $1 \le r \le n$, then there is a disk diagram $D'$ with boundary label $s_1 \dots s_{r-1}s_{r+2}\dots s_n$. Moreover, the natural map from edges traversed by the boundary path of $D'$ to edges traversed by the boundary path of $D$ preserves boundary combinatorics. 
		\item Given any generator (or inverse of a generator) $s$ and any $r$, with $1 \le r \le n$, it follows that there exists a disk diagram $D'$ with boundary label $s_1 \dots s_r (s s^{-1}) s_{r+1} \dots s_n$. Moreover, the natural map from edges traversed by the boundary path of $D$ to the edges traversed by the boundary path of $D'$ preserves boundary combinatorics.
	\end{enumerate} 
\end{lemma}

\begin{proof}
	We first prove (1).	Let $q$ be a square whose edges are labeled consecutively by $s_r$, $s_{r+1}$, $s_r^{-1}$, $s_{r+1}^{-1}$. We form the disk diagram $D'$ by identifying consecutive edges of $q$ labeled by $s_r$ and $s_{r+1}$ to the edges of $\partial D$ labeled by $s_r$ and $s_{r+1}$ (these edges must be distinct as $s_r \neq s_{r+1}$). The claim is readily checked. 
	
	We next prove (2). 	Let $e$ and $f$ be the edges of $\partial D$ labeled respectively by $s_r$ and $s_{r+1}$. Suppose first that $e$ and $f$ are distinct. In this case, form the disk diagram $D'$ by identifying $e$ and $f$,
	i.e. ``fold'' these edges together. On the other hand, if $e = f$, then as $D$ has boundary label $w$, it must follow that $e$ is a spur, i.e. an edge attached to $D$ that is not contained in any square and which contains a vertex of valence $1$. In this case we can remove the edge $e$ from $D$ to obtain $D'$. In either case, the claim is readily checked.
	
	To show (3), form $D'$ by inserting a spur edge with label $s$ to the vertex traversed by the boundary path of $D$ between $s_r$ and $s_{r+1}$.
\end{proof}

\section{Visual RAAG subgroups of right-angled Coxeter groups}
\label{sec:visual}
 In this and the next section we study visual RAAG subgroups of RACGs, as described in the introduction.  We begin by describing some notation that will be used throughout these sections.

Let $\G$ be a 
graph, and let $W_\G$ be the corresponding RACG. 
Let $\G^c$ denote the complement of $\G$,
 that is, the graph with the same vertex set as $\G$, which has an edge between two (distinct) vertices if and only if the corresponding vertices are not adjacent in $\G$.
Let $\Lambda$ be a subgraph of $\G^c$ 
with no isolated vertices, i.e., one in which every vertex of $\L$ is contained in some edge.

We form a new graph $\Theta = \Theta(\G, \L)$ which we think of as a graph containing the edges of both $\G$ and $\L$. 
More formally, $V(\Theta) = V(\G)$ and $E(\Theta) = E(\G)\cup E(\L)$.
Note that as $E(\L) \subset \G^c$, it follows that $\Theta$ is simplicial. We refer to edges of $\Theta$ that correspond to edges of $\G$ (resp.~$\L$) as  $\G$-edges (resp.~$\L$-edges). 

A $\L$-edge between vertices $a$ and $b$ corresponds to an inverse pair of infinite order elements of $W_\G$, namely $ab$ and $ba$.  By a slight abuse of terminology, we 
 will use the term $\L$-edge to refer to one of these elements and vice versa.
 We identify $E(\L)$ with a subset of $W_\G$ by arbitrarily choosing one of the two infinite order elements corresponding to each $\L$-edge, and we define 
 $G^\Theta$ to be the subgroup of $W_\G$ generated by $E(\L)$. 
As we are dealing with subgroups generated by $E(\L)$, there is no loss in generality in assuming that $\L$ has no isolated vertices.
 The goal of this section is to study when $(G^\Theta, E(\L))$ is a RAAG system.

Let $\Delta$ be the commuting graph corresponding to $E(\L$) (as defined in the introduction),
and let $A_\D$ be the corresponding RAAG. 
Recall that, by definition, $(G^\Theta, E(\L))$ is a RAAG system if and only if the natural homomorphism $\phi: A_\D \to G^\Theta$ extending the bijection between $V(\D)$ and $E(\L)$ is an isomorphism. 
As $\phi$ is always surjective, we would like to understand when $\phi$ is injective.

For the remainder of this section, we fix $\Gamma, \L, \Theta, A_\Delta$, and $\phi$ as above.
Furthermore, we will use the following terminology.   The path $\gamma$ in $\Theta$  \emph{visiting vertices $x_1, x_2,   \dots, x_n$}  is defined to be the path which starts at $x_1$, passes through the remaining vertices in the order listed, and ends at $x_n$.  We say that $\gamma$ is simple if $x_i \neq x_j$ for $i \neq j$, and that $\gamma$ is a loop if $x_1 = x_n$.  Finally, $\gamma$ is a cycle if it is a loop with $n\ge 3$, such that $x_i \neq x_j$ unless $\{i, j\} = \{1, n\}$.
We call a path 
(resp.~cycle) in $\Theta$ consisting only of $\G$-edges a $\G$-path 
(resp.~$\G$-cycle). 
We define $\L$-paths and $\L$-cycles similarly.

We begin by describing some graph theoretic conditions 
on $\Theta$ 
 which are consequences of 
 either $G^\Theta$ being a RAAG or of $(G^\Theta, E(\L))$ being a RAAG system.
 
 Conditions $\mathcal R_1$ and $\mathcal R_2$, defined below, when combined, are equivalent to LaForge's star-cycle condition. LaForge proves that  $\mathcal R_1$ and $\mathcal R_2$ are necessary conditions for $(G^\Theta, E(\L))$ to be a RAAG system \cite[Lemma 8.2.1]{laforge}. We include proofs here for completeness.

\begin{definition}[Condition $\mathcal R_1$]\label{def:R1}
We say that  $\Theta$ satisfies \emph{condition $\mathcal R_1$} if it does not contain a $\L$-cycle.
\end{definition}

\begin{lemma}[\cite{laforge}] \label{lem:nocycle}
 If $(G^\Theta, E(\L))$ is a RAAG system, then $\Theta$ satisfies $\mathcal R_1$.
\end{lemma}
\begin{proof}
 Suppose $\Theta$ does not satisfy $\mathcal R_1$.  Then it contains a $\L$-cycle, say with vertices $a_1, \dots, a_k$, where $k\ge 3$, 
such that 
for each $i$ (mod k),  $a_i$ is connected to $a_{i+1}$ by a $\L$-edge. 
Let $g_i$
be the generator of $A_\Delta$ (or its inverse) corresponding to the (oriented) $\L$-edge $a_ia_{i+1}$. 
As the $a_i$'s are along a cycle, no $\L$-edge is repeated, and we have that    $g_i \neq g_j^{-1}$ for all $i \neq j$.
 This, together with the fact that RAAGs satisfy the deletion condition 
 (see Theorem~\ref{thm_raag_char}), 
 implies that 
$g=g_1g_2\dots g_{k}$ is a non-trivial element of $A_\Delta$.
 Moreover, 
$\phi(g)= (a_1a_2) (a_2a_3)\dots (a_ka_1) =1$, so $g$ is in the kernel of $\phi$, and therefore $\phi$ is not injective.
\end{proof}

\begin{definition}[Condition $\mathcal R_2$]\label{def:R2}
We say that  $\Theta$ satisfies \emph{condition $\mathcal R_2$} if each component of 
$\L \subset \Theta$ (with the natural inclusion)
is an induced subgraph of~$\Theta$.
\end{definition}

\begin{lemma} [\cite{laforge}] \label{lem:no-almost-cycle} 

If $G^\Theta$ is a RAAG, then $\Theta$ satisfies $\mathcal R_2$.
\end{lemma}
\begin{proof}
 Suppose $\Theta$ does not satisfy $\mathcal R_2$, and let $u$ and $v$ be a pair of vertices in 
a component of $\L$, such that $u$ and $v$ are adjacent in $\Theta$. It follows that $u$ and $v$ are connected by a $\G$-edge, and therefore, they commute.
Since $u$ and $v$ are in the same component of $\L$, there is a simple
 $\L$-path from $u$ to $v$ whose vertices (in order) are  $u=a_1, \dots, a_k=v$.
Note that $k \ge 3$, since  $\Theta$ is a simplicial graph. 
For $1 \le i\le k-1$, 
let 
 $g_i$ 
be the generator of $A_\Delta$ (or its inverse) corresponding to the $\L$-edge $a_ia_{i+1}$, and let $g = g_1g_2\dots g_{k-1}$. 
The
element $g$  
is a non-trivial element of $A_\D$, 
as RAAGs satisfy the deletion condition by Theorem~\ref{thm_raag_char}.

We now have that $\phi(g)^2 =  \big((a_1a_2) (a_2a_3)\dots (a_{k-1}a_k) \big)^2 =(a_1a_k)^2 =(uv)^2 =1$, 
since $u$ and $v$ commute. This implies that $G^\Theta$ has torsion. Thus, $G^\Theta$ cannot be a RAAG as RAAGs are torsion-free~(see \cite{charney}). 
\end{proof}

Our next condition, $\mathcal R_3$, is motivated by the following example. 
\begin{example}
Let $\Theta$ be the graph in Figure~\ref{fig:C3-ex}, where the $\G$ edges are black and the $\Lambda$ edges are colored.  Since $u$ and $v$ each commute with $x$ and $z$, the commutator 
$[uv, xz] $ represents the trivial element in $W_\G$.  Now observe that $ [uv, xz]  \gequal (uv) (xy)(yz) (vu)(zy)(yz)$, which is a product of $\L$-edges, and therefore represents an element $g$ of $G^\Theta$. Now we can see that $(G^\Theta, E(\L))$ is not a RAAG system: if it were, then it would be possible to show that $g$ is trivial in $G^\Theta$ using only swap and deletion moves involving RAAG generators.  However, since 
$y$ does not commute with $u$ and $v$, no such moves are possible (see Lemma~\ref{lem:commuting}).  On the other hand, if there had been $\G$ edges, from $y$ to both $u$ and $v$, then there would be no contradiction.

\begin{figure}[h!]
\medskip\begin{overpic}[scale=0.6]
{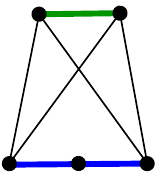}
\put(17, 99){\tiny $u$}
\put(71, 99){\tiny $v$}
\put(0, -7){\tiny $x$}
\put(42, -7){\tiny $y$}
\put(85, -7){\tiny $z$}
\end{overpic}
%\medskip
\caption{
} 
\label{fig:C3-ex}
\end{figure}
\end{example}

A $\Lambda$-edge word similar to the one in the above example can be constructed whenever $\G$ has a square whose vertices alternate between two components of $\Lambda$.   The example suggests that for such a $\Lambda$ to define a RAAG, the ``intermediate'' vertices in $\Lambda$  between the endpoints of the square  must all mutually commute.   This is made precise in 
the definition of $\mathcal R_3$ (Definition~\ref{def:R3}) and Lemma~\ref{lem:4cycle} below.  Before stating these, we introduce some terminology, which will be used throughout this section.

\begin{definition}\label{def:path-2-cpt} 
(2-component paths and cycles).
We say the $\G$-path $\gamma$ in $\Theta$ is a \emph{2-component path} if $\gamma$ visits vertices (in order) $c_1, d_1, c_2, d_2 \dots, c_{n}, d_{n}$ for some $n\ge 1$ (where $d_n$ could be 
omitted 
 if $n>1$) such that  the $c_i$'s all lie in a single component $\L_c$ of $\L$, and the $d_i$'s all lie in a single component $\L_d \neq \L_c$ of $\L$.   
If it is important to emphasize the components visited by $\gamma$, we will call it a $\L_c\L_d$-path.  

A \emph{2-component loop} 
is a 2-component path visiting $c_1, d_1, \dots, c_n, d_n, c_{n+1}$ such that $c_1=c_{n+1}$.  A \emph{2-component cycle} is a 2-component loop which is a 
 $\G$-cycle. A 2-component cycle of length four will be called a \emph{2-component square}.
\end{definition}

\begin{definition}\label{def:conv-hull} ($\L$-convex hull)
We define the \emph{$\L$-convex hull} of a set $X \subset V(\Theta)$ to be the convex hull of $X$ in $\L$. 
\end{definition}

\begin{definition}[Condition $\mathcal R_3$]\label{def:R3}
We say that  $\Theta$ satisfies \emph{condition $\mathcal R_3$} if the following holds for every 2-component square in $\Theta$.
Consider a $2$-component square in $\Theta$ visiting 
 vertices $c_1, d_1, c_2, d_2$, where $c_1, c_2 \in \L_c$,  $d_1, d_2 \in \L_d$, 
and  $\L_c, \L_d$ are distinct components of $\L$.  Then the graph $\G$ contains the join of $V(T_c)$ and $V(T_d)$, where $T_c$ and $T_d$ are the $\L$-convex hulls of $\{c_1, c_2\}$ and $\{d_1, d_2\}$ respectively. 
(See Figure~\ref{fig:C3}.)
\end{definition}

\begin{figure}[h!]
\medskip\begin{overpic}[scale=0.6]
{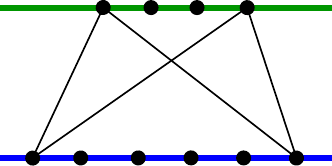}
\put(22, 52){\tiny $c_1$}
\put(76, 52){\tiny $c_2$}
\put(1, -6){\tiny $d_1$}
\put(92, -6){\tiny $d_2$}
\put(104, 45){\tiny $\L_c$}
\put(104, 0){\tiny $\L_d$}
\put(30, 52){\tiny $\overbrace{\hspace{0.6in}}^{T_c} $}
\put(10, -2){\tiny $\underbrace{\hspace{1.05in}}_{T_d} $}
\end{overpic}
\medskip
\caption{
In the figure, 
the colored parts consist of $\L$-edges, and the black parts consist of $\G$-edges.  The condition $\mathcal R_3$ says that if $\Theta$ contains a black square as shown, then every vertex of $T_c$ is joined by a $\G$-edge to every vertex of $T_d$.
} 
\label{fig:C3}
\end{figure}
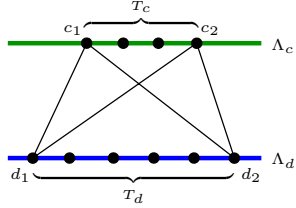

We will often need to utilize an expression for a word in $W_\G$ which is the product of $\L$-edges. This construction is the content of the following definition.
\begin{definition}[$\L$-edge words] \label{def:edge_words}
	Suppose $\Theta$ satisfies condition $\mathcal{R}_1$, and let 
 $w$ be a word in $W_\G$ such that $w = (a_1a_1')(a_2 a_2')\dots (a_na_n')$, where 
	 $a_i$ and $a_i'$ are in the same $\L$-component of $\Theta$ for each $1 \le i \le n$. 
	 As $\Theta$ satisfies $\mathcal{R}_1$, there is a unique simple 
	$\L$-path from $a_i$ to $a_i'$. Let $a_i = a_1^i, \dots, a_{m_i}^i = a_i'$ be the vertices visited by this path. Form the word:
	\[w' = \Big((a_1^1 a_2^1)(a_2^1 a_3^1)\dots (a_{(m_1 - 1)}^1 a_{m_1}^1) \Big) \dots \Big( (a_1^n a_2^n)(a_2^n a_3^n) \dots (a_{(m_n - 1)}^n a_{m_n}^n) \Big)     \]
	We call $w'$ the \textit{$\L$-edge word} associated to $w$. Note that $w'$ is well-defined, 
		as long as $\Theta$ satisfies $\mathcal{R}_1$. In particular if $(G^\Theta, E(\L))$  is a RAAG system, then $w'$ is well-defined by Lemma~\ref{lem:nocycle}.
	Also note that $w \gequal w'$ and $w'$ is a product of $\L$-edges. 
\end{definition}

\begin{remark}\label{rmk:lambda_edge_words}
	Suppose that $\Theta$ satisfies $\mathcal R_1$ and that $a, a' \in \Theta$ are two vertices in the same $\Lambda$-component. Let $w' = (a_1 a_2)(a_2 a_3) \dots (a_{n-1}a_n)$ be the $\Lambda$-edge word associated to $w = aa'$ (in particular $a = a_1, a_2, \dots, a_n = a'$ is the unique simple $\Lambda$-path from $a$ to $a'$). We remark that given a $\Lambda$-edge $xy$ of $\Theta$, there is at most one occurrence of one of $xy$ or $yx$ in $w'$. This fact will be relevant in the proofs of the next two lemmas.
\end{remark}

	Before diving into the next lemma, we briefly discuss some of the ideas used in its proof, and the proof of Lemma~\ref{lem:cycles}.
	In each case, we will have a word $w$ over the RACG $W_\Gamma$ representing the identity element.
	We then find a $\Lambda$-edge word $w'$ associated to $w$ as in Definition~\ref{def:edge_words}. The word $w'$ has a natural decomposition into $\Lambda$-edges, $w' = (s_1s_1')\dots(s_{n} s_n')$. Moreover, there is a RAAG generator $g_i \in \Delta$ associated to each $s_i s_i' = \phi(g_i)$.
	By a slight abuse of notation, we also think of $w' = g_1 \dots g_n$ as a word over the RAAG $A_\Delta$. 
	Doing so, we consider a disk diagram $D$ \emph{in the RAAG} $A_\Delta$ with boundary $g_1 \dots g_n$.
	The edges of $D$ are labeled by the $g_i$'s. To simplify things, by another abuse of notation we also think of these edges as labeled by the $\Lambda$-edges $s_is_i'$.
	We use the intersection patterns of hyperplanes in $D$ to deduce commuting relations between the generators of the RAAG. Consequently, this gives us commuting relations between the $\Lambda$-edges and for generators in the RACG $W_\Gamma$.

\begin{lemma}\label{lem:4cycle}
	If $(G^\Theta, E(\L))$ 
	is a RAAG system, then $\Theta$ satisfies $\mathcal R_3$.
\end{lemma}

\begin{remark}(Comparison of Lemma~\ref{lem:4cycle} with Laforge's chain-chord condition.) 
In~\cite[Lemma 8.2.3]{laforge}, LaForge introduced a necessary condition called the chain-chord condition which, if interpreted in the language of joins and 2-component cycles, is close to our condition $\mathcal R_3$.  We note that there are  errors in the statement and proof 
of~\cite[Lemma 8.2.3]{laforge}.  
\end{remark}

\begin{proof}[Proof of Lemma~\ref{lem:4cycle}]
Suppose there is a 2-component square $\gamma$ in 
 $\Theta$ 
visiting vertices $c_1, d_1,$ $c_2, d_2$ as in condition $\mathcal{R}_3$. Let $\L_c$ and $\L_d$ be the components of $\L$ respectively containing $\{c_1, c_2\}$ and $\{d_1, d_2\}$. Let $T_c$ and $T_d$ be the $\L$ convex hulls respectively of $\{c_1, c_2\}$ and $\{d_1, d_2\}$. 
By Lemma~\ref{lem:nocycle}, 
 there is a unique simple $\L$-path from $c_1$ to $c_2$ (resp.~$d_1$ to $d_2$) and this path is equal to 
$T_c$ (resp.~$T_d$).

Let $w$ denote the commutator $[c_1c_2, d_1d_2]$. 
 The existence of $\gamma$ tells us that $c_1$ and $c_2$ both commute with $d_1$ and $d_2$, so $w$ represents the identity in $W_\G$. 

Let $w_1$, $w_2$ and $w'$ be the $\L$-edge words associated to respectively $c_1c_2$, $d_1d_2$ and $w$.
As $\phi$ is injective $w'$ represents the trivial element of $A_\Delta$, and there is a disk diagram $D$ over $A_\Delta$ with boundary label $w'$.  
We warn that the edges of $D$ are labelled by $\L$-edges, i.e., generators of $A_\Delta$.
We will analyze hyperplanes of this diagram.  

Let $p_{w_1}$, $p_{w_2}$, $p_{w_1^{-1}}$ and $p_{w_2^{-1}}$ be the paths in $\partial D$ with labels $w_1$, $w_2$, $w_1^{-1}$ and $w_2^{-1}$ respectively.
For $i \in \{1, 2\}$, the word $w_i$ (thought of as a word over $V(\Delta) = E(\L)$) does not contain any repeated letters 
	(or their inverses) 
	in $V(\Delta)$ 
	by Remark~\ref{rmk:lambda_edge_words}.  
Consequently, a hyperplane is dual to at most one edge of $p_{w_1}$ (resp. $p_{w_2}$, $p_{w_1^{-1}}$ and $p_{w_2^{-1}}$). Furthermore, $w_1$ and $w_2$ are words over respectively $E(T_c)$ and $E(T_d)$. As $\L_c$ and $\L_d$ are distinct components of $\L$, a hyperplane dual to an edge of $p_{w_1}$ must be dual to an edge of $p_{w_1^{-1}}$ and vice versa. A similar statement holds for hyperplanes dual to $p_{w_2}$ and $p_{w_2^{-1}}$.

It follows that every hyperplane dual to $p_{w_1}$ intersects every hyperplane dual to $p_{w_2}$. Consequently, every $\L$-edge in the word $w_1$ commutes with every $\L$-edge in the word $w_2$.  Since $\phi$ is a homomorphism, the Coxeter group elements corresponding to these $\L$-edges must commute as well.  
By Lemma~\ref{lem:commuting} each vertex of $T_c$ commutes with each vertex of $T_d$. 
The result~follows.
\end{proof}

The next example shows that the conditions obtained so far are not sufficient for $(G^\Theta, E(\L))$ to be a RAAG system.

\begin{example}\label{ex:R4_necessary}
Let $\G$ be a hexagon, and let $\L$ be the graph with two components shown on the left side in Figure~\ref{fig:hexagon}.  It is clear that $\mathcal R_1, \mathcal R_2$, and $\mathcal R_3$ are satisfied.  
However, by considering the word $w=
(c_1c_2)(d_1d_2)(c_2c_3)(d_2d_3)(c_{3}c_{1})(d_{3}d_1)$
we can see that $(G^\Theta, E(\L))$ is not a RAAG system.  Specifically, the commutation relations specified by $\G$-edges show that $w\gequal 1$ in $W_\G$.  Moreover $w$ can be expressed as a product of $\L$-edges using the $\L$-edge words corresponding to each parenthetical element.  However, it is not possible to reduce this word to the empty word using just swap and deletion moves involving the $\Lambda$-edges, and as a result, $(G^\Theta, E(\L))$ cannot be a RAAG system.  A rigorous proof of this fact follows from Lemma~\ref{lem:cycles} below. 
\end{example}

\begin{figure}%[h!]
\begin{overpic}[scale=1.1]
{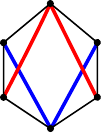}
\put(31, 105){\tiny $c_1$}
\put(78, 27){\tiny $c_2$}
\put(-12 , 27){\tiny $c_3$}
\put(78, 68){\tiny $d_1$}
\put(31, -7){\tiny $d_2$}
\put(-12, 68){\tiny $d_3$}
\put(13, 46){\tiny $x$}
\put(58, 46){\tiny $y$}
\end{overpic}
\hspace{1.3in}
\begin{overpic}[scale=1.1]
{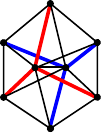}
\put(31, 105){\tiny $c_1$}
\put(78, 27){\tiny $c_2$}
\put(-12 , 27){\tiny $c_3$}
\put(78, 68){\tiny $d_1$}
\put(31, -7){\tiny $d_2$}
\put(-12, 68){\tiny $d_3$}
\put(13, 46){\tiny $x$}
\put(58, 46){\tiny $y$}
%\put(24, 22){\tiny $a_2$}
\end{overpic}
\caption{The graph on the left concerns Example~\ref{ex:R4_necessary} and the graph of the right concerns Example~\ref{ex:R4_complicated}.
}
\label{fig:hexagon}
\end{figure}
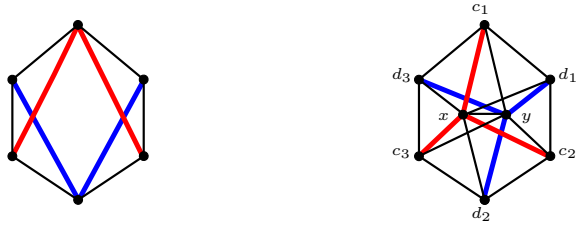

Example~\ref{ex:R4_necessary} shows that at least one additional condition is needed in order to obtain a characterization of visual RAAGs, 
and suggests that this condition may be a generalization of $\mathcal R_3$ involving longer $2$-component cycles instead of squares.
It is tempting to conjecture that, given any $\L_c\L_d$-cycle with corresponding 
$\L$-convex hulls $T_c$ and $T_d$, 
the graph $\G$ 
contains the join of $V(T_c)$ and $V(T_d)$ (as is the case when the cycle has length four, by Lemma~\ref{lem:4cycle} above). 
However, the following example shows this is not necessarily true for longer cycles. 

\begin{example}\label{ex:R4_complicated}
In Figure~\ref{fig:hexagon}, let $\Theta$ be the graph on the right where $\G$-edges are black and $\L$ edges are colored. Observe that $\L$ has two components, colored red and blue.  
Consider the 2-component cycle visiting vertices $c_1, d_1, c_2, d_2, c_3, d_3, c_1$.  
 Then $T_c$ is the entire red tree and $T_d$ is the entire blue tree.  However, $\G$ does not contain the join of $V(T_c)$ and $V(T_d)$.  (For example, 
  there is no edge in $\G$ connecting $c_1$ and $d_2$.) 
On the other hand, 
 $(G^\Theta, E(\L))$ is a RAAG system in this case.  (See 
Corollary~\ref{cor:example} for a proof.)
\end{example}

Despite the fact that $\mathcal R_3$ does not generalize to a necessary condition on longer cycles in the obvious way, the following weaker statement does turn out to be necessary to guarantee that $(G^\Theta, E(\L))$ is a RAAG system and is missing from~\cite{laforge}.

\begin{definition}[Condition $\mathcal R_4$]\label{def:R4}
	We say that  $\Theta$ satisfies \emph{condition $\mathcal R_4$} if the following holds.  
	Let $\gamma$ be   any
	$\L_c\L_d$-cycle in $\Theta$ visiting vertices 
	$c_1, d_1, c_2, d_2 \dots, c_{n}, d_{n}, c_1$ for some $n\ge 2$. 
	Let $T_c$ and $T_d$ be the $\L$-convex hulls of $\{c_1, \dots, c_n\}$ and $\{d_1, \dots, d_n\}$  respectively.
	Then every edge of $\gamma$ is contained in a 
	2-component square of $\Theta$ 
	with two vertices in $T_c$ and two vertices in $T_d$. 
	 (See Figure~\ref{fig:C4}.)
\end{definition}

\medskip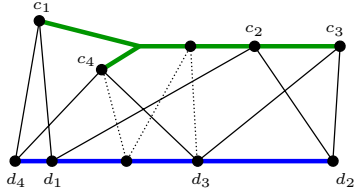
\begin{figure}[h!]
	\begin{overpic}[scale=0.6]
		{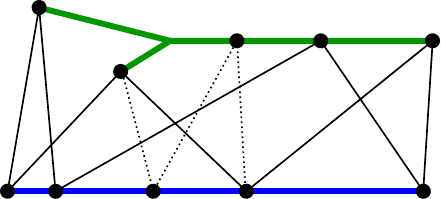}
		\put(7, 47){\tiny $c_1$}
		\put(70, 40){\tiny $c_2$}
		\put(94, 40){\tiny $c_3$}
		\put(20, 31){\tiny $c_4$}
		\put(10, -5){\tiny $d_1$}
		\put(97, -5){\tiny $d_2$}
		\put(54, -5){\tiny $d_3$}
		\put(-1, -5){\tiny $d_4$}
	\end{overpic}
	\medskip
	\caption{This figure illustrates condition $\mathcal{R}_4$. The green subgraph is $T_c$ and the blue subgraph is $T_d$.  The condition says that any edge in the 2-component cycle (shown in solid black edges) is part of a square of $\G$ with two vertices in $T_c$ and two in $T_d$.  This is illustrated for the edge from $d_3$ to $c_4$. The dotted lines are $\G$-edges which are not necessarily in the 2-component cycle.} 
	\label{fig:C4}
\end{figure}

The next lemma shows $\mathcal{R}_4$ is necessary for $(G^\Theta, E(\L))$ to be a RAAG system.

\begin{lemma}\label{lem:cycles}
If $(G^\Theta, E(\L))$ is a RAAG system, then $\Theta$ satisfies $\mathcal{R}_4$.
\end{lemma}

\begin{proof}
Let $\gamma$ be a $\L_c\L_d$-cycle visiting vertices $c_1, d_1, \dots, c_n, d_n, c_1$, and let $T_c$ and $T_d$ be as in Definition~\ref{def:R4}.
Let $w$ be the word below.
\begin{equation}\label{eq:word-g}
w= (c_1c_2)(d_1d_2)(c_2c_3)(d_2d_3)\cdots (d_{n-1}d_{n})(c_{n}c_{1})(d_{n}d_1)
\end{equation}
Then $w \gequal 1$ in $W_\G$.  To see this, note that for each $i$, we know that 
$c_i$ commutes with $d_{i-1}$ and $d_{i}$ (where $i$ is taken mod $n$).  Using this we can cancel the $c_i$ for $i>1 $ in pairs to get
$$w \gequal c_1 d_1d_2 d_2d_3\cdots  d_{n-1}d_{n}c_1 d_{n} d_1 \gequal 
 c_1 d_{1}c_{1}d_{1} \gequal 1
$$ 
Let $w'$ be the $\L$-edge word associated to $w$.
Let $D$ be a disk diagram over $A_\Delta$ with boundary label $w'$. 
As in 
 the proof of 
Lemma~\ref{lem:4cycle}, edges of $D$ are labeled by $\L$-edges, which are thought of as generators of $A_\Delta$.

Color the part of the boundary of $D$ and the hyperplanes coming out of it green if they correspond to $\L$-edges from $\L_c$ and blue if they correspond to $\L$-edges from $\L_d$.  Now we see from the structure of $w'$, that $\partial D$ alternates between green and blue stretches, and a stretch of a given color corresponds to a simple path in the corresponding component of $\L$. It follows 
from Remark~\ref{rmk:lambda_edge_words}
that a hyperplane of a given color must start and end in different stretches of that color.

Let $L=|E(T_c)|$ denote the number of $\L$-edges in $T_c$.  
We will prove that condition $\mathcal{R}_4$ holds for $\gamma$ by induction on $(n, L)$.  
The conclusion  of the lemma is obvious for $\gamma$ corresponding to $(2, L)$ for any $L$, since 
the cycle itself is a square.  
This includes the base case, when $n=2$ (i.e.~$\gamma$ is a square) and $T_c$ is an edge.  
Now let $n>2$, and assume the claim is true for all $(n', L')$ such that 
either $n'<n$ or $n' = n$ and $L'<L$.

By Lemma~\ref{lem:nocycle}, $T_c$ and $T_d$ are trees.  Now suppose $c_j$ is a leaf of $T_c$, and let $xc_j$ be the $\L$-edge incident to $c_j$ in $T_c$.  Since the $c_i \neq c_j$ for all $i \neq j$ (by the definition of a 2-component cycle), we know that $xc_j$ occurs exactly 
	once in $w'$ (as part of the subword of $w'$ representing $c_{j-1}c_j$) and $c_jx$ occurs exactly once in $w'$ (as part of the subword representing $c_jc_{j+1}$).
It follows there is a unique hyperplane $H$ labeled $xc_j$ which is dual to both the path whose label is an expression for $c_jc_{j+1}$ and the path whose label is an expression for $c_{j-1}c_j$.
Moreover, the subword $w''$ of $w'$ between these two subwords is the product of $\L$-edges which is an expression for $d_{j-1}d_j$.
It follows that every hyperplane dual to the path in $\partial D$ labeled $w''$ must intersect the hyperplane $H$.  
By Lemma~\ref{lem:commuting}, both $x$ and $c_j$ commute (in $W_\G$) with each letter of $V(\G)$ used  in the word $w''$. 
 In particular, $d_{j-1}$ and $d_{j}$ each commute with $x$.

Now there are two possibilities. Suppose first that $x=c_t$ for some $t\neq j$.  Since $t \neq j$ and $n>2$ (which implies that $\gamma$ has more than $4$ edges), it follows that 
 either $c_td_{j-1}$ or $c_td_j$ is a diagonal of $\gamma$. 
We can use this diagonal to cut $\gamma$ into two 2-component cycles $\gamma_1$ and $\gamma_2$ as follows.  Assume $c_td_j$ is a diagonal $\delta$ of $\gamma$ (the other case is analogous), and let $\beta_1$ and $\beta_2$ be the two components of $\gamma$ obtained by removing the vertices labeled $c_t$ and $d_j$.  Set $\gamma_1 = \beta_1 \cup \delta$ and $\gamma_2 = \beta_2 \cup \delta$. Note $\gamma_1$ and $\gamma_2$ each have strictly fewer vertices than $\gamma$. For $i=1,2$ let $T_c^i$ and $T_d^i$ be the components of the $\L$-convex hull of $\gamma_i$ contained respectively in $\L_c$ and $\L_d$.
By the induction hypothesis, we see that every edge in $\gamma_i$ is part of a square in $\G$ with two vertices in $T^i_c \subset T_c$ and two in $T^i_d \subset T_d$. Since each edge of $\gamma$ is either an edge of $\gamma_1$ or of $\gamma_2$, the claim follows for this case.

On the other hand, suppose that $x \neq c_i$ for any $1 \le i \le n$. Consider the new 2-component cycle $\gamma'$ obtained from $\gamma$ by replacing the edges $d_{j-1}c_j$ and $c_jd_j$ with $d_{j-1}x$ and $xd_j$.
As $x \neq c_i$ for any $1 \le i \le n$, this does not violate the requirement that $2$-component cycles do not repeat vertices.
 Let $T'_c$ and $T'_d$ be the components of the $\L$-convex hull of $\gamma'$ contained respectively in $\L_c$ and $\L_d$.
Since $c_j$ is a leaf of $T_c$, it follows that $|E(T'_c)|<  |E(T_c)|$, and we also have that $|V(\gamma')| = |V(\gamma)| = n$. We now apply the induction hypothesis to conclude that 
each edge of $\gamma'$ is part of a square of $\G$ with two vertices in $T_d'=T_d$ and two vertices in $T_c' \subset T_c$. 
This means that this property holds automatically for all edges of $\gamma$, except possibly
$d_{j-1}c_j$ and $c_jd_j$.  However, these edges are part of the square in $T_c$ with vertices $x, c_j, d_{j-1}$ and $d_j$.  Thus, the claim follows for this case as well.
\end{proof}

The following proposition summarizes Lemma~\ref{lem:nocycle}, Lemma~\ref{lem:no-almost-cycle}, Lemma~\ref{lem:4cycle} and Lemma~\ref{lem:cycles}.

\begin{prop} \label{prop_necessary}
	If $(G^\Theta, E(\L))$ is a RAAG system, then $\Theta$ satisfies $\mathcal{R}_1 - \mathcal{R}_4$.
\end{prop}

If $\L$ has at most two components, then it turns out that 
there are no additional
obstructions to $(G^\Theta, E(\L))$ being a RAAG system.
More precisely:

\begin{thm}\label{thm:2cpt} 
	Suppose $\L$ has at most two components. Then $(G^\Theta, E(\L))$ is a RAAG system if and only if $\mathcal R_1$--$\mathcal R_4$ are satisfied.
\end{thm}

\subsubsection*{Proof outline} 
Proposition~\ref{prop_necessary} constitutes one direction of the theorem. The following strategy will be used to prove that $\mathcal R_1$--$\mathcal R_4$ imply that $(G^\Theta, E(\L))$ is a RAAG system. 
We wish to show that the image of every non-trivial element of $A_\Delta$ under $\phi$  is non-trivial in $W_\G$.

Towards a contradiction, we assume that there exists some non-trivial $g\in A_\Delta$ such that $\phi(g) =1$. Then there is a disk diagram $D$ whose  boundary label is a word in $\Lambda$-edges which represents $\phi(g)$.  We will put this word in a certain normal form which will be defined in terms of the configuration of hyperplanes in $D$.   

To define the normal form, we first show that the set of all hyperplanes can be partitioned into subsets that we call ``closed chains of hyperplanes'' (see Definition~\ref{def:hyperplane-chain} and Figure~\ref{fig:closed-chain}).  Properties of  hyperplanes and closed chains can be translated into information about the graph $\Theta$ and vice versa (see Observations~\ref{obs:C2},~\ref{obs:lambda_paths}, and~\ref{obs:intersecting_chains}).  Next, we prove, in Lemma~\ref{lem:special-chain}, that we can fix a particular closed chain $\mathcal H$ which is ``maximally nested'' in a certain sense.  Specifically, $\mathcal H$ has a distinguished hyperplane $H_0$, such that every other closed chain either intersects $H_0$ or is separated from the rest of $\mathcal H$ by $H_0$ (see Figure~\ref{fig:special-chain}).  

Our normal form is defined in terms of the fixed closed chain $\mathcal H$.  We first choose a basepoint $p$ on $\partial D$ which is the endpoint of an edge of $\partial D$ dual to $H_0$. (This has the effect of possibly replacing our original element $g\in A_\Delta$ with a conjugate.)   Let $w$ be the label of $\partial D$ read clockwise starting at $p$.  We show in Claim~\ref{claim:lambda-edge-right}  that $w$, $D$ and $\mathcal H$ may be replaced by an equivalent word $\tilde w$ and corresponding disk diagram $\tilde D$ and maximally nested closed chain $\widetilde{\mathcal H}$, 
with the property that the $\Lambda$-edges in $\tilde w$ coming from 
$\widetilde{\mathcal H}$ are 
``as far right as possible'', i.e.~it is not possible to swap one of these $\Lambda$-edges with a $\Lambda$-edge to its right by a commutation relation.  We consider $\tilde w$ to be a word in normal form representing $\phi(g)$.  

Finally, to complete the proof of Theorem~\ref{thm:2cpt}, we will show (by analyzing interactions between closed chains in $\tilde D$) that if $\mathcal R_1$--$\mathcal R_4$ are satisfied, then the normal form is violated.

\bigskip

Before we embark on the proof, we need to develop some preliminaries on disk diagrams, and on transferring information from the disk diagram $D$ to the graph $\Theta$.  In what follows, we assume that $D$ is a disk diagram whose boundary is a word 
 $w$ in the RACG $W_\Gamma$.
Unlike in the proofs of Lemmas~\ref{lem:4cycle} and~\ref{lem:cycles}, 
 we are now working in $W_\G$ rather than $A_\Delta$, so the edges and hyperplanes of $D$ are labeled by generators of $W_\G$ rather than elements of $A_\Delta$ corresponding to $\L$-edges.
As the words $w$ we consider are the images of elements of $A_\Delta$ under $\phi$, they have 
  a natural decomposition into $\L$-edges.

We associate 
	a color (red and green)
to each component of $\L$.  Each hyperplane 
	of $D$
then  inherits the color corresponding to the component of $\L$ in which its label lies. 
Thus, two edges of $\partial D$ 
contained in the same $\L$-edge 
are dual to hyperplanes of the same color.

\begin{obs}\label{obs:C2}
If $\Theta$ satisfies $\mathcal R_2$, then no two hyperplanes of the same color 
 intersect.
This is because if two hyperplanes intersect, then their labels are distinct  and commute, and so are connected by a $\G$-edge.  Thus they cannot be in the same component of $\L$, since each component of $\L$ is an induced subgraph of $\Theta$, by $\mathcal R_2$.
\end{obs}
The hyperplanes of $D$ can be partitioned 
into ``closed chains of hyperplanes,'' as described in Definition~\ref{def:hyperplane-chain}
below.  Although the proof of Theorem~\ref{thm:2cpt} only uses  disk diagrams whose boundary labels are words in $\L$-edges, the definition below applies to slightly more general disk diagrams, as this will be needed in Section~\ref{sec:reflections}.

\begin{figure}
\begin{overpic}[scale=1.3]
{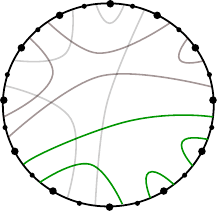}
\put(65, 26){\small $\mathcal H$}
\put(50,34){\scriptsize $H_0$}
\put(87,35){\scriptsize  $H_1$}
\put(74, 18){\scriptsize  $H_2$}
\put(49, 18){\scriptsize  $H_3$}
\put(4.5, 18){\scriptsize $e_0$}
\put(100, 43){\scriptsize $f_0$}
\put(98, 32){\scriptsize $e_1$}
\put(90, 18){\scriptsize $f_1$}
\put(81, 9){\scriptsize $e_2$}
\put(68, 0.5){\scriptsize $f_2$}
\put(55, -2.5){\scriptsize $e_3$}
\put(14, 7){\scriptsize $f_3$}
\end{overpic}
\hspace{1.5cm}
\begin{overpic}[scale=1.3]
{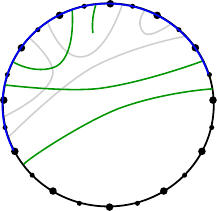}
\put(50,40){\scriptsize $H_0$}
\put(4.5, 18){\scriptsize $e_0$}
\put(100, 54){\scriptsize $f_0$}
\put(97, 70){\scriptsize $e_1$}
\put(15, 7){\scriptsize $f$}
\put(15, 88){\small \color{blue} $\mu$}
\end{overpic}
\vspace{0.5cm}
\caption{The figure on the left shows a disk diagram $D$ such that the label of $\partial D$ 
 has a natural decomposition into $\L$-edges, delineated by large black dots.   The green hyperplanes form a closed chain of hyperplanes $\mathcal H$, as defined in 
Definition~\ref{def:hyperplane-chain} (here we can take $\eta =  \partial D$).  Two other closed chains of hyperplanes are shown in grey.  The figure on the right shows an impossible configuration pertaining to the proof of Lemma~\ref{lem:partition}. The path $\mu$ from the lemma is colored blue. 
 }
\label{fig:closed-chain} 
\end{figure}

\begin{definition}(Chains of hyperplanes)\label{def:hyperplane-chain}
	Let $D$ be a disk diagram whose boundary $\partial D$ contains a connected subpath $\eta$ (possibly all of $\partial D$), 
such that the label of $\eta$  is a word in $\L$-edges. 
Let $H_0, \dots, H_n$ be a sequence of distinct hyperplanes in $D$. Let $e_i$ and $f_i$ be the edges on $\partial D$ that are dual to $H_i$. 
(See Figure~\ref{fig:closed-chain} for an illustration when $n=3$.)
	We say that $\{H_0, \dots, H_n\}$ is a \textit{chain} in $D$, if for all $0 \le i < n$, the edges $f_i$ and $e_{i+1}$ 
	are contained in $\eta$ and are dual to the same $\L$-edge of $\eta$. Note that $e_0$ and $f_n$ can be dual to edges not contained in $\eta$. 
	
	Additionally, if $e_0$ and $f_n$ are contained in the same $\L$-edge of $\eta$, we say that $\{H_0, \dots, H_n\}$ is a \textit{closed chain}.
(Figure~\ref{fig:closed-chain} shows three closed chains.)
\end{definition}

Since the two hyperplanes dual to a $\L$-edge have the same color, each chain also inherits a well-defined color.

\begin{lemma} \label{lem:partition}
If the label of $\partial D$ is 
 a word in $\L$-edges, then every hyperplane of $D$ is contained in a unique closed chain. 
  Thus, there is a partition of the hyperplanes of a given color into closed chains.  
\end{lemma}
\begin{proof}
Let $H_0$ be a hyperplane of $D$ dual to edges $e_0$ and $f_0$ of $\partial D$.  Assume $H_0$ is green.  Let $f$ and $e_1$ be edges of $\partial D$ which pair with $e_0$ and $f_0$ respectively to form $\L$-edges. We claim that $f$ and $e_1$ are in the same component of $D \setminus H_0$.  
If not, there would be an odd number of edges in a part $\mu$ of $\partial D$ between $e_0$ and $f_0$ (see the  right side of Figure~\ref{fig:closed-chain}).    
Since the  hyperplanes dual to the two edges of a $\L$-edge have the same color, an odd number of these edges would be dual to green hyperplanes.  This is a contradiction, since no green hyperplanes can cross $H_0$ by Observation~\ref{obs:C2}, so there must be an even number of edges in $\mu$ dual to green hyperplanes.

The hyperplane $H_1$ dual to $e_1$ is green, and cannot cross $H_0$.  Let $f_1$ be the other edge dual to $H_1$.  If $f_1=f$ we have a closed chain.  Otherwise, there is an edge $e_2$ which pairs with $f_1$ to form a $\Lambda$-edge.  By the same argument as before, $e_2$ 	is in the same component of $D \setminus H_1$ as $f_0$ and there is a green hyperplane $H_2$ dual $e_2$ and another edge $f_2$, such that $H_2$ does not cross $H_0$ or $H_1$ (see the left side of Figure~\ref{fig:closed-chain}).    Continuing this process we obtain a sequence of hyperplanes as in Definition~\ref{def:hyperplane-chain}. Since the number of possibilities for $f_i$ reduces each time, eventually the process stops, with $f_n=f$ for some $n$, and $H_0, \dots H_n$ form a closed chain. 
The lemma follows. 
\end{proof}

We say that a chain $\mathcal K$ intersects a hyperplane $H$, if some $K\in \mathcal K$ intersects $H$.   We say that chains $\mathcal H$ and $\mathcal K$ intersect if 
$\mathcal K$ intersects some $H \in \mathcal H$.   We will need the following observation:

\begin{obs}\label{obs:2intersections}
If a hyperplane $H$ intersects a closed chain $\mathcal K$, then it intersects $\mathcal K$ in exactly two distinct hyperplanes.  To see this, note that given a hyperplane $K\in \mathcal K$, the hyperplanes in $\mathcal K \setminus \{K\}$ all lie in a single component of $D \setminus K$.  It follows that if $H$ intersects $\mathcal K$ more than twice, it must intersect some hyperplane of $\mathcal K$ twice.  This contradicts the fact that $D$ has no bigons (see Remark~\ref{rem:no-bigons}).
\end{obs}

The following two observations enable us to transfer information from the disk diagram $D$ to the graph $\Theta$. 

\begin{figure}
\begin{overpic}[scale=1.3]
{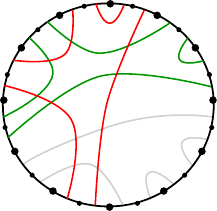}
\put(45, -12){$ D$}
\put(36, 42){\small $\mathcal H$}
\put(70, 70){\small $\mathcal K$}
\put(4, 76){\small $p$}
\put(25.5, 62){\scriptsize $a$}
\put(44, 74){\scriptsize $b$}
\put(42, 56){\scriptsize $b$}
\put(79, 67){\scriptsize $c$}
\put(31, 68.5){\scriptsize $x$}
\put(22, 52.5){\scriptsize $y$}
\put(51, 66){\scriptsize $z$}
\put(53.5, 85){\scriptsize $y$}
\end{overpic}
\hspace{1.5cm}
\begin{overpic}[scale=1]
{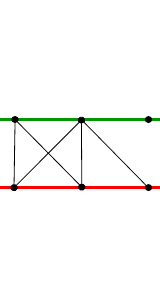}
\put(25, 10){$\Theta$}
\put(5, 62){\scriptsize $a$}
\put(28, 62){\scriptsize $b$}
\put(50, 62){\scriptsize $c$}
\put(5, 28){\scriptsize $x$}
\put(28, 28){\scriptsize $y$}
\put(50, 28){\scriptsize $z$}
\end{overpic}

\vspace{0.75cm}
\caption{The figure illustrates Observations~\ref{obs:lambda_paths} and \ref{obs:intersecting_chains}.
On the left is a disk diagram $D$ with red and green closed chains called $\mathcal H$ and $\mathcal K$ respectively.   The graph on the right is a part of $\Theta$.  The labels of the hyperplanes in $D$ correspond to  vertices of $\Theta$ (in particular of $\L$).   Starting at the basepoint $p$ and going around $\partial D$ clockwise, the closed chain $\mathcal H$ defines a $\Lambda$-loop in $\Theta$ visiting vertices $x, y, z, y, x$ and the closed chain $\mathcal K$ defines a $\Lambda$-loop visiting vertices $a, b, c, b, a$.  The polygon coming from the intersection of $\mathcal H$ and $\mathcal K$ defines a $2$-component loop in $\Theta$ visiting vertices $a, x, b, z, b, y, a$.  Observe that this $2$-component loop is not a cycle. 
 }
\label{fig:closed-chain} 
\end{figure}

\begin{obs}(Chains in $D$ give $\L$-paths in $\Theta$.)\label{obs:lambda_paths}
Let $\mathcal K = \{K_0, \dots, K_l\}$ be a chain in $D$, and for $0\le i \le l$, let $k_i$ be the label of $K_i$.  Then by the definition of a chain, 
$K_i$ and $K_{i+1}$ are dual to the same $\L$-edge in $\partial D$  for each $i$, so there is an edge in $\L$ between 
$k_i$ and $k_{i+1}$.  It follows that $\mathcal K$ naturally defines a $\L$-path in $\Theta$ visiting vertices $k_0, k_1, \dots, k_l$.  Moreover, if $\mathcal K$ is a closed chain, then the corresponding $\L$-path is a loop.  
See Figure~\ref{fig:closed-chain}. 
\end{obs}

\begin{obs}
\label{obs:intersecting_chains}
(Pairs of intersecting closed chains give 2-component loops in $\Theta$.)
Consider two closed chains which intersect, say a red chain $\mathcal H $ and a green chain  $\mathcal H$. Let $H_1\in \mathcal H$ and $K_1 \in \mathcal K$ be intersecting hyperplanes.  
By Observation~\ref{obs:2intersections}, the hyperplane $K_1$ intersects $\mathcal H$ in a second hyperplane $H_2 \neq H_1$.  Similarly, $H_2$ intersects $\mathcal K$ in a second hyperplane $K_2$.  Proceeding in this way, we obtain a polygon with at least four sides, with sides alternating between red and green hyperplanes.  
See the left side of Figure~\ref{fig:closed-chain}.

Since an intersecting pair of hyperplanes corresponds to an edge of $\G$,  
a 2-colored polygon of the type we just constructed defines a 2-component loop
 in $\Theta$
  (where each edge of the 2-component loop comes from a corner of the 2-colored polygon).
See Figure~\ref{fig:closed-chain}. 
 We warn that  the 2-component loop obtained from a 2-colored polygon in $D$ may not be a 2-component cycle. (Note that a 2-component cycle is a 2-component loop in which all of the vertices are distinct, and there are at least two vertices in each component.)
\end{obs}

In order to define a normal form for the word $u$ from the proof outline, we will need to choose a closed chain in $D$ with some special properties:

\begin{lemma}\label{lem:special-chain}
Let $u$ and $D$ be as in the proof outline.
There exists a closed chain $\mathcal H$ 
of $D$, containing a distinguished hyperplane $H_0$,
 such that given any closed chain $\mathcal K \neq \mathcal H$, one of the following holds:

\begin{enumerate}
\item $\mathcal K$ and $\mathcal H \setminus \{H_0\}$
lie in different components of $D\setminus H_0$.
\item $\mathcal K$ intersects $H_0$.
\end{enumerate}
\end{lemma}
\begin{proof}
We 	iteratively
 construct a sequence of closed chains 
$\mathcal H^1, \mathcal H^2, \dots$ with distinguished hyperplanes 
$  H_0^1, H_0^2,  \dots$, such that for all $i>1$, we have:
\begin{enumerate}
\item[(i)]$\mathcal H^i$ and $\mathcal H^{i-1} \setminus \{H_0^{i-1}\}$ lie in the same component of $D \setminus H_0^{i-1}$, and 
\item[(ii)] $H_0^{i-1}$ and $\mathcal H^i \setminus \{H_0^i\}$ lie in different components of 
$D \setminus H_0^i$.
\end{enumerate}
\begin{figure}
\begin{overpic}
{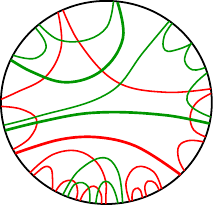}
\put(72, 26){\tiny $\mathcal H^1$}
\put(75, 60){\tiny $\mathcal H^2$}
\put(47, 75){\tiny $\mathcal H^3$}
\put(40,34.5){\tiny $H_0^1$}
\put(70,45){\tiny $H_0^2$}
\put(29, 61){\tiny $H_0^3$}
\end{overpic}
\caption{The figure illustrates the procedure for finding $\mathcal H$ and $H_0$ in Lemma~\ref{lem:special-chain}.  Each closed chain in the sequence is labeled in its interior.  The hyperplanes $H_0^i$ are shown in bold. In this example $\mathcal H = \mathcal H^3$ has the desired property.}
\label{fig:special-chain} 
\end{figure}
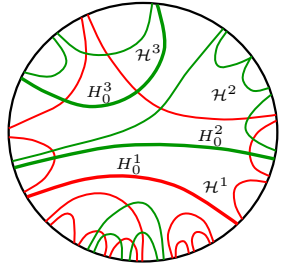

Let $\mathcal H^1$ and $ H_0^1$ be arbitrary.  Now for any $j$, if  
$\mathcal H^j$ and $H_0^j$ do not satisfy the conclusion of the lemma, then there must exist another closed chain $\mathcal H^{j+1}$ 
which lies entirely 
in $C_j$, where $C_j$ is 
 the component  of $D \setminus H_0^j$ containing $
\mathcal H \setminus \{H_0^j\}$.
 (Figure~\ref{fig:special-chain} illustrates this for $j=1, 2$.)
 There is a unique hyperplane
  in $\mathcal H^{j+1}$ satisfying condition (ii) above with $i=j+1$, and we set this equal to $H_0^{j+1}$. 
Thus, we can produce a longer sequence of closed chains 
 with properties (i) and (ii).

By construction, there is a nesting of components $C_1 \supset C_2 \supset C_3 \dots$, and it follows that $H_0^1, H_0^{2}, \dots $ are distinct hyperplanes in $D$.
As $D$  has finitely many hyperplanes, this process can only be repeated finitely many times. Thus, $\mathcal H ^j$ satisfies the claim for some $j$.
\end{proof}

We are now ready to prove the theorem.  

\begin{proof}[Proof of Theorem~\ref{thm:2cpt}]
As discussed, 
 we need to show that if $\mathcal R_1$--$\mathcal R_4$ are satisfied, then  
the map  
 $\phi: A_\Delta \to G^\Theta$ is injective.
Let $g\in A_\Delta$ be a non-trivial element. 
Let $v=v_1v_2\cdots v_n$ be a reduced word over the set of the generators of $A_\Delta$, which represents~$g$.  By the definition of $A_\Delta$, we have that $\phi(v_i)$ is a $\Lambda$-edge of $\Theta$, for $1 \le i \le n$.  Then $u = \phi(v_1)\phi(v_2)\cdots    \phi(v_n)$ is a concatenation of $\Lambda$-edges which represents $\phi(g)$.  
Towards a contradiction, we assume that $u$ represents the identity element of $W_\G$. 
Then there is a disk diagram $D$ whose  boundary label (read clockwise  starting from some basepoint)  is $u$. 
By Lemma~\ref{lem_remove_pathologies} we may assume that $D$ has no bigons. 

An element  has 
trivial
image under $\phi$ if and if  every element of its conjugacy class does. Thus, 
we  may assume that 
$g$ is of minimal length in its conjugacy class, where the length of an element is defined to be the minimal length of a word representing it.  

We partition the hyperplanes of $D$ into  closed chains.  (See Lemma~\ref{lem:partition}.)
By Lemma~\ref{lem:special-chain}, we can choose a chain $\mathcal H$, with distinguished hyperplane $H_0$, such that given any other chain $\mathcal K$, either $H_0$ separates $\mathcal K$ from $\mathcal H\setminus \{H_0\}$, or $\mathcal K$ intersects $H_0$.  
Let $a_0, a_1, \dots, a_{s}$ be the labels of the hyperplanes of $\mathcal H$, starting from $H_0$, and proceeding in order in the clockwise direction around $\partial D$.  
Then the $\L$-edges $a_0a_1, a_1a_2, \dots, a_{s-1}a_s, a_sa_0$ appear in $\partial D$ in that order, possibly interspersed with some other $\L$-edges.

Let $p$ denote the vertex on $\partial D$ which is the endpoint of the $\L$-edge from $\mathcal H$ labeled $a_sa_0$, read clockwise.  (See Figure~\ref{fig:HKpoly}.) Let $w$ be the word labeling $\partial D$ clockwise, starting from $p$.  Then $w$ is a cyclic conjugate of $u$.  Let $x$ be the corresponding cyclic conjugate of $v$.  Since $v$ was chosen to be reduced, and since $g$  (the element of $A_\Delta$ represented by $v$) 
is  of minimal length in its conjugacy class by assumption, it follows that $x$ is reduced.

\begin{figure}
\begin{overpic}
{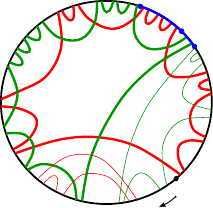}
\put(30, 34){\tiny $c_0$}
\put(18, 40){\tiny $c_1$}
\put(24, 59){\tiny $c_2$}
\put(50, 78){\tiny $c_3$}
\put(68, 68){\tiny $c_4$}
\put(17, 32){\tiny $b_0$}
\put(10, 47){\tiny $b_1$}
\put(37, 72){\tiny $b_2$}
\put(60, 73){\tiny $b_3$}
\put(48, 49){\tiny $b_4$}
\put(-0.5, 20){\tiny $a_0$}
\put(-3, 25){\tiny $a_1$}
\put(-5, 30){\tiny $b_0$}
\put(-7, 35){\tiny $b_1$}
\put(28, 96 ){\tiny $\mathcal H$}
\put(5, 80){\tiny $\mathcal K$}
\put(82, 8){\tiny $ p$}
\put(85, 11){\tiny $ a_0$}
\put(88, 15){\tiny $ a_s$}
\put(81, -1){\tiny $ w$}
%\put(92, 76){\tiny $ e'_4$}
\end{overpic}
\caption{This example illustrates the proof of Theorem~\ref{thm:2cpt}. The chain ${\mathcal H}$ satisfying Claim~\ref{claim:lambda-edge-right}  is shown in thick red lines.  
In particular, no $\L$-edge from ${\mathcal H}$ (except possibly the last one) commutes with the $\L$-edge appearing after it in $\partial D$.  The chain $\mathcal K$, which contributes the first $\L$-edge not in ${\mathcal H}$ (after $a_0a_1$),
is shown in thick green lines.  The polygon formed by the intersection of ${\mathcal H}$ and $\mathcal K$ induces a 2-component loop which visits (in this example) $c_0, b_0, c_1, b_1, c_2, b_2, c_3, b_3, c_4, b_4, c_0$.  The blue subpaths of $\partial D$ are the subpaths defined in Claim~\ref{clm:dual_hyperplanes}, with $i=4$. }
\label{fig:HKpoly}
\end{figure}
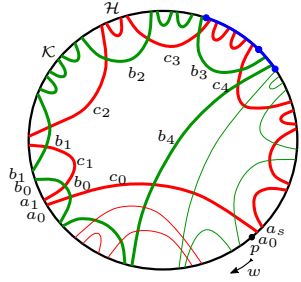

We now show that we can modify $D$ in such a way 
that the resultant boundary label is a word 
representing
 $w = \phi(x)$ 
 which is in a certain normal form:

\begin{claim}\label{claim:lambda-edge-right}
There exists a disk diagram $\tilde D$ 
such that the following hold. 
\begin{enumerate}
\item  There is a closed chain $\widetilde {\mathcal H}$ in $\tilde D$ which has a 
distinguished hyperplane $\tilde H_0$ satisfying the criterion in Lemma~\ref{lem:special-chain}.   The labels of the hyperplanes of $\widetilde {\mathcal H}$ starting from $\tilde H_0$ and proceeding clockwise, are $a_0, \dots, a_s$ (i.e.~they are the same labels as the labels of the hyperplanes in $\mathcal H$).  
\item 
Let  $\tilde p$ be the endpoint of the $\L$-edge $a_sa_0$ from $\widetilde{\mathcal H}$, 
 and let $\tilde w$ be the word  labeling $\partial \tilde D$ in the clockwise direction starting from $\tilde p$.  Then 
$\tilde w = \phi(\tilde x)$, where $\tilde x$ is a reduced word in $A_\D$ obtained from $x$ by Tits swap moves. 
 
\item  The $\L$-edges from $\widetilde {\mathcal H}$  appear as far right as possible in $\tilde w$. More formally, the word $\tilde w$ has no  subword of the form $a_{i}a_{i+1} bb'$
such that
$a_{i}a_{i+1}$ is one of the $\L$-edges coming from $\widetilde{\mathcal H}$ with $0 \le i \le s$ (with indices mod $s$), $a_{i}a_{i+1}\neq bb'$, 
and $a_{i}a_{i+1}$ commutes with $bb'$.  
\end{enumerate}

\end{claim}

\begin{proof}
We construct $\tilde D$ iteratively, starting with $D$.  If $\mathcal H, p, w,$ and $x$ are as defined above, then the first two conditions in the claim are satisfied.  If (3) is not satisfied, then  
$w$  has a subword $a_{i}a_{i+1} bb'$ as in (3).    
Since $a_sa_0$ is the last $\L$-edge of $w$, 
we conclude that $a_{i}a_{i+1} \neq a_sa_0$.  

Note that $a_{i}a_{i+1}\neq bb'$ (by condition (3)) and $ a_ia_{i+1}\neq (bb')^{-1}$ (since $x$ is reduced and $w = \phi(x)$).  Then it follows from 
 Lemma~\ref{lem:commuting}, 
that each of $a_i$ and $a_{i+1}$ commutes with each of $b$ and $b'$. 
 By applying Lemma~\ref{lem:new_diagrams}(1) four times, we obtain a new disk diagram $D'$ such that 
the label of $\partial D'$ is obtained from the label of $\partial D$ by swapping the 
$\L$-edges $a_{i}a_{i+1}$ and $bb'$.  Moreover, 
the natural map $\psi$ from the edges of $\partial D$ to the edges of $\partial D'$ (defined in 
Lemma~\ref{lem:new_diagrams}(1)) preserves boundary combinatorics.  By applying Lemma~\ref{lem_remove_pathologies} if necessary, we may assume that $D'$ has no bigons, so hyperplanes in $D'$ intersect at most once.  

Since boundary combinatorics are preserved, $\psi$ induces a bijection between the hyperplanes dual to $\partial D$ and  those dual to $\partial D'$.
Since the transition from $D$ to $D'$ involves swapping a pair of $\L$-edges, the label of $\partial D'$ is still a product of $\L$-edges, and so the hyperplanes of $D'$ can be partitioned into closed chains of hyperplanes.  Moreover, $\psi$ induces a bijection between the closed chains of hyperplanes in $D$ and~$D'$.  

If $\mathcal H'$ and $H_0'$ denote the images of $\mathcal H$ and $H_0$ respectively under $\psi$, it is clear that the labels of the hyperplanes of ${\mathcal H'}$, starting from $ H_0'$ and proceeding clockwise, are $a_0, \dots a_s$. We now prove that $\mathcal H'$ together with $H_0'$ 
still satisfies the criterion in Lemma~\ref{lem:special-chain} 
 required in~(1).

Let $\mathcal K'$ be a closed chain in $D'$, and let $\mathcal K$ be its preimage in $D$.  Our choice of $\mathcal H$ implies that either $H_0$ separates $\mathcal K$ from $\mathcal H \setminus \{H_0\}$, or $\mathcal K$  intersects $H_0$.  
In the former case, $H_0'$ still separates $\mathcal K'$ from  $\mathcal H' \setminus \{H_0'\}$.  This is because the swap performed does not involve any hyperplanes from chains which do not intersect $H_0$, since (as noted above) $a_ia_{i+1} \neq a_sa_0$.

On the other hand, suppose that $\mathcal K$ intersects $H_0$.  By Observation~\ref{obs:2intersections}, there are exactly two hyperplanes 
$K_1$ and $K_2$ in $\mathcal K$ which intersect $H_0$. 
 If $K_j$, for $j = 1,2$, is not dual to the $\L$-edge labeled by $bb'$, then the image of $K_j$ intersects $H_0'$. 
 Moreover, if $i \neq 0$, then it follows that the images of $K_1$ and $K_2$ in $D'$ intersect the hyperplane $H_0'$.
 Thus, we only need to consider the case where the $\L$-edge $a_0a_1$ is swapped, and (up to relabeling) $K_1$ is dual to $b$ and $K_2$ is dual to $b'$. In this case, $K_1$ and $K_2$ are dual to the same $\L$-edge. It follows that no hyperplane in $\mathcal K \setminus \{K_1, K_2\}$ is contained in the same component of $D \setminus H_0$ as $\mathcal H \setminus 
\{H_0\}$. 
Thus, in $D'$, no hyperplane of $\mathcal K'$ is contained in the same component of  $D' \setminus H'_0$ as $\mathcal H' \setminus \{H'_0\}$. We have shown that $\mathcal H'$, with distinguished hyperplane $H_0'$, satisfies the conclusion of Lemma~\ref{lem:special-chain}.

Let $p'$ be the vertex on $\partial D'$ which is the endpoint of the $\L$-edge from $\mathcal H'$ labeled $a_sa_0$. Since the swap performed did not involve $a_sa_0$, the label $w'$ of $\partial D'$, read clockwise from $p'$, is obtained from $w$ by swapping a single pair of  $\L$-edges, and its preimage in $x' $ in $A_\Delta$ is obtained from $x$ by swapping one pair of generators. 
 This shows (2).

We have established that  $D'$, together with $\mathcal H'$,  satisfies (1) and (2) of Claim~\ref{claim:lambda-edge-right}.
 If (3) still fails, we may repeat the process above.  
 Since each individual iteration involves moving one $\L$-edge from the 
 image of $\mathcal H$  to the right, this process eventually stops. 
 After finitely many iterations, we arrive at 
a disk diagram $\tilde D$
such that all three conditions hold. 
\end{proof}

For the rest of the proof we assume, without loss of generality, that $D, \mathcal H, p, w,$ and $x$ satisfy the conclusion of Claim~\ref{claim:lambda-edge-right}.  

We now analyze closed chains which intersect ${\mathcal H}$.  First consider the case that there are no such chains.  This includes the case when $\L$ has a single component.
Since $\mathcal H$ is a closed chain, it defines a loop in $\L$. (See Observation~\ref{obs:lambda_paths}.)
On the other hand, since no chains intersect $\mathcal H$,  the union of the edges of $\partial D$ dual to the hyperplanes of $\mathcal H$ is a continuous subpath (with label $(a_0a_1)(a_1a_2) \cdots (a_sa_0)$).  Applying the following claim to this subpath, we conclude that the $\L$-loop defined by $\mathcal H$ is a cycle.  This contradicts $\mathcal R_1$. 
(The claim will be used again later in this proof.)

\begin{claim}\label{claim:simple_path}
Let $\nu$ be a subpath of $\partial D$ labeled by a product of $\L$-edges.  Suppose there exists a closed chain $\mathcal X$, such that each edge of $\nu$ is dual to a hyperplane in $\mathcal X$.  
 It follows that the label of $\nu$ is 
$(x_1x_2)\cdots(x_{n-1}x_n)$, where $x_1, x_2,\dots, x_n$ are the labels of the hyperplanes of
$\mathcal X$ dual to $\nu$, in order.  Furthermore, the $\L$-path through vertices $x_1, \dots x_n$ is simple.

\end{claim}
\begin{proof}
The claim about the label of $\nu$ is immediate.  If the path through vertices $x_1, \dots, x_n$ is not simple, then there is a $\L$-loop through vertices $x_i, x_{i+1}, \dots,  x_{i+j}=x_i$ for some $i, j$.  By $\mathcal R_1$, the image of this loop in $\L$ is a tree.  Let $x_r$ be a leaf of this tree, with $i<r<j$.
It follows that $x_{r-1} = x_{r+1}$.  Consequently, the label of $\nu$ (and therefore of the word  $w$) has a subword 
$(x_{r-1}x_r )(x_rx_{r-1})$.  This is a contradiction, as it implies that the preimage $x$ of $w$ in $A_\Delta$ is not reduced. 
\end{proof}

Thus, we may assume that there is at least one chain intersecting $\mathcal H$.  In particular, 
$\L$ has two components: say a red component $\L_a$ 
which contains the labels of $\mathcal H$, and a green component $\L_b$.
By Claim~\ref{claim:lambda-edge-right}, each chain intersecting $\mathcal H$ intersects $H_0$. 
Let $\mathcal K$ be the ``first''  such chain, in the sense that the first $\L$-edge from a chain other than $\mathcal H$ appearing in $w$ to the right of $a_0a_1$ is from $\mathcal K$. 
 (See Figure~\ref{fig:HKpoly}.)
By $\mathcal R_2$ and Observation~\ref{obs:C2}, we conclude that $\mathcal K$ is green.
Let $b_0, \dots b_{s'}$ be the labels of the hyperplanes of $\mathcal K$, where $b_0b_1$ is the label of the first $\L$-edge from $\mathcal K$ appearing in $w$ to the right of $a_0a_1$.

Now consider the 2-colored polygon in $D$ whose sides alternate between hyperplanes in ${\mathcal H}$ and $\mathcal K$, as described in Observation~\ref{obs:intersecting_chains}.  Let $c_0, d_0, \dots c_k, d_k$ be the labels of these sides, where $c_0=a_0$, 
$d_0=b_0$,
and $c_0, \dots c_k$ 
(resp.~$d_0, \dots, d_k$)
is a subsequence of $a_0, \dots a_s$
(resp.~of $b_0, \dots, b_{s'}$).
 (See Figure~\ref{fig:HKpoly}.)

The following technical claim about the hyperplanes dual to certain subpaths of $\partial D$  associated to this $2$-colored polygon will be needed in what follows:

\begin{claim}\label{clm:dual_hyperplanes}
For $0\le i \le {k}$,  let $e_i$ and $f_i$ (resp.~$e_i'$ and $f_i'$) be the edges dual to the hyperplane of $\mathcal H$ labeled $c_i$ (resp.~the hyperplane of $\mathcal K$ labeled $d_i$), where $e_i$ 
(resp.~$e_i'$) appears before $f_i$ (resp.~$f_i'$) reading clockwise from $p$.  

For $i>0$, let  $\eta_i$ be the subpath of $\partial D$ from (and including) $f_{i-1}$ to (and including) $e_{i}$, and let  $\mu_i$ be the subpath of $\partial D$ from the endpoint of $\eta_i$ to  (and including) $e_{i}'$. (Figure~\ref{fig:dual_hyperplanes}.)
Then every edge of $\eta_i$ (resp.~$\mu_i$) is dual to a hyperplane in $\mathcal H$ (resp.~$\mathcal K$).   
\end{claim}

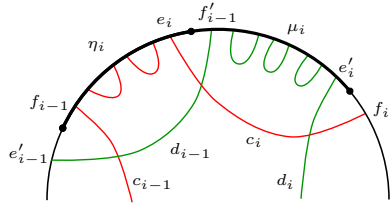
\begin{figure}
\medskip
\begin{overpic}[scale=0.7]{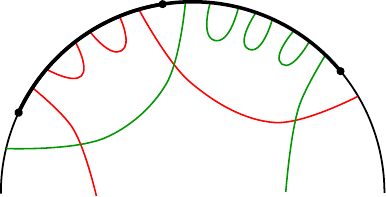}
\put(12, 45){\tiny $ \eta_i$}
\put(70, 50){\tiny $ \mu_i$}
\put(25, 5){\tiny $c_{i-1} $}
\put(36, 14){\tiny $d_{i-1} $}
\put(58, 17){\tiny $c_i $}
\put(67, 5){\tiny $d_i $}
\put(32, 52){\tiny $e_i $}
\put(95, 26){\tiny $f_i $}
\put(-5, 28){\tiny $f_{i-1} $}
\put(-11, 13){\tiny $e'_{i-1} $}
\put(44, 54){\tiny $ f'_{i-1}$}
\put(85, 38){\tiny $e'_{i} $}
\end{overpic}
\caption{The paths $\eta_i$ and $\mu_i$ from Lemma~\ref{clm:dual_hyperplanes} are shown in bold, delineated by dots. We remark that if $i=k$, then there could be additional hyperplanes not in $\mathcal K$ or $\mathcal H$ between the endpoint of $\mu_k$ and the start of the edge $f_i$.  }
\label{fig:dual_hyperplanes}
\end{figure}

\begin{proof}
Suppose there is some hyperplane $L$ dual to an edge $e$ of $\eta_i$ such that the closed chain $\mathcal{L}$ containing $L$ is not equal to $\mathcal H$.
From the definition of $\eta_i$, we conclude that $e$
is on the same side of $H_0$ as $\mathcal H\setminus \{H_0\}$ in $D$, and by our choice of 
 $\mathcal H$ (and Lemma~\ref{lem:special-chain}), it follows that $\mathcal L$ intersects $H_0$. 
Therefore, $L$ is green by Observation~\ref{obs:C2}. 

Let $K$ denote the hyperplane of $\mathcal K$ labeled $d_{i-1}$.   Then 
$K$ separates $e$ from $\mathcal K\setminus \{K \}$, so $L \notin \mathcal K$, i.e.~$\mathcal L \neq \mathcal K$.  If $d_{i-1} =  b_0$, i.e.~if $K$ does intersect $H_0$, then
our choice of $\mathcal K$ implies that $K$ is the first 
 hyperplane not in $\mathcal H$ dual to $\partial D$ after the $\L$-edge $a_0a_1$, so such an $L$ cannot exist.  On the other hand, 
if $K$ does not intersect $H_0$, then  $K$ separates $e$ from $H_0$. So, in order to intersect $H_0$, the chain $\mathcal L$ must also intersect $\mathcal K$, which is a contradiction, since $\mathcal L$ and $\mathcal K$ are both green.

Now suppose $L \in \mathcal L \neq \mathcal K$ is dual to an edge $e$ of $\mu_i$.  
Since $\mu_i$ is only defined for $i>0$, it is on the same side of $H_0$ as $ \mathcal H\setminus \{H_0\}$, and consequently, the same holds for $e$. Therefore, we conclude as before that $L$ is green.  

Additionally, we conclude as before
that the hyperplane 
$K\in \mathcal K$ labeled $d_{i-1}$ does not intersect $H_0$.  Now consider the subchain of $\mathcal K'$ of $\mathcal K$ consisting of the hyperplanes dual to all but the last edge $e_i'$ of $\mu_i$.  
Since $K$ does not intersect $H_0$, it follows that $e$ is separated from $H_0$ by some hyperplane in $\mathcal K'$.  Thus, in order to intersect $H_0$, 
$\mathcal L$ must intersect $\mathcal K$, which is again a contradiction.
\end{proof}

The 2-colored polygon 
obtained above 
gives a 2-component loop in $\Theta$, as described in Observation~\ref{obs:intersecting_chains}.
A priori this loop may not be a 2-component cycle, i.e.,~it is possible that $c_i=c_j$ or $d_i=d_j$ for some $i,j$.   However, we now show that  it contains a cycle.  We will then be able to apply 
$\mathcal R_4$ to this cycle to make progress towards obtaining a contradiction to the normal form in Claim~\ref{claim:lambda-edge-right}.

\begin{claim} \label{claim:cycle}
Consider the 2-component loop in $\Theta$ visiting $c_0, d_0, \dots c_k, d_k, c_{k+1}=c_0$ defined above.  There exist $0 \le l \le k-1$ and $m\ge 2$, 
such that one of the two following subsequences of vertices   (with indices taken mod $k+1$)  defines a 2-component cycle in $\Theta$:
\begin{enumerate}
\item 
$c_l, d_l, c_{l+1},  \dots, d_{l+m-1}, c_{l+m} = c_l$
\item  $d_l, c_{l+1}, d_{l+1}\dots, c_{l+m}, d_{l+m} = d_l$
\end{enumerate}

\end{claim}

\begin{proof}
Observe that since $c_{k+1} = c_0$, the following set is non-empty:
$$
 \{ j \;|\; c_i = c_{i+j}  \text{ or } d_i = d_{i+j} \text{ for some } 0 \le i< k-1 \text{ and } 1 \le j\le k+1\}
$$
Let $m$ denote its minimum value.  We first show that $m \ge 2$, or equivalently that, for each $0 \le i \le k-1$, 
both $c_i \neq c_{i+1}$ and  $d_i \neq d_{i+1}$ are true.  
Suppose $c_i = c_{i+1}$ for some $i$. 
Consider the path $\eta_{i+1}$ from Claim~\ref{clm:dual_hyperplanes}. 
It is labeled by $\L$-edges, and every edge in it is dual to a hyperplane from $\mathcal H$.  
Then by Claim~\ref{claim:simple_path}, it follows that $\eta_{i+1}$ defines a simple $\L$-path from 
the vertex $c_i$ to the vertex $c_{i+1}$. 
However, this  contradicts the assumption that $c_i=c_{i+1}$. This proves that for all $0 \le i \le k-1$, we have $c_i \neq c_{i+1}$.  The proof that  $d_i \neq d_{i+1}$ is similar. 

Now if  $l$ is such that $c_l =c_{l+m}$ (the case when $d_l = d_{l+m}$ is similar), then 
it readily follows from the minimality of $m$ that the vertices  $c_l, d_l, c_{l+1},$ $\dots c_{l+m-1},   d_{l+m-1}$ are distinct, 
and therefore define the desired cycle.  
\end{proof}

Continuing the proof of the theorem, we can now 
 assume $\Theta$ has a 2-component cycle $\gamma$ as in (1) from Claim~\ref{claim:cycle}.  (The case in which $\Theta$ has a  2-component cycle as in (2) is similar.)
Let $T_c$ and $T_d$ be the $\L$-convex hulls of $\{c_l,\dots, c_{l+m-1}\}$ and $\{d_l, \dots d_{l+m-1} \}$ respectively.  
Then  $T_c$ and $T_d$ are trees by $\mathcal R_1$.
Let $c_j$
 be a leaf of $T_c$ with $c_j\neq c_0$.  
 Then $c_j$ labels a hyperplane $H_t \in \mathcal H$ for some $t \neq 0$, so that $a_t=c_j$.  Similarly, 
 $c_{j-1}$ labels a hyperplane $H_{t-r}$ of $\mathcal H$, while 
 $d_{j-1}$ and $d_j$ label hyperplanes $K_{t'}$ and $K_{t'+r'}$ respectively of $\mathcal K$, where 
 %.  Thus, 
 $d_{j-1}=b_{t'}$ and~$d_j = b_{t'+r'}$

 Consider the paths $\eta_j$ and  $\mu_j$
 defined in Claim~\ref{clm:dual_hyperplanes}.  The last $\L$-edge of $\eta_j$ is $a_{t-1}a_t$.  
By Claim~\ref{clm:dual_hyperplanes}, the first edge of $\mu_j$ is dual to a hyperplane in $\mathcal K$.  It follows that this must be $K_{t'}$, with label $b_{t'}$, for otherwise $K_{t'}$ would separate this edge from $\mathcal K\setminus K_{t'}$.  It follows that the first $\L$-edge of 
	$\mu_j$
%$\mu_i$
 is 
$b_{t'}b_{t'+1}$, and that the word $w$ has a subword $a_{t-1}a_tb_{t'} b_{t'+1}$.

To complete the proof, we will show that the presence of this subword violates the 
 normal form established in Claim~\ref{claim:lambda-edge-right}(3).  Since the labels of $\mathcal H$ and $\mathcal K$ are from different components of $\L$, it is immediate that 
$a_{t-1}a_t \neq b_{t'} b_{t'+1}$.   We now show that $a_{t-1}a_t$ and $b_{t'}b_{t'+1}$ commute.

The 2-component cycle $\gamma$ in $\Theta$ contains an edge with endpoints $a_t$ and $b_{t'}$.  Applying 
 $\mathcal R_4$ to this edge, we conclude that there is a 2-component square visiting $a_t, b_{t'}, a,$ and $b$, where $a\in T_c$ and $b \in T_d$.  Next, applying  $\mathcal R_3$ to this 2-component square, we see that $b_{t'}$ commutes with the vertices of the $\L$-convex hull of $\{a_t, a\}$. 
Claim~\ref{clm:dual_hyperplanes}  and Claim~\ref{claim:simple_path} together imply that the path $\eta_j$ induces a simple $\L$-path visiting vertices $a_{t-r}, a_{t-r+1}, \dots, a_t$. 
 Consequently, the vertices along this path, and in particular $a_{t-1}$, are in $T_c$. 
Moreover,  $a_{t-1}$ is the unique vertex of 
$T_c$ adjacent to $a_t$, since $a_t = c_j$ is a leaf of $T_c$. 
It follows that $a_{t-1}$ is contained in the $\L$-convex hull (which is the same as the $T_c$-convex hull) of $\{a_t, a\}$.
 Thus, $a_{t-1}$ and $b_{t'}$ commute.  The same reasoning, applied to the edge of $\gamma$ with endpoints $a_t$ and $b_{t+r'}$, implies that $a_{t-1}$ and $b_{t+r'}$ commute.

Using the $\G$-edges whose existence is implied by these two additional commutation relations, we obtain a 2-component square visiting $a_t, b_{t'}, a_{t-1}, b_{t'+r'}$.  Applying $\mathcal R_3$ to this square, we conclude that $a_t$ and $a_{t-1}$ commute with each vertex in the 
$\L$-convex hull of $\{b_{t'}, b_{t'+r'}\}$.  By Claim~\ref{claim:simple_path}, we see that the path $\mu_j$ from Claim~\ref{clm:dual_hyperplanes} defines a simple $\L$-path visiting  $b_{t'}, b_{t'+1}, \dots b_{t'+r'}$.  It follows that $b_{t'+1}$ is in the convex hull of $\{b_{t'}, b_{t'+r'}\}$, and 
consequently, $a_t$ and $a_{t-1}$ commute with $b_{t'+1}$.

Putting together the commutation relations established in the previous paragraphs, we conclude that $a_{t-1}a_t$ commutes with $b_{t'}b_{t'+1}$.
  This contradicts the fact that we have chosen $D$ so that it satisfies 
(3) of Claim~\ref{claim:lambda-edge-right}. 
\end{proof}

\subsection{Three or more $\L$-components}
In the case that $\L$ contains at most two components, Theorem~\ref{thm:2cpt} shows that $\mathcal{R}_1 - \mathcal{R}_4$ are necessary and sufficient conditions that guarantee $(G^\Theta, E(\L))$ is a RAAG system. In this subsection, we do not place any restriction on the number of components of $\L$. We give an additional necessary Condition $\mathcal{R}_5$ for $(G^\Theta, E(\L))$ to be a RAAG system, and Example~\ref{ex:R5 necessary} shows this condition is independent of conditions $\mathcal{R}_1 - \mathcal{R}_4$. The authors are aware that \textit{even more} conditions are required in order to generalize Theorem~\ref{thm:2cpt} to this setting. These extra conditions are not included here, as they are complicated and the authors do not believe to yet possess the complete list of the necessary and sufficient conditions for this generalization.

We further show in this subsection that if $\Theta$ contains certain subgraphs and $(G^\Theta, E(\L))$ is a RAAG system, then $\G$ must necessarily contain a triangle. These results are needed in the next section.

\begin{definition}[Condition $\mathcal{R}_5$] \label{def:R5}
	We say that $\Theta$ satisfies \textit{condition $\mathcal{R}_5$} if the following holds. Let $\L_a, \L_c$ and $\L_d$ be distinct components of $\L$. Suppose we have vertices $a, a' \in \L_a$, $c, c' \in \L_c$ and $d, d' \in \L_d$, such that  $\Theta$ contains a 2-component square visiting $c$, $d$, $c'$ and $d'$.
	Furthermore, suppose that $c$ and $c'$ are each adjacent to $a$ in $\G$ and that $d$ and $d'$ are each adjacent to $a'$ in $\G$. (See Figure~\ref{fig:R5}.) Let $T_a, T_c$ and $T_d$ be the $\L$-convex hulls of  $\{a, a'\}$, $\{c, c'\}$ and $\{d, d'\}$ respectively. Then given any $\L$-edge $xx'$ of $T_a$,  
	the graph $\G$ contains either the join of $\{x, x'\}$ with $V(T_c)$ or  the join of $\{x, x'\}$ with $V(T_d)$.
\end{definition}

\begin{figure}[h!]
	\begin{overpic}[scale=0.6]
		{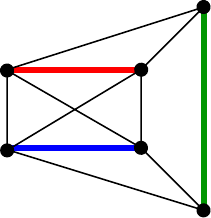}
		\put(-9, 69){\tiny $c$}
		\put(72, 69){\tiny $c'$}
		\put(-10, 29){\tiny $d$}
		\put(72, 29){\tiny $d'$}
		\put(99, 95){\tiny $a$}
		\put(99, 0){\tiny $a'$}	
	\end{overpic}
	
	\caption{This figure illustrates condition $\mathcal{R}_5$. The red, blue and green segments are respectively $T_c$, $T_d$ and $T_a$. Condition $\mathcal{R}_5$ states that any $\L$-edge contained in the green segment must either commute with every $\L$-edge in the red segment or must commute with every $\L$-edge in the blue segment.}
	\label{fig:R5}
\end{figure}
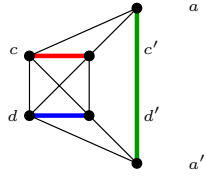

The following is a concrete example showing that when $\L$ has more than two components, 
the conditions $\mathcal{R}_1 - \mathcal{R}_4$ are not sufficient to guarantee that 
$(G^\Theta, E(\L) )$ is a RAAG system.

\begin{example} \label{ex:R5 necessary}
	Let $\G$ be the graph whose vertex set is $\{a, a', c, c', d, d'\}$ and whose edge set is the set of black edges in Figure~\ref{fig:R5}. Let $\Lambda \subset \G^c$ consist of exactly three $\L$-edges: $aa'$, $cc'$, $dd'$. Then $\Theta = \Theta(\G, \L)$ satisfies conditions $\mathcal{R}_1 - \mathcal{R}_4$ and does not satisfy condition $\mathcal{R}_5$. By Lemma~\ref{lem:R5} below, $(G^\Theta, E(\L) )$ is not a RAAG system.
\end{example}

We now show that condition $\mathcal{R}_5$ is necessary.

\begin{lemma}\label{lem:R5}
	If $(G^\Theta, E(\L))$ is a RAAG system, then $\Theta$ satisfies condition $\mathcal{R}_5$.
\end{lemma}
\begin{proof}
	By Theorem~\ref{thm:2cpt}, we may assume that $\Theta$ satisfies conditions $\mathcal{R}_1 - \mathcal{R}_4$. Let $a, a' \in \L_a$, $c, c' \in \L_c$ and $d, d' \in \L_d$ be as in Definition~\ref{def:R5}.
	Define the words $z_a = a'a$, $z_c = cc'$, $z_d= dd'$ and $z = [z_az_cz_a^{-1}, z_d]$. 
	By the commuting relations imposed in Definition~\ref{def:R5}, it follows that $z \gequal 1$ in $W_\G$.
	Let $w_a$, $w_c$, $w_d$ and $w$ be the $\L$-edge words corresponding respectively to $z_a$, $z_c$, $z_d$ and $z$. Let $D$ be a disk diagram over $A_\D$ with boundary label $w$.
	
	Let $\gamma_c$, $\zeta_c$, $\gamma_d$ and $\zeta_d$ be the paths in $\partial D$ labeled respectively by $w_c$, $w_c^{-1}$, $w_d$ and $w_d^{-1}$. Note that no hyperplane is dual to two distinct edges of $\gamma_c$ (resp. $\zeta_c$, $\gamma_d$ and $\zeta_d$). This follows as $z_c$ is a word in unique $\L$-edges.
	Thus, every hyperplane dual to $\gamma_c$ (resp. $\gamma_d$) is also dual to $\zeta_c$ (resp. $\gamma_d$).
	
	Let $\alpha$ be a path in $\partial D$ between $\gamma_c$ and $\gamma_d$ (which is labeled by $w_a$). Again, no hyperplane is dual to two distinct edges of $\alpha$. Let $xx'$ be a $\L$-edge of $T_a$, and let $H$ be the unique hyperplane dual to $\alpha$ with label $xx'$. Note that either $H$ intersects every hyperplane dual to $\gamma_c$ or $H$ intersects every hyperplane dual to $\gamma_d$. Furthermore, every $\L$-edge of $T_c$ (resp. $T_d$) is the label of a hyperplane dual to $\gamma_c$ (resp. $\gamma_d$). 
	The claim now follows from Lemma~\ref{lem:commuting}, and the fact that intersecting  hyperplanes correspond to commuting generators of $A_\Delta$.
\end{proof}

The following corollary shows that if $\Theta$ contains a configuration like that in the hypothesis of condition $\mathcal{R}_5$, then $\G$ must contain a triangle. This corollary is a warm-up to the more complicated Lemma~\ref{lem:triangle}.

\begin{cor} \label{cor:triangle}
	Suppose $(G^\Theta, E(\L))$ is a RAAG system and $\Theta$ contains a set of vertices $\{a, a', b, b', c, c'\}$ satisfying the hypothesis of  $\mathcal{R}_5$. Then $\G$ contains a triangle.
\end{cor}
\begin{proof}
	Let $P = \{a, a', c, c', d, d'\}$ be a subset of vertices of $\Theta$ satisfying the hypothesis of $\mathcal{R}_5$. We call such a $P$ a \emph{configuration} in $\Theta$. 
	Keeping the same notation as in Definition~\ref{def:R5},  we call the number of vertices of $T_a$ the \textit{complexity} of $P$, and we prove the claim by induction on complexity.
	Note that $a = a'$ is possible in the hypothesis of $\mathcal R_5$, so the lowest possible complexity is $N=1$. 
	The corollary follows in this case, 
	as $\G$ then contains a triangle spanned by the vertices $a = a'$, $c$ and $d$. 
	
	Now let $N>1$ and suppose the claim is true for all configurations $P$ of smaller complexity. As $N>1$,
	 there is a vertex $y$ such that $ay$ is a $\L$-edge of $T_a$.
	By Lemma~\ref{lem:R5}, either $y$ is adjacent in $\G$ to both $c$ and $c'$, or $y$ is adjacent in $\G$ to both $d$ and $d'$. In either case, we see that $\Theta$ contains a configuration of smaller complexity.
\end{proof}

The next lemma shows that if $\Theta$ contains certain subgraphs which generalize the configurations in the hypothesis of $\mathcal R_5$, then $\G$ must contain a triangle.

\begin{lemma}\label{lem:triangle}
	Let $\L_a$, $\L_c$ and $\L_d$ be distinct components of $\L$. 
	Suppose  $\Theta$ 
	has a $\L_a\L_c$-path visiting $c_1, a_1, c_2, \dots, a_{n-1}, c_n$, and a $\L_a\L_d$-path visiting $d_1, a_1', d_2, \dots,$ $a_{m-1}', d_m$,
	where $c_i \in \L_c$, $d_i\in \L_d$, and $a_i, a_i'\in \L_a$ for all appropriate $i$.  Further suppose that $\Theta$ contains a 2-component square visiting $c_1, d_1, c_n$ and $d_n$.
	(See Figure~\ref{fig:C6}).  If $(G^\Theta, E(\L))$ is a RAAG system, then $\G$ has a triangle.
\end{lemma}

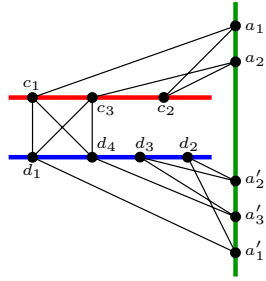
\begin{figure}[h!]
	\begin{overpic}[scale=0.6]
		{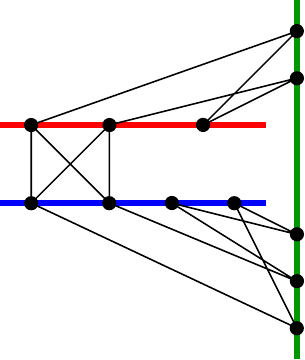}
		\put(5,69){\tiny $c_1$}
		\put(86,90){\tiny $a_1$}
		\put(86,78){\tiny $a_2$}
		\put(54,60){\tiny $c_2$}
		\put(32,60){\tiny $c_3$}
		\put(5,37){\tiny $d_1$}
		\put(86,8){\tiny $a'_1$}
		\put(62,47){\tiny $d_2$}
		\put(86,34){\tiny $a'_2$}
		\put(46,47){\tiny $d_3$}
		\put(86,21){\tiny $a'_3$}
		\put(32,47){\tiny $d_4$}
	\end{overpic}
	\caption{This figure illustrates the configuration described in Lemma~\ref{lem:triangle} in the case $n=3$ and $m=4$. The black edges are edges of $\G$.  The red, green and blue 
		parts consist of $\L$-edges, and are all contained in $\L$.  The different colors indicate that they are in three distinct components of $\L$.}
	\label{fig:C6}
\end{figure}

\begin{proof}  
	By Theorem~\ref{thm:2cpt} and Lemma~\ref{lem:R5}, we may assume that $\Theta$ satisfies conditions $\mathcal{R}_1 - \mathcal{R}_5$. 
	Let $A = \{a_1, \dots, a_{n-1}, a_1', \dots, a_{m-1}'\}$, $C =  \{c_1, \dots, c_n\}$ and $D = \{d_1, \dots, d_m\}$ be vertices of $\Theta$ as in the statement of the lemma. We call such a triple $(A, B, C)$ a configuration of $\Theta$. 
	Let $T_a$, $T_c$ and $T_d$ be the $\L$-convex hulls of $A$, $C$ and $D$ respectively. 
	We define the complexity of $(A, C, D)$ to be the integer $N= |C| + |D| + |T_a|_E + |T_c|_E+|T_d|_E$, where $|X|_E$ denotes the number of edges in a graph $X$. 
	The proof will be by induction on complexity of configurations.
	
	By hypothesis, we have that $n, m \ge 2$ and $|T_c|_E, |T_d|_E \ge 1$.
	If $n = m = 2$ then $\G$ contains a triangle by Corollary~\ref{cor:triangle}. In particular, the base case follows.

	We now fix a configuration $A = \{a_1, \dots, a_{n-1}, a_1', \dots, a_{m-1}'\}$, $C =  \{c_1, \dots, c_n\}$ and $D = \{d_1, \dots, d_m\}$ as above of complexity $N$, and we assume 
	that the result holds for configurations of smaller complexity. By the previous paragraph, we may also assume (up to relabeling) that $n >2$. We prove the lemma by showing that either $\G$ contains a triangle or $\Theta$ contains a configuration of smaller complexity.
	
	Define $\alpha_{ac}$ and $\alpha_{ad}$ to respectively be the hypothesized $\L_a\L_c$-path and $\L_a \L_d$-path.  
	We may assume that $\alpha_{ac}$ and $\alpha_{ad}$ are simple paths, for if not, we would be able to excise a loop to obtain a configuration of smaller complexity. 
	
	We claim that for all $1 < i < n$, we may assume that $c_i$ does not lies on the simple path 
	in $\L_c$ from $c_1$ to $c_n$.  For suppose there exists such a vertex $c_i$. By  $\mathcal{R}_3$, it follows that $c_i$ commutes with both $d_1$ and $d_m$.  There then exists a $\L_a\L_c$ path visiting $c_1, a_1, \dots, a_{i-1}, c_i$, and it follows that $\Theta$ contains a configuration of smaller complexity (obtained by replacing $\alpha_{ac}$ with this new path).
	Thus  we  may make this assumption without loss of generality. Furthermore, as $n \ge 3$, there exists an integer $j$ such that $c_j$ is a leaf vertex of $T_c$ and such that $1 < j < m$. We fix such a vertex $c_j$.
	
	Define the word $z_c$ to be:
	\[
		z_c=(a_1'a_1)(c_1c_2)(a_1a_2)(c_2c_3)(a_2a_3)\cdots (a_{n-2}a_{n-1})(c_{n-1}c_{n})(a_{n-1}a_1')
	\]
	and define the word $z_d$, depending on the value of $m$, to be:
	\[
		\begin{cases} 
			z_d = d_1d_2 & \text{ if } m = 2 \\
			z_d=(d_1d_2)(a'_1a'_2)(d_2d_3)(a'_2a'_3)\cdots (a'_{m-2}a'_{m-1})(d_{m-1}d_{m})(a_{m-1}'a_1') & \text{ if } m > 2
		\end{cases}
	\]
	In $W_\G$ we have that $z_c \gequal a_1' c_1c_{n}a_1'$ and $z_d \gequal a_1' d_1d_{m}a_1'$. 
	Let $z=[z_c, z_d]$. Note that as $c_1$ and $c_{n}$ commute with $d_1$ and $d_{m}$ in $W_\G$ we have:
	\[
	z \gequal [a_1' c_1c_{n}a_1', a_1' d_1d_{m}a_1' ] \gequal a_1'[ c_1c_{n},  d_1d_{m}]a_1' \gequal 1
	\]
	Let $w_c, w_d$ and $w$ be the $\L$-edge words associated to $z_c, z_d$ and $z$ respectively. Let $D$ be a disk diagram over $A_\D$ with boundary label $w$. 	Let $\gamma_c$, $\zeta_c$, $\gamma_d$ and $\zeta_d$ be the subpaths of $\partial D$ labeled respectively by $w_c$, $w_c^{-1}$, $w_d$ and $w_d^{-1}$.  
	
	Let $yc_j$ be the $\L$-edge of $T_c$ incident to $c_j$.  Since $\alpha_{ac}$ does not repeat vertices and since $c_j$ is a leaf of $T_c$, it follows that $w_c$ contains exactly two occurrences of the letter $y$ contained in the subword labeled by $(yc_j )(a_{j-1}x_1) (x_1 x_2) \cdots (x_la_{j}) (c_j y)$, where the $x_i$'s are vertices in $\L_a$. In particular, there are exactly four edges of $\partial D$ (two on $\gamma_c$ and two on $\zeta_c$) labeled by either $yc_j$ or $c_jy$. Correspondingly, there are exactly two hyperplanes, $H$ and $H'$ in $D$ labeled $yc_j$.  
	
	We  claim that we  may assume that $H$ is dual to both $\gamma_c$ and $\zeta_c$, and the same is true for $H'$. For suppose otherwise, and suppose that $H$ is dual to two edges of $\gamma_c$. 
	(The case of $H'$ is similar.)
	It follows that any hyperplane dual to the subpath of $\gamma_c$ labeled by
	$(a_{j-1}x_1)(x_1 x_2) \dots (x_la_{j})$ (which lies between the endpoints of $H$) must intersect $H$. Thus, in particular, $(a_{j-1}x_1)$ and $(x_la_{j})$ commute with $yc_j$, and applying Lemma~\ref{lem:commuting}, we conclude that $y$ commutes with both $a_{j-1}$ and $a_{j}$.  
	We now show that we can replace $\alpha_{ac}$ with a new $\L_a\L_c$-path from $c_1$ to $c_n$ such that $|T_a|_E$ is reduced, and thus $\Theta$ contains a smaller complexity configuration. 
	If $y$ is not equal to any $c_k$ for any $1 \le k \le m$, then we obtain this path by simply replacing $c_j$ with $y$ in $\alpha_{ac}$.  On the other hand,  if $y = c_k$ for some $k$, then we replace $\alpha_{ac}$ with the $\L_a \L_c$ path visiting $c_1, a_1, \dots, c_k, a_j, c_{j+1}, a_{j+1} \dots a_{n-1}, c_n$ if $k< j$ and perform a similar replacement if $k >j$.  
	In either case, we have produced a configuration of smaller complexity. Thus, we now assume that each of $H$ and $H'$ is dual to both $\gamma_c$ and~$\zeta_c$.
	
	Let $Q$ and $Q'$ be the hyperplanes in $D$ dual to the edges of $\gamma_c$ labeled  by $a_{j-1}x_1$ and $x_l a_j$ respectively. If both $Q$ and $Q'$ intersect $H \cup H'$, then we can conclude, as above, that $y$ commutes with both $a_{j-1}$ and $a_j$. We can then find a smaller complexity configuration as in the previous paragraph. Thus, we can assume that either $Q$ or $Q'$ is dual to both $\gamma_c$ and $\zeta_c$. We assume that $Q$ has this property (the case of $Q'$ is similar). 
	
	We now examine hyperplanes dual to $\gamma_d$ and $\zeta_d$. If $m=2$, then the unique hyperplane whose label contains $d_1$ is dual to both $\gamma_d$ and $\zeta_d$, and this hyperplane intersects both $H$ and $Q$. Thus, $d_1$ commutes with both $c_j$ and $a_{j-1}$. Since $c_j$ commutes with $a_{j-1}$, it follows that $\G$ contains a triangle. On the other hand, if $m>2$ by the same reasoning as before, we can assume there is  a leaf vertex $d_{j'}$ of $T_d$ and a hyperplane with label $y'd_{j'}$ that intersects both $\gamma_d$ and $\zeta_d$.  This then implies that $d_{j'}$ commutes with both $c_j$ and $a_{j-1}$ and consequently, $\G$ contains a triangle. The lemma now follows.
\end{proof}

\section{Finite index visual RAAGs}

As in the previous section, given  a simplicial graph $\G$ and a subgraph $\L$ of $\G^c$ with no isolated vertices, we set 
 $\Theta=\Theta(\G, \L)$, and let $G^\Theta$ be the subgroup generated by $E(\L)$.  
Our goal is to characterize graphs $\L \subset \G^c$ such that $(G^\Theta, E(\L))$ is a RAAG system and $G^\Theta$ has finite index in $W_\G$.

Suppose the graph $\G$ contains a vertex $s$ which is $\G$-adjacent to every other vertex of $\Gamma$. We say that $s$ is a \textit{cone vertex}. In this case, it easily follows that $W_{\G \setminus s}$ has index 2 in $W_\G$ and that $s$ cannot be contained in any $\L$-edge. 

We now recall a construction from~\cite{DL} which will help us compute the index of $G^\Theta$. 
	The construction is general, but for simplicity, and as it is all that we %need, 
	use,
	we choose to only describe it in the context where $\G$ is triangle-free.
 We refer the reader to~\cite{DL} for full~details. 

Let $\G$ be a triangle-free graph. 
We say a cell complex is $\G$-labeled if every edge of the complex is labeled by a vertex of $\G$.
Let $X$ be a $\G$-labeled complex. Suppose two edges of $X$ have the same label and a common endpoint. A \textit{fold operation} produces a new complex from $X$ by naturally identifying these two edges. 

Suppose now that $f_1$ and $f_2$ are edges of $X$ which share a common vertex $u$ and whose labels $s_1, s_2 \in V(\G)$ have an edge between them in $\G$. Let $c$ be a $2$-cube with edges $c_1, c_2, c_3$ and $c_4$ such that $c_i \cap c_{i+1}$ is a vertex of $c$ for each $i \mod 4$. We label $c_1$ and $c_3$ by $s_1$, and $c_2$ and $c_4$ by $s_2$. A \textit{square attachment operation} produces a new complex from $X$ by attaching $c$ to $X$ by identifying $c_1$ to $f_1$ and $c_2$ to $f_2$. Note that, unlike in~\cite{DL}, we do not need to define cube attachments for higher dimensional cubes, as we are in the case that $\G$ is triangle-free. 

Finally, given a collection of $2$-cubes in $X$ with common boundary, we can produce a new complex from $X$ by naturally identifying every $2$-cube in this collection to a single 
2-cube.
In this case, we say a \textit{cube identification operation} was performed to $X$.

We define a $\G$-labeled complex $\Omega_0$ associated to $G^\Theta$ as follows. First, we enumerate the $\L$-edges as
 $s_1t_1, \dots, s_nt_n$, where $s_i$ and $t_i$ are the two endpoints of the $i$th $\L$-edge. 
We set $\Omega_0$ to be a bouquet of $n$ circles, each of which is subdivided into two edges, such that the $i$th circle has label $s_it_i$.

Next, we describe a series of complexes built iteratively from $\Omega_0$. These are
\[\Omega_0 \to \Omega_1 \to \Omega_2 \to \dots \]
For each $i > 0$, the complex $\Omega_i$ is obtained by either a fold, square attachment or square identification operation performed to $\Omega_{i-1}$. 
Furthermore, we assume that the order of operations is as follows: first all possible fold and square identifications are performed, then all possible square attachment operations are applied to the resulting complex, and these processes are alternated (see~\cite{DL} for details). 

 Let $\Omega$ be the direct limit of such a sequence. We call $\Omega$ a \textit{completion} of $G^\Theta$. In~\cite{DL} we show that properties of $\Omega$ reflect those of the subgroup $G^\Theta$. 

The index of $G^\Theta$ can be determined by properties of $\Omega$.
We say that a vertex $u$ of a $\G$-labeled complex has \textit{full valence} if for any vertex $s \in \G$, there is an edge incident to $u$ with label $s$.  
Below we present a version of~\cite[Theorem~6.9]{DL} together with~\cite[Lemma~6.8]{DL} under the hypotheses which we will need:

\begin{thm} \label{thm_omega_characterization}
	Let $\G$ be a triangle-free graph with no cone vertex. 
	A subgroup $G < W_\G$ has finite index in $W_\G$ if and only if $\Omega$ is finite and every vertex of $\Omega$ has full valence. Furthermore, if $G$ is indeed of finite index, then its index is exactly the number of vertices of $\Omega$.
\end{thm}

We introduce two new properties below which will help us characterize when $G^\Theta$ has finite index in $W_\G$.
\begin{definition}[Conditions $\mathcal F_1$ and $\mathcal F_2$] \label{def_finite_index_conditions}
    We say that	$\Theta=\Theta(\G, \L)$ satisfies \textit{condition $\mathcal F_1$} if given any 
    $s \in V(\Theta)$ 
which is not a cone vertex of $\G$, it follows that
     $s$ is 
     the endpoint of 
    some $\L$-edge.
	We say that $\Theta$ satisfies \textit{condition $\mathcal F_2$} if given any distinct components $\L_s, \L_t$ of $\L$, and vertices $s $ of $ \L_s$ and $t $ of $ \L_t$, there is a $\L_s\L_t$-path in $\Theta$ 	 from $s$ to $t$.
\end{definition}

\begin{remark} \label{rmk_f2_simplification}
	Suppose $\G$ is connected, $\L$ contains exactly two components and that $\Theta = \Theta(\G, \L)$ satisfies $\mathcal{R}_2$ and $\mathcal{F}_1$. Then $\Theta$ satisfies $\mathcal{F}_2$. For given any two vertices contained in different components of $\L$, as $\G$ is connected, there is a $\G$-path between them. Furthermore, this has to be a $2$-component path as $\Theta$ satisfies $\mathcal{R}_2$, and the two $\L$-components this path visits have to be the ones containing the chosen vertices (as there are only two $\L$ components). 
	This remark will prove to be useful when verifying whether certain graphs satisfy $\mathcal{F}_2$.
\end{remark}

\begin{remark} \label{rmk:f2_gives_2_component_paths}
	Suppose $\Theta=\Theta(\G, \L)$ satisfies $\mathcal F_2$, and let $\L_1$ and $\L_2$ be distinct $\L$-components. Then there exists an $\L_1 \L_2$-path between any two distinct vertices of $\L_1$. To see this, let $s$ and $s'$ be distinct vertices of $\L_1$, and let $t$ be vertex of $\L_2$.  By $\mathcal F_2$ there is a $\L_1 \L_2$-path from $s$ to $t$, and similarly there is an $\L_1 \L_2$-path from $t$ to $s'$. Combining these two paths gives an $\L_1 \L_2$-path form $s$ to $s'$.
\end{remark}

\begin{lemma} \label{lemma_finite_index1}
Let $\G$ be a triangle-free graph with no cone vertex, and let $\L$ be a subgraph of $\G^c$
	with no isolated vertices,
 such that $(G^\Theta, E(\L))$ is a RAAG system. 
	If $\L$ has at most $k \le 2$ components and $\Theta$ satisfies $\mathcal F_1$ and $\mathcal F_2$, then $G^\Theta$ is of index $2^k$ in $W_\G$.
\end{lemma}

 We remark that this proof readily generalizes to the case of arbitrary $k$.  However, we only need the case $k \le 2$.

\begin{proof}
	Let $\Omega_0$ be the $\G$-labeled complex defined above, and let $\Omega'$ be the complex  obtained from $\Omega_0$ by all possible fold operations. 

	Suppose first that $\L$ has one component. 
	As $\L$ is connected, it is easily seen that $\Omega'$ consists of two vertices with an edge labeled by $s$ between them for $s \in V(\L)$.  
As $\L$ satisfies $\mathcal R_2$ by Proposition~\ref{prop_necessary}, 
	no two vertices of $\L$ have an edge between them in $\G$. Thus,  no square attachments can be performed to $\Omega'$, and it follows that $\Omega = \Omega'$. Hence, $\Omega$ is finite and has exactly two vertices.

	Note that by the description of $\Omega = \Omega'$ above, every vertex of $\Omega$ is adjacent to every edge of $\Omega$. Also note that by condition $\mathcal F_1$, for every vertex $s \in \G$ there is some edge in $\Omega$ labeled by $s$. From these two facts we deduce that every vertex of $\Omega$ has full valence. Thus, $G^\Theta$ has index $2$ in $W_\G$ by Theorem~\ref{thm_omega_characterization}.

	Now suppose that $\L$ has two components $\L_1$ and $\L_2$.  In this case, $\Omega'$ is readily seen to be a complex consisting of three vertices, $u, v_1, v_2$, with an edge from $u$ to $v_i$ labeled $s$ corresponding to each vertex $s$ of $\L_i$,  for $i=1,2.$ By condition $\mathcal F_1$, the vertex  $u$ has full valence. Furthermore, by $\mathcal R_2$, for each $i \in \{1, 2\}$, no two edges of $\Omega'$ that are each adjacent to both $v_i$ and $u$ 
	 have labels which are adjacent in $\G$.

	Let $\Omega''$ be the complex obtained from $\Omega'$ by performing all possible square attachment operations to $\Omega'$, and let $\Omega'''$ be the complex obtained from $\Omega''$ by all possible fold and square identification operations.
	In particular, 
	$\Omega'' = \Omega_l$ and $\Omega''' = \Omega_k$ for some $0 \le l \le k$. 
	Let $s,s'$ be distinct vertices of $\L_1$, and let $t$ be any vertex of $\L_2$. By condition $\mathcal F_2$, there is a $\L_1\L_2$-path whose vertices are $s, t_1, s_1, t_2, s_2, \dots, t_m, s_m, t$ where $s_i \in \L_1$ and $t_i \in \L_2$ for all $1 \le i \le m$. Similarly, there is a $\L_1\L_2$-path whose vertices are $s', t'_1, s'_1, t'_2, s'_2, \dots, t'_n, s'_n, t$ where $s'_i \in \L_1$ and $t'_i \in \L_2$ for all $1 \le i \le n$. Thus, $\Omega''$ must contain length two paths,
	which do not intersect $u$,
	 from $v_1$ to $v_2$ with each of the following labels 
	\[t_1s, t_1s_1, t_2s_1, t_2s_2, \dots, t_ms_{m-1}, t_ms_m, ts_m, \]
	and similarly length two paths,
 which do not intersect $u$,
	 from $v_1$ to $v_2$ with each of the following labels 
	\[t'_1s', t'_1s'_1, t'_2s'_1, t'_2s'_2, \dots, t'_ns'_{m-1}, t'_ms'_m, t's_m\]
	It follows that the middle vertices of all these paths get folded to a single vertex $v_3$ in $\Omega'''$. This analysis can be done for any $s, s' \in \L_1$. Similar paths can also be produced for any $t, t' \in \L_2$. It then follows that $\Omega'''$ consists of exactly 4 vertices: $u, v_1, v_2$ and $v_3$. Furthermore, there is an edge with label $s$ between $v_1$ and $v_3$ for each $s \in \L_1$, and there is an edge with label $t$ between $v_2$ and $v_3$ for each $t \in \L_2$. Thus, every vertex of $\Omega'''$ can be seen to have full valence. Additionally, by condition $\mathcal R_2$, no additional square attachment operations can be performed to~$\Omega'''$. Hence, $\Omega = \Omega'''$. It follows that $G^\Theta$ has index exactly four in $W_\G$.
\end{proof}

The next lemma shows that $\mathcal{F}_1$ and $\mathcal{F}_2$ are necessary conditions for $G^\Theta$ to have finite index.

\begin{lemma} \label{lemma_finite_index2}
	Let $\G$ be a triangle-free graph with no cone vertex, and let $\L$ be a subgraph of $\G^c$ 
	with no isolated vertices,
	 such that $(G^\Theta, E(\L))$ is a RAAG system. 
	If $G = G^\Theta$ is of finite index in $W_\G$, then $\Theta$ satisfies $\mathcal F_1$ and $\mathcal F_2$.
\end{lemma}
\begin{proof}
	We first check that condition $\mathcal F_1$ holds. Let $\Omega$ be a completion of $G:=G^\Theta$ as described in the beginning of this section. Theorem~\ref{thm_omega_characterization} implies in particular that given any vertex $s \in \G$ there is an edge of $\Omega$ with label $s$. This implies the vertex~$s$ is contained in some $\L$-edge. Thus, $\mathcal F_1$ must hold.
	
	We now check condition $\mathcal F_2$. Let $s \in \L_s$ and $t \in \L_t$ be as in the definition of condition $\mathcal F_2$ (Definition~\ref{def_finite_index_conditions}). If $s$ commutes with $t$, then there is an edge in $\G$ between $s$ and $t$, and we are done. So we may assume that $s$ and $t$ do not commute. 
	
	As $G$ is of finite index, it follows that there exist $g_1, \dots, g_n \in W_\G$ such that $W_{\Gamma} = Gg_1 \sqcup Gg_2 ... \sqcup Gg_n$. Let $w_1, \dots, w_n$ be reduced words representing $g_1, \dots, g_n$, and let $K = \max\{|w_1|, ..., |w_n|\}$. Define the word $h = s_1t_1s_2t_2...s_{K+4}t_{K+4}$ where $s_i = s$ and $t_i = t$ for all $1 \le i \le K+4$. It readily follows from Tits' solution to the word problem (see Theorem~\ref{thm_tits_solution})
	that $h$ is reduced.
	Furthermore, we can write $h \gequal ww'$, where
	%there exist words 
	$w$ and $w'$ are words in $W_\G$ such that $w' = w_i$ for some $1 \le i \le n$ and $w$ is a product of $\L$-edges representing an element of $G$.
	 We can form a disk diagram in $W_\G$ with boundary label $h w'^{-1} w^{-1}$. Let $\alpha_h, \alpha_w$ and $\alpha_{w'}$ respectively be the corresponding paths along the boundary of $D$ with labels respectively $h, w$ and $w^{-1}$.
	
	Note that as $h$ is reduced, no hyperplane intersects $\alpha_h$ twice. Also note that any pair of hyperplanes emanating from $\alpha_h$ cannot intersect as $s$ and $t$ do not commute. As $|h| > |w'| + 4$, it follows that the hyperplanes $H_{s_1}, H_{t_1}, H_{s_2}$ and $H_{t_2}$, dual respectively to the first four edges of $\alpha_h$ (namely those labeled by $s_1, t_1$, $s_2$ and $t_2$), must each intersect $\alpha_w$. It must now be the case that there exists a chain of hyperplanes (see Definition~\ref{def:hyperplane-chain}) 
	$H_{s_1} = H_0, H_1, \dots H_m = H_{s_2}$ and another chain of  hyperplanes  $H_{t_1} = H_0',H_1' \dots, H_n' = H_{t_1}$.  These two chains intersect, and by reasoning similar to that in 
	Observation~\ref{obs:intersecting_chains}, it follows that there is a $\L_s\L_t$-path from $s$ to $t$.  
\end{proof}

\begin{lemma} \label{lemma_two_comps}
Let $\G$ be a triangle-free graph. Let $\L$ be a subgraph of $\G^c$
	with no isolated vertices,
	 such that $(G^\Theta, E(\L))$ is a RAAG system and $G^\Theta$ has finite index in $W_\G$.
	If $\G$ contains a cone vertex, then $\L$ contains exactly one component. If $W_\G$ is not virtually free, then $\L$ contains exactly two components. Otherwise, $\L$ contains at most two components.
\end{lemma}
\begin{proof}
	Suppose first that $\G$ contains a cone vertex $s \in \G$. 
	We may assume that $\G$ does not consist of a single edge, as $\L$ would be empty in that case.
	As $\G$ is triangle-free
	 in addition,
	there can be at most one cone vertex. 
	Since $\G$ is triangle-free, it follows that $\G' = \G \setminus s$ is a graph with no edges and is therefore virtually free. Furthermore, every $\L$-edge is contained in $\G'$, and $G^\Theta$ is a finite-index subgroup of $W_{\G'}$.
	By Lemma~\ref{lemma_finite_index2}, we conclude that $\Theta' = \Theta(\G', \L)$ satisfies condition $\mathcal F_2$. In particular, there is an $\G'$-edge between any two $\L$ components. As $\G'$ does not have any edges, $\L$ has exactly one component and the claim follows in this case. 
	
	We now assume that $\G$ does not contain a cone vertex. Furthermore, by Lemma~\ref{lemma_finite_index2} we may assume that $\Theta = \Theta(\G, \L)$ satisfies $\mathcal F_1$ and $\mathcal F_2$, and that $\Theta$ satisfies $\mathcal{R}_1 - \mathcal{R}_4$ by Proposition~\ref{prop_necessary}.

	Suppose now that no two distinct $\L$-edges commute. It follows that $G^\Theta$  is isomorphic to a free group, and since $G^\Theta$ is of finite index, $W_{\G}$ is virtually free. 
	Suppose, for a contradiction, that $\L$ has three distinct components $\L_1$, $\L_2$ and $\L_3$. Let $s$ and $t$ be distinct vertices of $\L_1$. By Remark~\ref{rmk:f2_gives_2_component_paths} there is an $\L_1\L_2$-path $\alpha_1$ from $s$ to $t$ which we can assume does not repeat vertices. Similarly, there is an $\L_1\L_3$-path $\alpha_2$ from $s$ to $t$ which does not repeat vertices. 
	Observe that $s, t \in \alpha_1 \cap \alpha_2$.  Starting at $s$ and traveling along $\alpha_1$, let $x$ be the first vertex after $s$ such that $x\in  \alpha_1 \cap \alpha_2$.  Then the subpath $\alpha_1'$ of $\alpha$ between $s$ and $x$ contains exactly two vertices of $\alpha_1 \cap \alpha_2$.  
	Let $\alpha_2'$ be the subpath of $\alpha_2$ between $s$ and $x$.  Note  
	that $|\alpha_1'|, |\alpha_2'| \ge 2$, as every other vertex of $\alpha_1$ is in $\L_2$ and $\alpha_2 \cap \L_2 = \emptyset$. It follows that $c = \alpha_1' \cup \alpha_2'$ is a cycle in $\Gamma$.
	Let $c'$ be a sub-cycle of $c$ which is an induced subgraph of $\G$. If $c'$ has three vertices, then this contradicts $\G$ being triangle-free. On the other hand, if $c'$ has more than $3$ vertices, then this contradicts $W_\G$ being virtually free. Thus, $\L$ can have at most two components and the claim follows in this case.

	Suppose now there exist $\L$-edges  $a_1a_2$ and $b_1b_2$ which commute, with  $a_1a_2 \neq (b_1b_2)^{\pm 1}$.  These $\L$-edges must be in different components of $\L$ by condition $\mathcal R_2$ and 
	Lemma~\ref{lem:commuting}. 
	In this case, $W_\G$ is not virtually free as it contains a subgroup isomorphic to $\mathbb{Z}^2$. 
	Suppose, for a contradiction, that $\L$ contains at least three distinct $\L$-edge components $\L_1, \L_2$ and $\L_3$. 
	Without loss of generality, we may assume that $a_1b_1 \in \L_1$ and that $a_2b_2 \in \L_2$.
	We will obtain a contradiction by showing that $\G$ must contain a triangle.

	By Lemma~\ref{lem:commuting}, $a_1, a_2, b_1, b_2$ form a square in $\Gamma$. By Remark~\ref{rmk:f2_gives_2_component_paths}, there is an $\L_1\L_3$-path from $a_1$ to $a_2$. Similarly, there is a $\L_2\L_3$-path from $b_1$ to $b_2$. 
 	Thus, $\G$ contains the configuration described in the statement of Lemma~\ref{lem:triangle}.  That lemma then implies that $\G$ contains a triangle, a contradiction.
\end{proof}

\begin{thm} \label{thm_fi_visual_raags}
	Let $W_\Gamma$ be a
	 $2$-dimensional RACG. Let $\L$ be a subgraph of $\G^c$ 
	with no isolated vertices,
	and let $G^\Theta$  be the subgroup of $W_\G$ generated by the $\L$-edges. 
	Then the following are equivalent.
	\begin{enumerate}
	\item $(G^\Theta, E(\L))$ is a RAAG system and $G^\Theta$ has finite index in $W_\G$.
	\item 
 $(G^\Theta, E(\L))$ is a RAAG system and $G^\Theta$ has index either two or four in $W_\G$ (and exactly four if $W_\G$ is not virtually free).
	\item  $\L$ has at most two 
	components and  $\Theta$ satisfies conditions $\mathcal R_1$--$\mathcal R_4$, $\mathcal F_1$ and~$\mathcal F_2$. 
	\end{enumerate}
\end{thm}
\begin{proof}
	 Clearly (2) implies (1).
	 
	 To see the remaining implications, suppose 
	  first that $\G$ contains a cone vertex~$s$. Then $\G' = \G \setminus s$ is a graph with no edges,
	 and $W_{\G'}$ is an index two subgroup of $\G$. Suppose that (1) holds. By Lemma~\ref{lemma_two_comps}, $\L$ has exactly one component.  By Theorem~\ref{thm:2cpt} and 
	 Lemma~\ref{lemma_finite_index2},  $\Theta' = \Theta(\G', \L)$ satisfies conditions $\mathcal R_1$--$\mathcal R_4$, $\mathcal F_1$ and $\mathcal F_2$. Consequently, $\Theta$ satisfies these conditions as well. Thus (3) holds. By Lemma~\ref{lemma_finite_index1}, we know that  	($G^{\Theta'}, E(\L))$ is a RAAG system of index 2 in $W_{\G'}$, and thus ($G^\Theta, E(\L))$ is a RAAG system of index four in $W_{\G}$. Therefore (2) holds. Finally,  if (3) holds then (1) holds by Theorem~\ref{thm:2cpt} and Lemma~\ref{lemma_finite_index1}.
	 
	 Now suppose that $\G$ does not have a cone vertex. If (1) holds, then by 
	 Lemma~\ref{lemma_two_comps}, $\L$ has exactly two components if $W_\G$ is not virtually free and at most two components otherwise. Thus (2) holds by  Lemma~\ref{lemma_finite_index1}.
	 By Theorem~\ref{thm:2cpt} and  Lemma~\ref{lemma_finite_index2}, (3) holds.
	 Finally if (3) holds, then (1) follows by Theorem~\ref{thm:2cpt} and  Lemma~\ref{lemma_finite_index1}.
\end{proof}

\begin{cor} \label{cor_bipartite}
	Let $W_\Gamma$ be a $2$-dimensional RACG. Let $\L$ be a subgraph of $\G^c$ 
	with no isolated vertices such that the subgroup $(G, E(\L))$  is a finite index RAAG system. Then either:
	\begin{enumerate}
		\item The graph $\G$ does not contain any edges and $E(\L)$ is a spanning tree in $\G^c$. In particular, $W_\G$ is virtually free.
	
		\item The group $W_{\G}$ is not virtually free. Furthermore, the vertices of $\Gamma$ can be $2$-colored by red and blue (i.e., each edge of $\G$ connects a red vertex and a blue vertex) and 
		 $G$ is isomorphic to the kernel of the homomorphism $\Psi: W_\G \to \mathbb{Z}_2 \times \mathbb{Z}_2 = \langle r , b ~|~ r^2 = b^2 = 1 \rangle$ which maps red and blue generators of $V(\Gamma)$ to $r$ and $b$ respectively.
	\end{enumerate}
\end{cor}
\begin{proof}
	By Theorem \ref{thm_fi_visual_raags}, $\L$ has at most two components. Suppose first that $\L$ contains exactly one component. Again by Theorem \ref{thm_fi_visual_raags}, the graph $\Theta$ satisfies $\mathcal{R}_1$, $\mathcal{R}_2$, and $\mathcal F_1$. From these conditions, it follows that $\G$ cannot contain any edges and that $E(\L)$ is a spanning tree in $\G^c$.
	As $\G$ does not contain any edges, $W_{\G}$ is virtually free. 
	
	Suppose now that $\L$ has exactly two components. We color the vertices of one component red and the vertices of the other component blue. By $\mathcal R_2$, each edge of $\G$ connects a red vertex and a blue vertex, i.e., we have a 2-coloring of $\G$.
	Furthermore, by the definition of $\Psi$, every $\L$-edge (thought of as an element of $G$) is in the kernel of $\Psi$. As $G$ is generated by such elements, it follows that $G < \text{ker}(\Psi)$. By Theorem \ref{thm_fi_visual_raags}, $G$ has index $4$ in $W_\G$. As $\text{ker}(\Psi)$ has index $4$ as well, it follows that $G$ is isomorphic to $\text{ker}(\Psi)$.
\end{proof}

\section{Applications}\label{sec:applicaitons}

In this section we give concrete families of RACGs containing  finite-index RAAG subgroups.
These cannot be obtained by applying the Davis--Januszkiewicz constructions to the defining graphs of the RAAGs they are commensurable to.

\subsection{Non-planar RACGs commensurable to RAAGs} 
In this subsection, we construct two families of RACGs with non-planar defining graphs containing finite-index RAAG subgroups.  These will serve as a warm-up for Theorem~\ref{thm_planar}.

We begin by constructing a family of quasi-isometrically distinct RACGs defined by the sequence of graphs $\G_n$
(shown in Figure~\ref{fig:application}) which are commensurable to RAAGs whose defining graphs are cycles.

\begin{cor} [to Theorem~\ref{thm_fi_visual_raags}]\label{cor:example}
For $n \ge 3$, let $\G_n$ be the graph obtained by starting with a $2n$-gon whose vertices (in cycle order) are $c_1, d_1, c_2, d_2,\dots, c_n, d_n$ and adding two vertices 
$x$ and $y$, such that $y$ is adjacent to $c_i$ for each $i$, $x$ is adjacent to $d_i$ for each $i$, 
and $x$ is adjacent to $y$ (see Figure~\ref{fig:application}).
Then 
\begin{enumerate}
\item The RACG $W_{\G_n}$ has %a finite-index 
a subgroup  of index four that is isomorphic to (and hence is commensurable to) the RAAG $A_{C_{2n}}$, where $C_{2n}$ is a cycle of length $2n$.  

\item $W_{\G_n}$ is not quasi-isometric to $W_{\G_m}$ for $m \neq n$. 
\end{enumerate}

\end{cor}

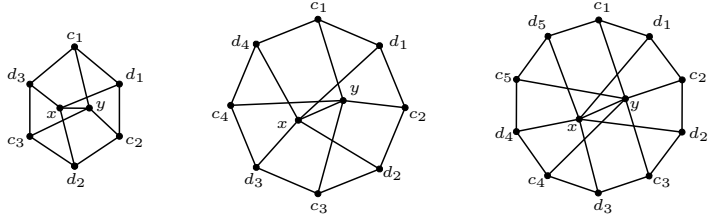
\begin{figure}[h!]
	\begin{overpic}[scale=0.75] 
		{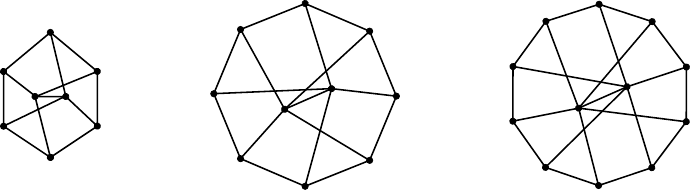}
		\put(6, 24){\tiny $c_1$}
		\put(43, 28.5){\tiny $c_1$}
		\put(86, 28.5){\tiny $c_1$}
		\put(15, 8){\tiny $c_2$}
		\put(58, 12){\tiny $c_2$}
		\put(100.5, 18){\tiny $c_2$}
		\put(-3, 8){\tiny $c_3$}
		\put(43, -2){\tiny $c_3$}
		\put(95, 2){\tiny $c_3$}
		\put(15, 18){\tiny $d_1$}
		\put(55, 23){\tiny $d_1$}
		\put(95, 26){\tiny $d_1$}
		\put(6, 2){\tiny $d_2$}
		\put(54, 3){\tiny $d_2$}
		\put(100.5, 9){\tiny $d_2$}
		\put(-3, 18){\tiny $d_3$}
		\put(33, 2){\tiny $d_3$}
		\put(86, -2){\tiny $d_3$}
		\put(28, 12){\tiny $c_4$}
		\put(76, 2){\tiny $c_4$}
		\put(31, 23){\tiny $d_4$}
		\put(71, 9){\tiny $d_4$}
		\put(71,18 ){\tiny $c_5$}
		\put(76, 26){\tiny $d_5$}
		\put(3, 12){\tiny $x$}
		\put(38, 10.5){\tiny $x$}
		\put(82, 10){\tiny $x$}
		\put(10.5, 13.5){\tiny $y$}
		\put(49, 16){\tiny $y$}
		\put(91.5, 13.5){\tiny $y$}
	\end{overpic}
	\caption{The figure illustrates the graphs $\G_n$ defined in  
		Corollary~\ref{cor:example} for $n = 3, 4, 5$.}
	\label{fig:application}
\end{figure}

\begin{proof}
	Fix $n \ge 3$, and  let $\G$ denote $\G_n$. We define a graph $\L \subset \G^c$ as follows.  Let $\L_x$ be the star graph consisting of the union of the edges of $\G^c$
	 from $x$ to $c_i$ for each $i$.  Let $\L_y$ be the star graph consisting of the edges of $\G^c$ from $y$ to $d_i$ for each $i$.  Let $\L = \L_x \cup \L_y$. 
	 (See Figure~\ref{fig:hexagon} for an illustration of $\L$ in the case $n=3$.)

	We show below that 
	$\Theta = \Theta(\G, \L)$
 satisfies $\mathcal R_1$--$\mathcal R_4$, $\mathcal F_1$ and $\mathcal F_2$.  
 Then it will  follow from Theorem~\ref{thm_fi_visual_raags}, that $(G^\Theta, E(\L)) $ is a RAAG system, and that $G^\Theta$ has index four in $W_\G$.  Moreover, it is easily checked that 
 the	commuting graph $\Delta$ associated to $\L$ (as defined in 
	Section~\ref{subsec_racgs_and_raags}) 
is isomorphic to $C_{2n}$.  Consequently, $G^\Theta$ is isomorphic to $A_{C_{2n}}$.  
Thus, this will show (1).

It is easy to verify $\mathcal F_1$, $\mathcal R_1$, and $\mathcal R_2$.  
Then by Remark~\ref{rmk_f2_simplification}, it follows that $\mathcal F_2$ holds as well.
We now check $\mathcal R_3$.  First note that there are exactly three squares 
 in $\G$ 
 containing the edge $c_1d_1$, and each of these satisfies the property in $\mathcal R_3$.  Now the fact that every square contains an edge of the 
$2n$-gon, together with the symmetry of the diagram, implies that $\mathcal R_3$ holds. 

To check $\mathcal R_4$, 
let $\gamma$ be a 
$\L_x \L_y$-cycle
 and let $e$ be an edge of $\gamma$.
By symmetry, 
we can assume that $e$ is either 
$c_1d_1, c_1y$ or $xy$.  
Suppose first that $e = c_1d_1$. Then $\gamma$ contains 
either 
$d_nc_1$
or $yc_1$.  In both cases, the 
 $\L$-convex hull of the vertices of $\gamma$ contains $y$.
Similarly, $\gamma$ contains either $d_1x$ or $d_1 c_2$, and in both cases 
the $\L$-convex hull of $\gamma$ contains $x$.
As 
$c_1d_1$ is
 contained in the square $c_1d_1 xy$, and $x,y$ are in the appropriate convex hulls, it follows that $\mathcal{R}_4$ holds for the $\G$-cycle $\gamma$ and edge $e$. 
 
Suppose now that $e = c_1y$. It follows that $\gamma$ contains either $yx$ or $yc_i$ for some $i >1$.  In each case, $x$ is in the $\L$-convex hull of $\gamma$.  Furthermore, $\gamma$ 
contains either $c _1d_1$ or $c_1d_n$.  In the former case, 
the square $c_1d_1xy$ contains $e$ and has vertices in the $\L$-convex hull of the vertices in $\gamma$. In the latter case the same argument applies to the square $c_1d_nxy$.

Finally, suppose that $e = xy$.  
By symmetry, we may assume that $\gamma$ contains 
$yc_1$.  Furthermore, $yc_1$ must be followed by either $c_1d_2$ or $c_1d_n$ in $\gamma$.  Then, as in the previous paragraph, either the square $c_1d_1xy$ or the square $c_1d_nxy$
contains $e$ and has vertices in the $\L$-convex hull of $\gamma$.    
Thus $\mathcal R_4$ is satisfied in all cases.

We have thus established that (1) holds, by showing that $\Theta$ satisfies $\mathcal R_1$--$\mathcal R_4$, $\mathcal F_1$ and $\mathcal F_2$. Consequently, for each $n$, we know that $W_{\G_n}$ is commensurable, and in particular quasi-isometric, to $A_{C_{2n}}$.  Claim (2) then follows from~\cite{BKS}. 
\end{proof}

	Next, we give a family of RACGs whose defining graphs are not planar and are commensurable to RAAGs which are not atomic (as defined in~\cite{BKS}).

	\begin{cor} \label{cor:non_planar_example}
		Given $n\ge 3$ and $k \ge 1$, let $\Delta_{nk}$ be the graph obtained by taking $k$ copies of $\G_n$ (defined in Figure \ref{fig:planar_examples}), and identifying them all along the subgraph induced by 
		$V(\G_n)\setminus \{a_0\}$.  Thus $\Delta_{nk}$ has vertices $a_1, a_2, \dots, a_n, b_1, \dots, b_n$ and also $a_{01}, \dots, a_{0k}$.  (The left side of Figure~\ref{fig:delta42} shows $\Delta_{42}$.)
		Then $W_{\Delta_{nk}}$ contains an index four subgroup isomorphic to a~RAAG.
	\end{cor}
	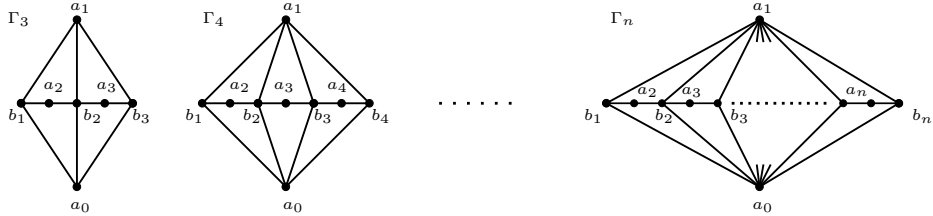
\begin{figure}[h!]
	\begin{overpic}[scale=1.4]%, grid, tics=5]
		{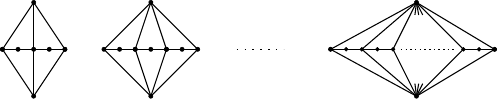}
		\put(-1,19){\tiny $\G_3$}
		\put(21,19){\tiny $\G_4$}
		\put(67,19){\tiny $\G_n$}
		\put(6,20.5){\tiny $a_1$}
		\put(30,20.5){\tiny $a_1$}
		\put(83,20.5){\tiny $a_1$}
		\put(6,-2){\tiny $a_0$}
		\put(30,-2){\tiny $a_0$}
		\put(83,-2){\tiny $a_0$}
		\put(-1,8){\tiny $b_1$}
		\put(19,8){\tiny $b_1$}
		\put(64,8){\tiny $b_1$}
		\put(7.5,8){\tiny $b_2$}
		\put(13,8){\tiny $b_3$}
		\put(72,8){\tiny $b_2$}
		\put(80.5,8){\tiny $b_3$}
		\put(101,8){\tiny $b_n$}
		\put(25.5,8){\tiny $b_2$}
		\put(33.5,8){\tiny $b_3$}
		\put(40,8){\tiny $b_4$}
		\put(3,11.5){\tiny $a_2$}
		\put(70,11){\tiny $a_2$}
		\put(24,11.5){\tiny $a_2$}
		\put(29,11.5){\tiny $a_3$}
		\put(35,11.5){\tiny $a_4$}
		\put(9,11.5){\tiny $a_3$}
		\put(75,11){\tiny $a_3$}
		\put(93.5,11){\tiny $a_n$}
	\end{overpic}
	\caption{The figure defines
	%illustrates 
	the family of graphs $\G_n$, for $n \ge 3$, used in Corollary~\ref{cor:non_planar_example}.}
	\label{fig:planar_examples}
\end{figure}

	\begin{proof}
		Fix $n \ge 3, k\ge 1$ and let $\Delta = \Delta_{nk}$.
		We define $\L$, a subgraph of $\Delta^c$ consisting of two components.
		The first component $\L_a$ is the union of the edges of $\Delta^c$ of the form $a_1a_i$, where $2\le i \le n$ and $a_1a_{0j}$ for $1\le j \le k$. 
		The second component $\L_b$ is the path in $\Delta^c$ visiting $b_1, b_2, \dots, b_n$. 
		(See the right side of Figure~\ref{fig:delta42} for an illustration of the case $n=4, k=2$).  

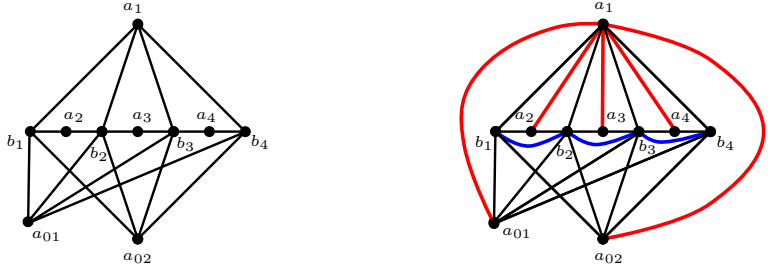
\begin{figure}[h!]
	\begin{overpic}[scale=1.8]
		{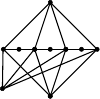}
			\put(-7, 43){\tiny $b_1$}
		\put(29.5, 38){\tiny $b_2$}
		\put(67.5 ,42.5){\tiny $b_3$}
		\put(100,43){\tiny $b_4$}
		\put(44, 103){\tiny $a_1$}
		\put(44,-7){\tiny $a_{02}$}
		\put(4, 3){\tiny $a_{01}$}
		\put(18, 55){\tiny $a_2$}
		\put(48,55){\tiny $a_3$}
		\put(76,55){\tiny $a_4$}
	
	\end{overpic}
	\hspace{1in}
	\begin{overpic}[scale=1.8]
		{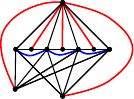}
		\put(5,32){\tiny $b_1$}
		\put(31.5, 29){\tiny $b_2$}
		\put(58.5,30.5){\tiny $b_3$}
		\put(84,32){\tiny $b_4$}
		\put(44, 77){\tiny $a_1$}
		\put(44,-5){\tiny $a_{02}$}
		\put(14, 3){\tiny $a_{01}$}
		\put(18, 41){\tiny $a_2$}
		\put(48,41){\tiny $a_3$}
		\put(69,41){\tiny $a_4$}
	\end{overpic}
	\caption{
	The figure shows $\Delta_{42}$ on the left, and 
	the two components of $\L$ for the graph $\Delta_{42}$ on the right. The component $\L_a$ is shown in red and the component $\L_b$ is shown in blue.}
	\label{fig:delta42}
\end{figure}

		Let $\Theta = \Theta(\Delta, \L)$ be as in the previous sections. We verify the properties $\mathcal R_1$--$\mathcal R_4$, $\mathcal F_1$ and $\mathcal F_2$. It will then follow from Theorem~\ref{thm_fi_visual_raags} that the subgroup generated by $E(\L)$ is an index four visual RAAG subgroup.

		The conditions $\mathcal R_1, \mathcal R_2$ and $\mathcal F_1$ are immediate.  
		Condition $\mathcal{F}_2$ holds by Remark~\ref{rmk_f2_simplification}.
		We now check $\mathcal R_3$.  Each square in $\Delta$ is of one of the following forms:
		\begin{enumerate}
			\item $b_ia_{i+1}b_{i+1}a_1$ for $1 \le i \le n-1$
			\item $b_ia_{i+1}b_{i+1}a_{0j}$ for $1 \le i \le n-1, 1 \le j \le k$
			\item $b_i a_1 b_{i'}a_{0j}$ for $1\le i<i' \le n, 1 \le j \le k$
			\item $b_i a_{0j} b_{i'}a_{0j'}$ for $1\le i<i' \le n, 1 \le j \le k$
		\end{enumerate}
		Condition $\mathcal R_3$ follows immediately for the first type of square, as the appropriate $\L$-convex hulls do not contain any additional vertices not included in the vertex set of the square.  For the second type of square, the convex hull in $\L_b$ does not contain any additional vertices, but the convex hull in $\L_a$ contains the additional vertex $a_1$, as this vertex lies on the $\L_a$-path between $a_{i+1}$ and $a_{0j}$.  Since $a_1$ is adjacent to $b_i$ and $b_{i+1}$, the condition $\mathcal R_3$ is verified for this type of 2-component square.  For the third
		type, the $\L$-convex hull of $\{a_1, a_{0j}\}$ does not contain any additional vertices of $\L$, and the $\L$-convex hull of $\{b_i, b_{i'}\}$ contains the additional vertices $b_{i+1}, \dots, b_{i'-1}$.  Since $a_1$ and $a_{0j}$ are adjacent to each of these, $\mathcal R_3$ is verified for this type of 2-component square as well. 
		Finally, for the last type,  the $\L$-convex hull of $\{a_{0j}, a_{0j'}\}$ contains the additional vertex $
		a_1$, and the $\L$-convex hull of $\{b_i, b_{i'}\}$ contains the additional vertices $b_{i+1}, \dots, b_{i'-1}$.  Once again, it is easily verified that $a_0, a_{1j}, a_{1j'}$ are each adjacent to each of 
		$b_i, \dots, b_{i'}$.  Thus $\mathcal R_3$ is verified. 
		
		Finally, we check $\mathcal R_4$.
		Let $\gamma$ be a 
		$\L_a \L_b$-cycle
		and let $e$ be an edge of $\gamma$.  First suppose $e$ is   of the form $a_i b_i$ for $2 \le i \le n$.
		In this  case, $\gamma$ necessarily passes through $b_{i-1}$ and 
		some $a$, where 
		either $a= a_1$ or $a= a_{0j}$,  for some $1 \le j \le k$.  Thus the $\L_a\L_b$-square $b_{i-1}a_ib_ia$ satisfies the criterion in $\mathcal R_4$, since it contains $e$, and the two vertices $a$ and $b_{i-1}$ are contained in the $\L$-convex hull of the vertices of $\gamma$. 
		The case where $e$ is of the form $a_i b_{i-1}$ for $2 \le i \le n$ is similar.  
		
		Now suppose 
		$e$ is of the form $a_1b_i$ for some  $1 \le i \le n$.  Then $\gamma$ necessarily passes through an edge of one of the following forms: $b_i a_{0j}$ for some $1 \le j \le k$, $b_i a_i$ or $b_ia_{i+1}$.  In the first of these cases, $\gamma$ necessarily also passes through a vertex $b_{i'}$ for some $i'\neq i$, and $a_1b_ia_{0j}b_{i'}$ is the desired square.  
		If the edge is of the form $b_i a_i$ (resp.~$b_ia_{i+1}$) then $\gamma$ must also pass through 
		$b_{i-1}$ (resp.~$b_{i+1}$), and the desired square is $a_1b_ia_ib_{i-1}$ (resp.~$a_1b_ia_{i+1}b_{i+1}$).  The case where $e$  is of the form $a_{0j}b_i$ for some  $1 \le i \le n$ and $1 \le j \le k$ is similar. 		This completes the verification of $\mathcal R_4$, and the corollary follows. 		
	\end{proof}

\begin{remark}
	We note that the RAAGs obtained in the above corollary do not 
	have a tree for defining graph when $k \ge 2$ and $n \ge 3$. 
	This is easy to check by computing the associated commuting graph. 
\end{remark}

\subsection{$2$-dimensional RACGs with planar defining graph}

In~\cite{nguyen-tran}, Nguyen--Tran characterize exactly which one-ended, 2-dimensional RACGs defined by 
planar 
non-join, CFS graphs are quasi-isometric to RAAGs.  In  this subsection, we use their work in conjunction with Theorem~\ref{thm_fi_visual_raags} to prove Theorem~\ref{intro_thm_planar} from the introduction.
Note that CFS is a graph-theoretic condition introduced in~\cite{Dani-Thomas} to characterize RACGs with at most quadratic divergence.  We omit the definition, as it is not needed here.  

\begin{remark}\label{rmk:cfs}
Any one-ended, $2$-dimensional RACG that is quasi-isometric to a RAAG must have CFS defining graph. This follows as one-ended RAAGs have either linear or quadratic divergence \cite{Behrstock-Charney}, and the defining graph of a $2$-dimensional RACG with linear or quadratic divergence is CFS \cite{Dani-Thomas}.
\end{remark}

Recall that a graph $\S$ is a \textit{suspension} if
$\S$ decomposes as a join $\S
=  \{a_1, a_2\} \star B$ where $a_1$ and $a_2$ are non-adjacent vertices.
We also say that $\S$ is the \textit{suspension} of the graph $B$.
We use the notation $\S_k(a, b)$ to denote the suspension graph $\{a_1, a_2\} \star \{b_1, \dots, 
b_{k} \}$, and we say that $a_1$ and $a_2$ are the \textit{suspension vertices}.

Let $\G$ be a graph which is connected, triangle-free, CFS
and planar. Suppose that a planar embedding from $\G$ into the sphere $\mathbb{S}^2$ is fixed.
In \cite{nguyen-tran}, Nguyen--Tran construct a tree $T$ (this is the \textit{visual
	decomposition tree} of Section 3 of that paper) associated to
$\Gamma$ with the following properties.  
The vertices of $T$ are in bijection with maximal suspension subgraphs of $\G$. 
As $\G$ is triangle-free, every maximal suspension
of $\G$ is of the form $\S_k(a, b)$, where both $\{a_1, a_2\}$ and $\{b_1,
\dots, b_k \}$ are each sets of disjoint
vertices of $\G$, and $k\ge 3$ if
$T$ contains at least two vertices.
Moreover, every vertex of $\G$ is contained in some suspension corresponding to
a vertex of $T$.
Two vertices of $T$ corresponding to suspensions $\S = \S_k(a,b)$ and $\S'
= \S_l(c,d)$ are connected
by an edge if $\S \cap \S'$ is a $4$-cycle $C$ which separates $\mathbb S^2$ into
two non-trivial components $B_1$ and $B_2$, such that $\S_1 \setminus C \subset
B_1$ and $\S_2 \setminus C \subset B_2$. 
Moreover, it must follow (by the
maximality of the suspensions) that $C
= \{a_1, c_1, a_2, c_2 \}$, i.e.~$C$ contains exactly the suspension vertices of $\S$ and $\S'$.

If $\G$ (with the above assumptions) is a join, then it readily follows that $\G$ is quasi-isometric to a RAAG whose defining graph is a tree of diameter at most $2$. 
Nguyen--Tran show that if $\G$ is not a join, then $W_\G$ is quasi-isometric to a RAAG if and only
if every vertex $v \in T$ has valence strictly less than $k$~\cite[Theorem~1.2]{nguyen-tran}, where $\Sigma_k(a,b)$ is the maximal suspension  in $\G$ corresponding to $v$. Moreover, they show that such RAAGs have defining graph a tree of diameter at least $3$.
Below, we prove such RACGs are in fact commensurable to RAAGs.

\begin{thm} \label{thm_planar}
	Let $W_\G$ be a $2$-dimensional, one-ended RACG with planar defining graph
	$\G$.
	Then $W_\G$ is quasi-isometric to a RAAG if and only if it 
	contains an index $4$ subgroup isomorphic to a RAAG.
\end{thm}

\begin{proof}
One direction of the theorem is obvious. 
Thus, we prove that if $W_\G$ 
satisfies these hypotheses and is quasi-isometric to a RAAG, then
$W_\G$ contains an index $4$ subgroup isomorphic to a RAAG. 
We do this by constructing a subgraph 
$\Lambda \subset \G^c$ 
with two components and satisfying the hypotheses of
Theorem \ref{thm_fi_visual_raags}.

Fix a planar embedding of $\G$ into the sphere $\mathbb{S}^2$.
Note that by Remark~\ref{rmk:cfs} and the hypotheses of the theorem, it follows that $\G$ is triangle-free, CFS and planar. 
Thus, there exists a visual decomposition tree $T$ associated to 
$\Gamma$ as described above.
Furthermore, as $W_\G$ is quasi-isometric to a RAAG, it follows from~\cite[Theorem~1.2]{nguyen-tran} 
that the valence of a vertex of $T$ corresponding to the maximal suspension $\S_k(a,b)$ is less than $k$.  

Henceforth, to simplify notation, the word suspension will always refer to a maximal suspension, and will consequently correspond to a vertex of $T$.
Given a suspension $\S_k(a,b) = \{a_1, a_2\} \star B$ we say that a labeling $\{b_1, \dots, b_k\}$ of the vertices of $B$ is  \textit{cyclic} if the following holds. If $C$ is a $4$-cycle spanning the vertices $\{a_1, b_i, a_2, b_{i+1}\}$ for some $1 \le i \le k$ or spanning the vertices $\{a_1, b_1, a_2, b_k\}$, then every vertex of $\S \setminus C$ is contained in a common component of $\mathbb{S}^2 \setminus C$.
Observe that if $E$ is a cycle corresponding to an edge of $T$ incident to the vertex of $T$ 
given by $\S_k(a, b)$, then the planarity of $\G$ implies that $E$ is one of the cycles $C$ mentioned in the previous sentence.

Let $N$ be the number of vertices of $T$.
Let $T_1 \subset \dots \subset T_N=T$ be a nested sequence of subtrees of $T$ such that
$T_1$ consists of a single vertex of $T$ and $T_i$ has exactly $i$ vertices.
Such choices are clearly possible.
For each $1 \le i \le n$, let
$\G_i$ be the subgraph of $\G$ spanned by every suspension
that corresponds to a vertex of $T_i$. 
Note that $\G_i \subset \G_{i+1}$ for all $1 \le i < N$ and that $\G_N = \G$. 
We define a nested sequence of graphs $\L_1 \subset \dots \subset
\L_N$ such that for each $1 \le i \le N$, $\L_i \subset \G_i^c$ and the following hold:
\begin{enumerate}
	\item Let $C$ be a $4$-cycle corresponding to an edge of $T$ that is incident to $T_i$. Then 
	each pair of non-adjacent vertices in $C$ is contained in a common edge of $\L_i$.
	\item The graph $\L_i$ contains exactly two components, and $\Theta_i = \Theta(\G_i, \L_i)$ satisfies conditions $\mathcal{R}_1
	- \mathcal{R}_4$,
	$\mathcal{F}_1$ and $\mathcal{F}_2$.
\end{enumerate}
The theorem clearly follows from this claim by using the graph 
$\Lambda= \Lambda_N \subset \Gamma^c$. 

We first define $\L_1 $ corresponding to the vertex $T_1 = \{v\}$.
Let $\S = \S_k(a,b) = \{a_1, a_2\} \star \{b_1, \dots, b_k\}$ be the suspension
corresponding to $v$, and assume that
$\{b_1, \dots, b_k\}$ is cyclic.
As the valence of $v$ in $T$ is less than $k$,
by possibly relabeling, we can assume that the $4$-cycle $\{a_1, b_1, a_2, b_k \}$ does not
correspond to an edge of $T$.
We define one component of $\L_1$ to be the edge $(a_1, a_2)$, and the other
component of $\L_1$ to consist of the edges $(b_1, b_2), (b_2, b_3), \dots,$
$(b_{k-1}, b_k)$.
By the observation above and our choice of labeling, any 4-cycle $C$ corresponding to an edge of $T$ incident to $v$ is of the form $\{a_1, b_i, a_2, b_{i+1} \}$ for some $1 \le i \le k-1$.  Thus condition (1) follows. Condition (2) is readily verified.

Suppose now that we have defined the graph $\L_{n-1}$ corresponding to the tree
$T_{n-1}$ satisfying conditions (1) and (2). 
We now define $\L_n$.

Let $u$ be the unique vertex in $T_n \setminus T_{n-1}$, and let $u'$ be the
unique vertex of $T_{n-1}$ that is adjacent to $u$.
Let $\S = \Sigma_k(a,b) = \{a_1, a_2\} \star \{b_1, \dots, b_k\}$ and $\S'
= \Sigma_l(c,d) = \{c_1, c_2 \} \star \{d_1, \dots, d_l \}$ be the suspension
graphs corresponding  to $u$ and $u'$ respectively. Furthermore, suppose these
labelings are cyclic.
It follows that $E = \{a_1, c_1, a_2, c_2\}$ is the $4$-cycle corresponding to
the edge in $T$ between $u$ and $u'$.
By possibly relabeling, we can assume that $c_1 = b_1$, $c_2 = b_k$, $a_1
= d_1$ and $a_2 = d_l$.
As $\L_{n-1}$ satisfies (1) above, $(a_1, a_2)$ and $(c_1, c_2)$ are edges of
$\L_{n-1}$.

As the valence of $u$ is less than $k$, there exist some $1 \le j <
k$ such that  the $4$-cycle $\{b_j, a_1, b_{j+1}, a_2\}$ 
does not correspond to an edge of $T$. 
We define $\L_n \subset \G_n^c$ to contain every edge of $\L_{n-1} \subset
\G_{n-1}^c \subset \G_n^c$ and additionally the edges: 
\[(b_1, b_2), (b_2, b_3), \dots, (b_{j-1}, b_{j}), (b_{j+1}, b_{j+2}), \dots,
(b_{k-1}, b_k) \]
This corresponds to adding one or two line segments each to a distinct vertex of $\L_{n-1}$. 
As $\L_{n-1}$ contains two components (by (2)) and does not contain any cycles (by $\mathcal{R}_1$), it follows that $\L_n$ contains two components and satisfies $\mathcal{R}_1$ as well.
Furthermore, (1) and condition $\mathcal{F}_1$ (for $\Theta_n$) follow from directly from our choices.
Condition $\mathcal{F}_2$ then follows from Remark \ref{rmk_f2_simplification}.

We now check $\mathcal{R}_2$.
Let $x, y \in \L_n$ be vertices contained in the same component of $\L_n$.
If $x$ and $y$ are both contained in $\L_{n-1}$, then the claim follows as $\L_{n-1}$ satisfies $\mathcal{R}_2$ and no new edges are added between vertices of $\G_{n-1}$ in forming $\G_n$. 
If $x$ and $y$ are both contained in $\S$, then by construction, they must lie in the same factor of the join $\S$ and there is no edge between them.
The only case left to check is that $x$ and $y$ lie in different components of $\mathbb{S}^2 \setminus E$. However, in this case there is no edge between $x$ and $y$ as $E$ separates $x$ from $y$ in the planar embedding.

We now check that $\mathcal{R}_3$ holds. Let $C$ be
a $2$-component square in $\L_{n}$. As $E$ separates every vertex of $\S
\setminus E$ from every vertex in $(\G_{n-1} \setminus E) \subset \G_n$, it
follows that either $C$ lies in $\G_{n-1} \subset \G_n$ or $C$ lies in
$\S$. In the first case the claim follows as $\Theta_{n-1}$ satisfies $\mathcal{R}_3$ (and noting that the convex hull of $C$ in $\L_n$ lies in $\L_{n-1}$). In the latter case, the claim is easily verified.

We now check $\mathcal{R}_4$.
Let $P$ be a $2$-component cycle in $\G_{n}$. 
If $P$ lies entirely in $\G_{n-1}$ then every edge of $P$ satisfies the condition in $\mathcal{R}_4$ as $\Theta_{n-1}$ satisfies $\mathcal{R}_4$. 
If $P$ lies entirely in $\S$, then $\mathcal{R}_4$ is easily verified.
Thus, we may assume that $P$ decomposes into two subpaths $P_1$ and $P_2$ such that $P_1 \subset \G_{n-1}$ and $P_2 \subset
\S \setminus E$.
As $P$ does not repeat vertices, it follows that $P_2$ consists of just two edges $(a_1, b_q)$ and $(a_2, b_q)$ for some $2 \le q \le k$.
As the valence of $u'$ is less than $l$, there exists some $1 \le q' < l$ and corresponding $4$-cycle $\{c_1, d_{q'}, c_2, d_{q'+1}\}$ such that every vertex of $\G_n$ is contained in a common component of $\mathbb{S}^2 \setminus C$. From this, we see that $a_1$ and $a_2$ are in different components of $\G_{n-1} \setminus \{c_1, c_2\}$. Thus, $P_1$ must either contain $c_1$ or $c_2$. 
Suppose that $P_1$ contains $c_1$ (the other case is similar). 
The path $P_1$ does not contain both  the edge $(a_1, c_1)$ and the edge $(a_2, c_1)$, for if it did, then $P$ would either the equal 4-cycle $\{a_1, c_1, a_2,b_q\}$ or contain it as a sub-cycle.  In the 
former case $P \subset \Sigma$, a case we have already ruled out, and in the latter case, $P$ 
necessarily repeats a vertex (which is not allowed).   
We now define a cycle $P'$ depending on which edges $P_1$ contains. We set $P'  =  (P_1 \setminus (a_1, c_1) ) 
\cup (a_2, c_1)$ if $(a_1, c_1) \subset P_1$ , $P' = (P_1  \setminus (a_2, c_1) )\cup (a_1, c_1)$ if $(a_2, c_1) \subset P_1$, and  $P' = P_1 \cup (a_1, c_1) \cup (a_2, c_1)$ if $P_1$ does not contain
either of  $(a_1, c_1)$ and $(a_2, c_1)$. In each case, it follows that $P'$ is a cycle in $\G_{n-1}$ containing every edge of $P_1$, except possibly $(a_1, c_1)$ and $(a_2, c_1)$.
Additionally, every vertex of $P'$ is a vertex of $P$,  so the $\L$-convex hull of $P'$ is contained in the $\L$-convex hull of $P$. 
From this and as $\G_{n-1}$ satisfies
$\mathcal{R}_4$, it follows that every edge of $P$ that is contained in $P_1$ satisfies $\mathcal{R}_4$ as well. 
Finally, every edge of $P \setminus P_1$ can been seen to satisfy $\mathcal{R}_4$ by using the $4$-cycle $\{a_1, b_q, a_2, c_1\}$.

We have thus checked that (2) holds. The theorem now follows.
\end{proof}

\section{Generalized reflection subgroups of RAAGs}
\label{sec:reflections} 

 Let $A_\G$ be a RAAG.
A \textit{generalized RAAG reflection} 
is a conjugate of an element of $V(\G)$, i.e.~$wsw^{-1}$ for some 
 $s \in V(\G) \cup V(\G)^{-1}$
 and $w$ a word in $A_\G$. Let $\mathcal{T}$
 be a set of reduced generalized RAAG reflections. We say that $\mathcal{T}$ is \textit{trimmed} if
  $\mathcal{T} \cap \mathcal{T}^{-1} = \emptyset$,
   and if  
 given any two distinct  generalized RAAG reflections $wsw^{-1}$ and $w's'w'^{-1}$ in $\mathcal{T}$, no expression for $w'$ has prefix $ws^{-1}$ 
 or prefix $ws$.
The following lemma follows from a straightforward adaptation of the proof of \cite[Lemma 10.1]{DL} to the setting of RAAGs.

\begin{lemma}\label{lemma_trimmed_reflections}
	Let $\mathcal{T}$ be a set of  generalized  RAAG reflections in the RAAG $A_\G$, and let $G$ be the subgroup generated by $\mathcal{T}$. Then $G$ is generated by a trimmed set of   generalized RAAG reflections which can be algorithmically obtained from $\mathcal{T}$. 
\end{lemma}

%The main result of this section is 
%the following:
In this section, we give a new proof of a result of Dyer:
\begin{thm}[\cite{Dyer}] \label{thm_reflection_raags}
	Let $\mathcal{T}$ be a finite set of generalized RAAG reflections in $A_\G$. Then the subgroup $G < A_\G$ generated by $\mathcal{T}$ is a RAAG. Moreover, if $\mathcal{T}$ is trimmed then $(G, \mathcal{T})$ is a RAAG system.
\end{thm}

We will use the characterization of RAAGs in Theorem~\ref{thm_raag_char}
to show that $G$ is a RAAG. We first prove a series of lemmas about disk diagrams of a special type, namely, ones  whose boundary labels are words over a trimmed set of generalized RAAG reflections.

The setup for these lemmas is as follows and will be fixed for the rest of this section. We fix  a trimmed set 
$\mathcal{T}$ of reduced generalized RAAG reflections in $A_\G$. Let $z = r_1 \dots r_n$
 be an expression for the identity element where  $r_i = w_i s_i w_i^{-1} \in \mathcal{T}$ for each $1 \le i \le n$. Let $D$ be a disk diagram whose boundary $\partial D$ is labeled by $z$. 
 For $1 \le i \le n$, let $p_{r_i}$ be the subpath of $\partial D$
 which is labeled by $r_i$. 
 Furthermore let $p_{w_i}$ and $p_{w_i^{-1}}$ denote the subpaths of $\partial D$ labeled $w_i$
 and $w_i^{-1}$ respectively, and let $e_i$ denote the edge labeled $s_i$. 
Let $H_i$ be the hyperplane dual to $e_i$, and let $\mathcal{H} = \{ H_i \}_{i=1}^n$ be the collection of all such hyperplanes. Note that as $r_i$ is a reduced word, no hyperplane is dual to two edges of $p_{r_i}$ for any $i$.

In all of the following lemmas, arithmetic is taken modulo $n$.

\begin{lemma} \label{lemma_is_raag1}  
	For each $1 \le i \le n$, the hyperplane $H_i$ does not intersect a hyperplane dual to $p_{w_i}$ or a hyperplane dual to $p_{w_i^{-1}}$ 
\end{lemma}
\begin{proof}
	Suppose $H_i$ intersects a hyperplane $K$ that is dual to an edge $f$ of $p_{w_i}$. Without loss of generality, we may assume that $f$ is the edge closest to $e_i$ out of all possible choices for $K$. As no hyperplane is dual to two edges of $p_{r_i}$, it follows that every hyperplane dual to an edge of $p_{w_i}$ which lies between $e_i$ and $f$ must intersect $K$. Thus, $w_i$ has suffix the word $t_1 \dots t_m$, where $t_1$ is the label of $K$ and $t_1$ commutes with
	$s_i$, as well as with $t_j$ for $2 \le j \le m$.
	This readily implies that $r_i$ is not reduced, for in $r_i = w_i s_i w_i^{-1}$, an occurrence of the RAAG generator $t_1$ in $w_i$ can be canceled with an occurrence of $t_1^{-1}$ is $w_i^{-1}$. However, this is a contradiction as $r_i$ is reduced by assumption. The argument for hyperplanes dual to $p_{w_i^{-1}}$ is analogous.
\end{proof}

\begin{lemma} \label{lemma_is_raag2}
	For each $1 \le i \le n$, 
	the hyperplane $H_i$ is not dual to $p_{w_{i+1}}$, $p_{w_{i+1}^{-1}}$, $p_{w_{i-1}^{-1}}$ or~$p_{w_{i-1}}$.
\end{lemma}
\begin{proof}
	For a contradiction, suppose $H_i$ is dual to an edge $f$ of $p_{w_{i+1}}$. By Lemma~\ref{lemma_is_raag1}, every hyperplane dual to an edge of $p_{w_{i}^{-1}}$ must also be dual to $p_{w_{i+1}}$. 
	Write $s_iw_i^{-1} = t_1 \dots t_m$ and $w_{i+1} = k_1 \dots k_l$
	where $t_j \in V(\Gamma)$ for $1 \le j \le m$ and $k_j \in V(\Gamma)$ for $1 \le j \le l$.
	The structure of the hyperplanes in $D$ implies that 
	 $w_{i+1}$ has an expression which begins with $t_m^{-1} \dots t_1^{-1} = w_is_i^{-1}$.
	This is a contradiction as $\mathcal{T}$ is trimmed.
	 A similar argument shows that $H_i$ is not dual to $p_{w_{i-1}^{-1}}$.
	
	Suppose now that $H_i$ is dual to $p_{w_{i-1}}$. By Lemma~\ref{lemma_is_raag1}, it follows that $H_{i-1}$ is dual to $p_{w_i}$. However, this is not possible by the same argument as above. Similarly, $H_i$ cannot be dual to $p_{w_{i+1}^{-1}}$.
\end{proof}

The proof of the following lemma is similar to that of the previous one.

\begin{lemma} \label{lemma_is_raag3}
	If $H_i = H_{i+1}$ for some $1 \le i \le n$ then $r_i \gequal  r_{i+1}^{-1}$.
	\hfill{\qed}
\end{lemma}

\begin{lemma} \label{lemma_is_raag4}
	If $H_i$ intersects $H_{i+1}$, then $r_i$ and $r_{i+1}$ commute. Furthermore, there is a disk diagram $D'$ with boundary label $r_1 \dots r_{i-1}r_{i+1}r_i r_{i+2} \dots r_n$, such that the natural bijection, from $e_i$, $e_{i+1}$
	 and the 
	 edges traversed by the subpath of the boundary path of $D$ 
	 labeled by $r_{i+2}\dots r_n r_1 \dots r_{i-1}$ to the 
	edges traversed by the corresponding 
	 subpaths 
	 of the boundary path of 
	 $D'$ with the same 
 labels, 
	 preserves boundary combinatorics.
\end{lemma}
\begin{proof}
	Suppose $H_i$ intersects $H_{i+1}$. By Lemma~\ref{lemma_is_raag1}, every hyperplane dual to $p_{w_i^{-1}}$ is either dual to $p_{w_{i+1}}$ or intersects $H_{i+1}$. Similarly, every hyperplane dual to $p_{w_{i+1}}$ is either dual to $p_{w_{i}^{-1}}$ or intersects $H_i$. It then readily follows that $w_i$ has a reduced expression $ba_1$  and $w_{i+1}$ has a reduced expression $ba_2$, where $a_1, a_2$ and $b$ are words,  
	such that the generators in the word $a_1s_i$ are all distinct from and commute with the generators in the word $a_2s_{i+1}$. 
	Consequently, $r_i$ commutes with $r_{i+1}$.
	
		We now construct the disk diagram $D'$.
	By Tits' solution to the word problem, the expression $ba_1$ 
 (resp.~$ba_2$) 
	can be obtained from $w_i$
 (resp.~$w_{i+1}$) 
	 by sequentially permuting adjacent letters.   Thus, by repeatedly applying Lemma~\ref{lem:new_diagrams}(1), we obtain a disk diagram with boundary label: 
	\[r_1 \dots r_{i-1}~ (ba_1s_ia_1^{-1}b^{-1} ) ~ ( ba_2s_{i+1}a_2^{-1}b ^{-1})~ r_{i+2} \dots r_n\]
	By repeatedly applying Lemma~\ref{lem:new_diagrams}(2), we can ``cancel'' $b^{-1}b$ and obtain a disk diagram with boundary label:
	\[r_1 \dots r_{i-1}~(ba_1s_ia_1^{-1}) ~ (a_2s_{i+1}a_2^{-1}b^{-1}) ~ r_{i+2} \dots r_n\]
	Then, by repeatedly applying Lemma~\ref{lem:new_diagrams}(1), we obtain a disk diagram with label:
	\[r_1 \dots r_{i-1}~(ba_2s_{i+1}a_2^{-1}) ~( a_1s_ia_1^{-1}b^{-1}) ~ r_{i+2} \dots r_n\]
	By Lemma~\ref{lem:new_diagrams}(3), we obtain a disk diagram with boundary label:
	\[r_1 \dots r_{i-1}~(ba_2s_{i+1}a_2^{-1}b^{-1}) ~ (ba_1s_ia_1^{-1}b^{-1}) ~ r_{i+2} \dots r_n\]
	Finally, by repeatedly applying Lemma~\ref{lem:new_diagrams}(1), we obtain a disk diagram $D'$ with boundary label:
	\[r_1 \dots r_{i-1}r_{i+1}r_i r_{i+2} \dots r_n\]
	Note that in each of these steps, the desired boundary combinatorics are preserved.
\end{proof}

\begin{lemma} \label{lemma_is_raag5}
	For every 
$1 \le i \le n$, 
	there exists some $j \neq i$ such that $H_i = H_j$.
\end{lemma}
\begin{proof}
Suppose we have a disk diagram with boundary label $z = r_1 \dots r_n$ such that, for some $1 \le i \le n$, the hyperplane
	$H_i$ is dual to an edge $f$ of $\partial D$ where $f \neq e_j$ for all $1 \le j \le n$. We call any disk diagram which has such an $H_i$ a \textit{pathological diagram} with \textit{pathology caused by $H_i$}. Given such a diagram, we define $p$ to be a path along $\partial D$ between $e_i$ and $f$, which does not include $e_i$ and $f$.
	 We also let $\mathcal{H}'$ denote the set of $H_j$ such that $e_j$ is contained in $p$.
	
	Given a pathological disk diagram $D$ we may choose a hyperplane $H_i$ causing the pathology together with a path $p$ so that the set $\mathcal{H}'$ is minimal among all possible choices of $H_i$ and $p$.
	After such a choice, we call $|\mathcal{H}'|$ the complexity of $D$. We will prove that pathological diagrams are not possible by induction on the complexity $c$ of such a diagram. The base case, when $c=0$, already follows from Lemma~\ref{lemma_is_raag2}.
	
	Now suppose we are given a pathological disk diagram $D$ with pathology caused by $H_i$ such that its complexity is $c = |\mathcal{H}'| > 0$, and suppose by induction there do not exist pathological disk diagrams of complexity smaller than $c$. 
	
	The edge 
	$f \neq e_i$ of $\partial D$ that is dual to $H_i$ 
	lies in a path $p_{r_{i'}}$ in $\partial D$ labeled by $w_{i'}s_{i'}w_{i'}^{-1}$ for some $1 \le i' \le n$ where $i \neq i'$. Let $Q$ denote the hyperplane $H_{i'}$.  Note that $Q$ may or may not be in $\mathcal H'$. 
	We prove our claim by considering two cases:
	\\\\
	\textit{Case 1:} Every hyperplane in $\mathcal{H}'$ intersects $H_i$. 
	
	We first observe that $\mathcal H'$ is non-empty (since the complexity of $D$ is positive) and does not consist of of $Q$ alone (by Lemma~\ref{lemma_is_raag2}). Therefore,
	we may choose $K \in \mathcal{H}' \setminus Q$ such that no hyperplane in $\mathcal{H}' \setminus Q$ intersects $H_i$ between $K \cap H_i$ and $H_i \cap e_i$. 
	Let $1 \le l \le n$ be such that $K$ is dual to $e_l \subset p_{r_l} \subset p$.
	Then for each $j$ with $i < j < l$, the hyperplane $H_j$ intersects $K= H_l$.  
	Thus, by repeatedly applying Lemma~\ref{lemma_is_raag4}, we can produce a new disk diagram with boundary label $r_1 \dots r_l r_i \dots r_{l-1} r_{l+1} \dots r_n$. 
	Furthermore, this new disk diagram is still pathological and has complexity smaller than $D$. However, this is not possible by our induction hypothesis.
	\\\\
	\textit{Case 2:} Some hyperplane $K \in \mathcal{H}'$ does not intersect $H_i$.
	
	We can choose such a hyperplane $K$ to be innermost, i.e. choose $K \in \mathcal{H}'$ such that $K$ does not intersect $H_i$ and  such that any hyperplane of $\mathcal H'$ dual to the subpath of $p$ between the edges dual to $K$ intersects $K$. 
	Since $H_i$ and $p$ were chosen to attain the complexity of $D$, it follows that $K$ does not cause a pathology, and is dual to distinct edges $e_l$ and $e_{l'}$  in $p$, where $1 \le l, l' \le n$.  By relabeling the $r_j$'s if necessary, we may assume that $l < l'$, and that the subpath of $\partial D$ from 
	$e_l$ to $e_{l'}$ is contained in $p$.  
By repeatedly applying Lemma~\ref{lemma_is_raag4}, 
	we can produce a new pathological disk diagram $D'$ with label 
		$r_1 \dots r_{l-1} r_{l+1} \dots r_{l'-1} r_l r_{l'} \dots r_n$
	and where some hyperplane, which we still denote by $K$, is dual to both the edge labeled by $e_l$ and the one labeled by $e_{l'}$. By Lemma~\ref{lemma_is_raag3}, 
	$r_l \gequal r_{l'}^{-1}$.  Furthermore, by repeatedly applying Lemma~\ref{lem:new_diagrams}(1) if necessary, we may assume that $r_l = r_{l'}^{-1}$ in the label of $\partial D'$.
	
	We now produce a new disk diagram $D''$ by identifying the consecutive paths in $\partial D'$ labeled by $r_l$ and $r_{l'}$,
	i.e.~we fold these two paths together. If $K \neq Q$, then we have produced a new pathological disk diagram with complexity $c - 2$, contradicting the induction hypothesis. On the other hand, if $K = Q$, note that the image of $H_i$ in $D''$ must intersect the path labeled by 
		$r_i \dots r_{l-1}r_{l+1} \dots r_{l' - 1}$
	in $\partial D''$.  
	Moreover we claim that it cannot be dual to an edge labeled by $e_j$ for $i < j \le l'-1$.  Suppose 
	it is dual to an edge labeled $e_j$.  It follows that the hyperplane $H_j$ in $D$ is dual to an edge 
	$f'$ in $p$, such that $f' \neq e_k$ for any $k$, and such that the images of $f$ and $f'$ are identified in $D''$.  This is a contradiction, as it implies that $H_j$ causes a pathology of lower complexity than $H_i$.  Thus, the image of $H_i$ in $D''$ causes a pathology of complexity at most $c-2$, which is again a contradiction. 
\end{proof}
We are now ready to prove Theorem~\ref{thm_reflection_raags}.
\begin{proof}[Proof of Theorem~\ref{thm_reflection_raags}]
	As 
	$G$ can be generated by a trimmed set of generalized RAAG reflections (by Lemma~\ref{lemma_trimmed_reflections}), we assume without loss of generality that $\mathcal{T}$ is trimmed. We will show that $(G, \mathcal{T})$ is a RAAG system by applying Theorem~\ref{thm_raag_char}. 
		Note that $\mathcal{T} \cap \mathcal{T}^{-1} = \emptyset$ as $\mathcal{T}$ is trimmed. 
	We check each condition of that theorem, by proving the corresponding two claims:
	\\\\
	\textit{i) Every $r \in \mathcal{T}$ has infinite order.}
	
	By definition, $r$ is equal to a reduced word $wsw^{-1}$ 
	with 
	$s \in V(\G) \cup V(\G)^{-1}$ 
	and $w$ a word in $W_\Gamma$.  It follows that $ws^nw^{-1}$ is an expression for $r^n$. Moreover, as $r$ is reduced, it readily follows from Theorem~\ref{thm_tits_solution}
	 that $ws^nw^{-1}$ is reduced as well. Hence, $r$ has infinite order.
	\\\\
	\textit{ii) Given any word $w = a_1 \dots a_m$, with $a_i  \in \mathcal{T}$, either $w$ is reduced over $\mathcal{T}$ or there is an expression for $w$ of the form $a_1 \dots \hat{a}_i \dots \hat{a}_j \dots a_m$.}
	
	Suppose $w = a_1 \dots a_m$ is not reduced over $\mathcal{T}$. Let $w' = b_1 \dots b_k$, with $b_i \in \mathcal{T}$ 
 and $k < m$, 
	 be an expression for $w$ which is reduced over $\mathcal{T}$. Form a disk diagram $D$ with boundary label~$ww'^{-1}$. 
	
	We relabel the generalized reflections in the word $ww'^{-1}$ by setting $r_i = a_i$ for $1 \le i \le m$, and $r_{m + i} = b_{k-i+1}^{-1}$ (the $i$th generalized RAAG reflection in $w'^{-1}$)
	 for $1 \le i \le k$.
	By Lemma~\ref{lemma_is_raag5}, every $H \in \mathcal{H}$ is only dual to edges of $\partial D$ labeled by $s_i$ for some $i$, where $r_i = w_i s_i w_i^{-1}$. 
	As $m > k$, there exists some hyperplane $H \in \mathcal{H}$ that is dual to two edges of the subpath $p$ of $\partial D$ labeled by $w$. Furthermore, we may choose an innermost such $H \in \mathcal{H}$,
	in the sense that every hyperplane in $\mathcal{H} \setminus H$ intersects $p$ at most once.

	Let $e_l$ and $e_{l'}$ be the edges dual to $H$ where $l < l' \le m$.
	By repeatedly applying 
		Lemma~\ref{lemma_is_raag4},
	 we produce a disk diagram whose boundary 
	label is $$r_1 \dots \hat{r}_l \dots r_{l'-1} r_l r_{l'} \dots r_n,$$ such that a hyperplane of $\mathcal{H}$ is still dual to  the images of the edges  $e_l$  and $e_{l'}$ under the natural map between the boundaries of the disk diagrams.
	By Lemma~\ref{lemma_is_raag3}, 
 $r_l = r_{l'}^{-1}$.
	Thus, $r_1 \dots \hat{r}_l \dots \hat{r}_{l'} \dots r_n$ is an expression for $ww'^{-1}$. 
	Consequently, $r_1 \dots \hat{r}_l \dots \hat{r}_{l'} \dots r_m = a_1 \dots \hat{a}_l \dots \hat{a}_{l'} \dots a_m$ is an expression for $w$.
\end{proof}

\bibliographystyle{amsalpha}
\bibliography{refs}

\providecommand{\bysame}{\leavevmode\hbox to3em{\hrulefill}\thinspace}
\providecommand{\MR}{\relax\ifhmode\unskip\space\fi MR }
\providecommand{\MRhref}[2]{%
  \href{http://www.ams.org/mathscinet-getitem?mr=#1}{#2}
}
\providecommand{\href}[2]{#2}
\begin{thebibliography}{BKS08}

\bibitem[Bah05]{bahls}
Patrick Bahls, \emph{The isomorphism problem in {C}oxeter groups}, Imperial
  College Press, London, 2005.

\bibitem[Bas02]{basarab}
\c{S}erban~A. Basarab, \emph{Partially commutative {A}rtin-{C}oxeter groups and
  their arboreal structure}, J. Pure Appl. Algebra \textbf{176} (2002), no.~1,
  1--25.

\bibitem[BB05]{BB}
Anders Bj\"{o}rner and Francesco Brenti, \emph{Combinatorics of {C}oxeter
  groups}, Graduate Texts in Mathematics, vol. 231, Springer, New York, 2005.

\bibitem[BC12]{Behrstock-Charney}
Jason Behrstock and Ruth Charney, \emph{Divergence and quasimorphisms of
  right-angled {A}rtin groups}, Math. Ann. \textbf{352} (2012), no.~2,
  339--356.

\bibitem[Beh19]{Behrstock}
Jason Behrstock, \emph{A counterexample to questions about boundaries,
  stability, and commensurability}, Beyond hyperbolicity, London Math. Soc.
  Lecture Note Ser., vol. 454, Cambridge Univ. Press, Cambridge, 2019,
  pp.~151--159.

\bibitem[BKS08]{BKS}
Mladen Bestvina, Bruce Kleiner, and Michah Sageev, \emph{The asymptotic
  geometry of right-angled {A}rtin groups. {I}}, Geom. Topol. \textbf{12}
  (2008), no.~3, 1653--1699.

\bibitem[CH17]{cordes-hume}
Matthew Cordes and David Hume, \emph{Stability and the {M}orse boundary}, J.
  Lond. Math. Soc. (2) \textbf{95} (2017), no.~3, 963--988.

\bibitem[Cha07]{charney}
Ruth Charney, \emph{An introduction to right-angled {A}rtin groups}, Geom.
  Dedicata \textbf{125} (2007), 141--158.

\bibitem[CS15]{charney-sultan}
Ruth Charney and Harold Sultan, \emph{Contracting boundaries of {$\rm CAT(0)$}
  spaces}, J. Topol. \textbf{8} (2015), no.~1, 93--117.

\bibitem[Dav15]{Davis}
Michael~W. Davis, \emph{The geometry and topology of {C}oxeter groups},
  Introduction to modern mathematics, Adv. Lect. Math. (ALM), vol.~33, Int.
  Press, Somerville, MA, 2015, pp.~129--142.

\bibitem[Deo89]{Deodhar}
Vinay~V. Deodhar, \emph{A note on subgroups generated by reflections in
  {C}oxeter groups}, Arch. Math. (Basel) \textbf{53} (1989), no.~6, 543--546.

\bibitem[DJ00]{DJ}
Michael~W. Davis and Tadeusz Januszkiewicz, \emph{Right-angled {A}rtin groups
  are commensurable with right-angled {C}oxeter groups}, J. Pure Appl. Algebra
  \textbf{153} (2000), no.~3, 229--235.

\bibitem[DL21]{DL}
Pallavi Dani and Ivan Levcovitz, \emph{Subgroups of right-angled {C}oxeter
  groups via {S}tallings-like techniques}, J. Comb. Algebra \textbf{5} (2021),
  no.~3, 237--295. \MR{4333953}

\bibitem[DT15]{Dani-Thomas}
Pallavi Dani and Anne Thomas, \emph{Divergence in right-angled {C}oxeter
  groups}, Trans. Amer. Math. Soc. \textbf{367} (2015), no.~5, 3549--3577.

\bibitem[Dye90]{Dyer}
Matthew Dyer, \emph{Reflection subgroups of {C}oxeter systems}, J. Algebra
  \textbf{135} (1990), no.~1, 57--73.

\bibitem[Gen17]{Genevois1}
Anthony Genevois, \emph{Cubical-like geometry of quasi-median graphs and
  applications to geometric group theory}, arXiv:1712.01618, 2017.

\bibitem[Gen19]{Genevois2}
\bysame, \emph{Embeddings into {T}hompson's groups from quasi-median geometry},
  Groups Geom. Dyn. \textbf{13} (2019), no.~4, 1457--1510. \MR{4033512}

\bibitem[Gre90]{Green}
Elisabeth~Ruth Green, \emph{Graph products of groups}, 1990, Thesis (Ph.D.) --
  The University of Leeds.

\bibitem[HW10]{haglund-wise}
Fr\'{e}d\'{e}ric Haglund and Daniel~T. Wise, \emph{Coxeter groups are virtually
  special}, Adv. Math. \textbf{224} (2010), no.~5, 1890--1903. \MR{2646113}

\bibitem[KK13]{Kim-Koberda}
Sang-hyun Kim and Thomas Koberda, \emph{Embedability between right-angled
  {A}rtin groups}, Geom. Topol. \textbf{17} (2013), no.~1, 493--530.

\bibitem[LaF17]{laforge}
Garret LaForge, \emph{Visible {A}rtin {S}ubgroups of {R}ight-{A}ngled {C}oxeter
  {G}roups}, ProQuest LLC, Ann Arbor, MI, 2017, Thesis (Ph.D.)--Tufts
  University.

\bibitem[NT17]{nguyen-tran}
Hoang~Thanh Nguyen and Hung~Cong Tran, \emph{On the coarse geometry of certain
  right-angled {C}oxeter groups}, arXiv:1712.01079, to appear in {A}lgebraic \&
  {G}eometric {T}opology, 2017.

\bibitem[Sag95]{Sageev}
Michah Sageev, \emph{Ends of group pairs and non-positively curved cube
  complexes}, Proc. London Math. Soc. (3) \textbf{71} (1995), no.~3, 585--617.

\bibitem[Tit69]{Tits}
Jacques Tits, \emph{Le probl\`eme des mots dans les groupes de {C}oxeter},
  Symposia {M}athematica ({INDAM}, {R}ome, 1967/68), {V}ol. 1, Academic Press,
  London, 1969, pp.~175--185.

\bibitem[Wis21]{wise-qc-hierarchy}
Daniel~T. Wise, \emph{The structure of groups with a quasiconvex hierarchy},
  Annals of Mathematics Studies, vol. 209, Princeton University Press,
  Princeton, NJ, [2021] \copyright 2021. \MR{4298722}

\end{thebibliography}

\end{document}